\documentclass[12pt]{amsart}
\usepackage{amsmath,amsfonts,amssymb,amsthm}
\usepackage[abbrev,nobysame]{amsrefs}
\usepackage{mathrsfs}
\usepackage[all,cmtip]{xy}
\usepackage{tikz-cd}
 \usepackage{relsize}
\usepackage{hyperref}
\usepackage{accents}
\usepackage{xcolor}
\usepackage{soul}\usepackage{calligra}
\usepackage{enumitem}
\usepackage{bbm}

\setlist[itemize]{align=parleft,left=0pt..1.5em}

 \allowdisplaybreaks

\DeclareMathAlphabet{\mathcalligra}{T1}{calligra}{m}{n}

\DeclareMathOperator{\Res}{Res}

\DeclareMathOperator{\Hom}{Hom}

\DeclareMathOperator{\Sing}{Sing}
\DeclareMathOperator{\Rat}{Rat}
\DeclareMathOperator{\Span}{Span}

\DeclareMathOperator{\End}{End}

\DeclareMathOperator{\diag}{diag}
\DeclareMathOperator{\Com}{Com}
\DeclareMathOperator{\obj}{obj}

\begin{document}

%--------------<Theorem Style Head>--------------
\newtheorem{thm}{Theorem}[section]
\newtheorem{prop}[thm]{Proposition}
\newtheorem{coro}[thm]{Corollary}
\newtheorem{conj}[thm]{Conjecture}
\newtheorem{example}[thm]{Example}
\newtheorem{lem}[thm]{Lemma}
\newtheorem{hy}[thm]{Hypothesis}
\newtheorem*{acks}{Acknowledgements}
\theoremstyle{definition}
\newtheorem{rem}[thm]{Remark}
\newtheorem{de}[thm]{Definition}
\newtheorem{ex}[thm]{Example}

\newtheorem{convention}[thm]{Convention}

\newtheorem{bfproof}[thm]{{\bf Proof}}
%\xymatrixcolsep{5pc}
%--------------<\Theorem Style Head>-------------

%--------------<Common Sets>---------------------
\newcommand{\C}{{\mathbb{C}}}
\newcommand{\Z}{{\mathbb{Z}}}
\newcommand{\N}{{\mathbb{N}}}
\newcommand{\Q}{{\mathbb{Q}}}
\newcommand{\te}[1]{\mbox{#1}}
\newcommand{\set}[2]{{
    \left.\left\{
        {#1}
    \,\right|\,
        {#2}
    \right\}
}}
\newcommand{\sett}[2]{{
    \left\{
        {#1}
    \,\left|\,
        {#2}
    \right\}\right.
}}

\newcommand{\choice}[2]{{
\left[
\begin{array}{c}
{#1}\\{#2}
\end{array}
\right]
}}
\def \<{{\langle}}
\def \>{{\rangle}}

\def\({\left(}

\def\){\right)}

\def \:{\mathopen{\overset{\circ}{
    \mathsmaller{\mathsmaller{\circ}}}
    }}
\def \;{\mathclose{\overset{\circ}{\mathsmaller{\mathsmaller{\circ}}}}}

\newcommand{\overit}[2]{{
    \mathop{{#1}}\limits^{{#2}}
}}
\newcommand{\belowit}[2]{{
    \mathop{{#1}}\limits_{{#2}}
}}

\newcommand{\wt}[1]{\widetilde{#1}}

\newcommand{\wh}[1]{\widehat{#1}}

\newcommand{\wck}[1]{\reallywidecheck{#1}}

\newlength{\dhatheight}
\newcommand{\dwidehat}[1]{%
    \settoheight{\dhatheight}{\ensuremath{\widehat{#1}}}%
    \addtolength{\dhatheight}{-0.45ex}%
    \widehat{\vphantom{\rule{1pt}{\dhatheight}}%
    \smash{\widehat{#1}}}}
\newcommand{\dhat}[1]{%
    \settoheight{\dhatheight}{\ensuremath{\hat{#1}}}%
    \addtolength{\dhatheight}{-0.35ex}%
    \hat{\vphantom{\rule{1pt}{\dhatheight}}%
    \smash{\hat{#1}}}}

\newcommand{\dwh}[1]{\dwidehat{#1}}

\newcommand{\dis}{\displaystyle}

\newcommand{\pd}[1]{\frac{\partial}{\partial {#1}}}

\newcommand{\pdiff}[2]{\frac{\partial^{#2}}{\partial #1^{#2}}}

%--------------<\Common Sets>--------------------

%--------------<Global>--------------------------
\newcommand{\g}{{\mathfrak g}}
\newcommand{\ff}{{\mathfrak f}}
\newcommand{\f}{\ff}
\newcommand{\gc}{{\bar{\g'}}}
\newcommand{\h}{{\mathfrak h}}
\newcommand{\cent}{{\mathfrak c}}
\newcommand{\notc}{{\not c}}
\newcommand{\Loop}{{\mathcal L}}
\newcommand{\G}{{\mathcal G}}
\newcommand{\D}{\mathcal D}
\newcommand{\T}{\mathcal T}
\newcommand{\Free}{\mathcal F}
\newcommand{\Cfk}{\mathcal C}
\newcommand{\nil}{\mathfrak n}
\newcommand{\al}{\alpha}
\newcommand{\be}{\beta}
\newcommand{\beck}{\be^\vee}
\newcommand{\ssl}{{\mathfrak{sl}}}
\newcommand{\id}{\te{id}}
\newcommand{\rtu}{{\xi}}
\newcommand{\period}{{N}}
\newcommand{\half}{{\frac{1}{2}}}
\newcommand{\reciprocal}[1]{{\frac{1}{#1}}}
\newcommand{\inverse}{^{-1}}
\newcommand{\inv}{\inverse}
\newcommand{\SumInZm}[2]{\sum\limits_{{#1}\in\Z_{#2}}}
\newcommand{\uce}{{\mathfrak{uce}}}
\newcommand{\Rcat}{\mathcal R}
\newcommand{\cS}{{\mathcal{S}}}

%--------------<\Global>-------------------------

%--------------<Local>---------------------------
\newcommand{\E}{{\mathcal{E}}}
\newcommand{\F}{{\mathcal{F}}}

\newcommand{\Etopo}{{\mathcal{E}_{\te{topo}}}}

\newcommand{\Ye}{{\mathcal{Y}_\E}}

\newcommand{\rh}{{{\bf h}}}
\newcommand{\rp}{{{\bf p}}}
\newcommand{\rrho}{{{\pmb \varrho}}}
\newcommand{\ral}{{{\pmb \al}}}

%% ----------------<functors>--------------------
\newcommand{\comp}{{\mathfrak{comp}}}
\newcommand{\ctimes}{{\widehat{\boxtimes}}}
\newcommand{\ptimes}{{\widehat{\otimes}}}
\newcommand{\ptimeslt}{{
%   \leftidx{ _{\te{tri}}}{\ctimes}{}
{}_{\te{t}}\ptimes
}}
\newcommand{\ptimesrt}{{\ot_{\te{t}} }}
\newcommand{\ttp}[1]{{
    {}_{{#1}}\ptimes
}}
\newcommand{\bigptimes}{{\widehat{\bigotimes}}}
\newcommand{\bigptimeslt}{{
%   \leftidx{ _{\te{tri}}}{\ctimes}{}
{}_{\te{t}}\bigptimes
}}
\newcommand{\bigptimesrt}{{\bigptimes_{\te{t}} }}
\newcommand{\bigttp}[1]{{
    {}_{{#1}}\bigptimes
}}

\newcommand{\ot}{\otimes}
\newcommand{\Ot}{\bigotimes}
\newcommand{\bt}{\boxtimes}

\newcommand{\affva}[1]{V_{\wh\g}\(#1,0\)}
\newcommand{\saffva}[1]{L_{\wh\g}\(#1,0\)}
\newcommand{\saffmod}[1]{L_{\wh\g}\(#1\)}

\newcommand{\otcopies}[2]{\belowit{\underbrace{{#1}\ot \cdots \ot {#1}}}{{#2}\te{-times}}}

\newcommand{\wtotcopies}[3]{\belowit{\underbrace{{#1}\wh\ot_{#2} \cdots \wh\ot_{#2} {#1}}}{{#3}\te{-times}}}

%% ----------------<\functors>-------------------

%% ----------------<algebras>--------------------
\newcommand{\tar}{{\mathcal{DY}}_0\(\mathfrak{gl}_{\ell+1}\)}
\newcommand{\U}{{\mathcal{U}}}
\newcommand{\htar}{\mathcal{DY}_\hbar\(A\)}
\newcommand{\hhtar}{\widetilde{\mathcal{DY}}_\hbar\(A\)}
\newcommand{\htarz}{\mathcal{DY}_0\(\mathfrak{gl}_{\ell+1}\)}
\newcommand{\hhtarz}{\widetilde{\mathcal{DY}}_0\(A\)}
\newcommand{\qhei}{\U_\hbar\left(\hat{\h}\right)}
\newcommand{\n}{{\mathfrak{n}}}
\newcommand{\vac}{{{\mathbbm 1}}}
\newcommand{\vtar}{{{
    \mathcal{V}_{\hbar,\tau}\left(\ell,0\right)
}}}

\newcommand{\qtar}{
    \U_q\(\wh\g_\mu\)}
\newcommand{\rk}{{\bf k}}
% ----------------<\algebras>-------------------

\newcommand{\hctvs}[1]{Hausdorff complete linear topological vector space}
\newcommand{\hcta}[1]{Hausdorff complete linear topological algebra}
\newcommand{\ons}[1]{open neighborhood system}
\newcommand{\B}{\mathcal{B}}
\newcommand{\rx}{{\bf x}}
\newcommand{\re}{{\bf e}}
\newcommand{\rphi}{{\boldsymbol{ \phi}}}

\newcommand{\der}{\mathcal D}

\newcommand{\prodlim}{\mathop{\prod_{\longrightarrow}}\limits}

\newcommand{\lp}[1]{\mathcal L(f)}

\newcommand{\cha}{\check a}
\newcommand{\chh}{\check \h}

%--------------<\Local>--------------------------

\makeatletter
\@addtoreset{equation}{section}
\def\theequation{\thesection.\arabic{equation}}
\makeatother \makeatletter

\title{Quantization of parafermion vertex algebras}

\author{Fei Kong}
\email{kongmath@hunnu.edu.cn}
\address{Key Laboratory of Computing and Stochastic Mathematics (Ministry of Education), School of Mathematics and Statistics, Hunan Normal University, Changsha, China 410081}

%\begin{titlepage}
%\title{Quantization of parafermion vertex algebras}
%\author{Fei Kong}
%\newcommand\institute{%
%Key Laboratory of Computing and Stochastic Mathematics (Ministry of Education), School of Mathematics and Statistics, Hunan Normal University, Changsha, China 410081}
%\makeatletter
%\centering
%{\Huge\bfseries\sffamily\@title} \bigskip\par
%{\Large\bfseries\@author} \bigskip\par
%\email{kongmath@hunnu.edu.cn}
%\makeatother
%\vfill
%\large\institute
%\end{titlepage}

%------
% Insert an abstract.
%------
\begin{abstract}
Let $\mathfrak g$ be a finite dimensional simple Lie algebra over $\mathbb C$, and let $\ell$ be a positive integer.
In this paper, we construct the quantization $K_{\hat{\mathfrak g},\hbar}^\ell$ of the parafermion vertex algebra $K_{\hat{\mathfrak g}}^\ell$ as an $\hbar$-adic quantum vertex subalgebra inside the simple quantum affine vertex algebra $L_{\hat{\mathfrak g},\hbar}^\ell$.
We show that $L_{\hat{\mathfrak g},\hbar}^\ell$ contains an $\hbar$-adic quantum vertex subalgebra isomorphic to the quantum lattice vertex algebra $V_{\sqrt\ell Q_L}^{\eta_\ell}$, where $Q_L$ is the lattice generated by the long roots of $\mathfrak g$.
Moreover, we prove the double commutant property of $K_{\hat{\mathfrak g},\hbar}^\ell$ and $V_{\sqrt\ell Q_L}^{\eta_\ell}$ in $L_{\hat{\mathfrak g},\hbar}^\ell$.
\end{abstract}

%------
% Optional: Dedication.
%------
%\dedication{This memoir is dedicated to XXX.}

%------
% Insert a list of keywords.
% -- Separate keywords with comma.
% -- Capitalize only the first keyword in the list.
% -- No final full stop.
%------
\keywords{Quantum vertex algebras, parafermion vertex algebras, quantum parafermion vertex algebras}

%------
% Insert MSC 2020 codes according to www.ams.org/msc/msc2020.html.
% -- There must be exactly *one* primary code.
% -- The number of secondary codes is not specified.
% -- No final full stop.
%------
\subjclass[2020]{17B69}
%------
% Insert acknowledgments.
%------
%\begin{ack}
%Part of this paper was finished during my visit at Xiamen University and Tianyuan Mathematical Center in Southeast China, in
%August 2023. I am very grateful to Professor Shaobin Tan, Fulin Chen, Qing Wang for their hospitality.
%\end{ack}
%
%%------
%% Insert information regarding funding.
%%------
%\begin{funding}
%NSF of China (No.12371027).
%\end{funding}
\maketitle

%\tableofcontents

%------
% Insert the body of the book here.
%------
\section{Introduction}

Let $\g$ be a finite dimensional simple Lie algebra over $\C$,
and let $L_{\hat\g}^\ell$ be the simple affine vertex operator algebra of positive integer level $\ell$.
The parafermion vertex operator algebra $K_{\hat\g}^\ell$ is the commutant of the Heisenberg vertex operator algebra in $L_{\hat\g}^\ell$,
and can also be regarded as the commutant of the lattice vertex operator algebra $V_{\sqrt\ell Q_L}$ in $L_{\hat\g}^\ell$ (\cite{DW-para-structure-double-comm})
where $Q_L$ is the lattice spanned by the long roots of $\g$.
A set of generators of the parafermion vertex operator algebra $K_{\hat\g}^\ell$ was determined in
\cite{DLY-para-gen-1,DLW-para-gen-2,DW-para-structure-gen,DR-para-rational}.
While the $C_2$-cofiniteness of $K_{\hat\g}^\ell$ was proved in \cite{DW-para-cofiniteness-1,ALY-para-cofiniteness-2},
and the rationality was proved in \cite{DR-para-rational} with the help of a result in \cite{CM-orbifold}
on the abelian orbifolds for rational and $C_2$-cofinite vertex operator algebras.
The irreducible modules of $K_{\hat\ssl_2}^\ell$ were classified in \cite{ALY-para-cofiniteness-2}
and their fusion rules and quantum dimensions (\cite{DJX-qdim-qGalois}) were computed in \cite{DW-para-irr-mods-fusion-ssl2}.
For general $\g$, the irreducible modules of $K_{\hat\g}^\ell$ were classified in \cite{DR-para-rational, ADJR-para-irr-mods-fusion},
the fusion rules were determined in \cite{ADJR-para-irr-mods-fusion} with the help of quantum dimensions,
and their trace functions were computed in \cite{DKR-trace-para}.
Parafermion vertex operator algebras also have a close relationship with $W$-algebras (\cite{DLY-para-gen-1, ALY-para-W-algebra}).

It is well known that vertex algebras are closely related to affine Kac-Moody Lie algebras (\cite{FZ,Li-local,LL,DL,MP1,MP2,DLM}).
In \cite{Dr-hopf-alg} and \cite{JimboM}, Drinfeld and Jimbo independently introduced the notion of quantum enveloping algebras,
which are formal deformations of the universal enveloping algebras of Kac-Moody Lie algebras.
The quantum enveloping algebras of affine types, which are called quantum affine algebras, are one of the most important subclasses.
Like the affinization realization of affine Kac-Moody Lie algebras, Drinfeld provided a quantum affinization realization of quantum affine algebras in \cite{Dr-new}.
Based on the Drinfeld presentation, Frenkel and Jing constructed vertex representations for simply-laced untwisted quantum affine algebras in
\cite{FJ-vr-qaffine} and formulated a fundamental problem of developing a certain ``quantum vertex algebra theory'' associated to quantum affine algebras,
in parallel with the connection between vertex algebras and affine Kac-Moody Lie algebras.

In \cite{EK-qva}, Etingof and Kazhdan developed a theory of quantum vertex operator algebras in the sense of formal deformations of vertex algebras.
Partly motivated by the work of Etingof and Kazhdan, H. Li
conducted a series of studies.
While vertex algebras are analogues of commutative associative algebras,
Kac introduced the notion of the field algebras \cite{Kac-VA} (see also \cite{bk}), which are analogues of noncommutative associative algebras.
These algebras were also called weak axiomatic $G_1$-vertex algebras in \cite{li-g1}, or nonlocal vertex algebras
in \cite{Li-nonlocal}.
A weak quantum vertex algebra \cite{Li-nonlocal} is a nonlocal vertex algebra satisfying the $S$-locality. In addition, it becomes a quantum vertex algebra \cite{Li-nonlocal} if the $S$-locality is controlled by a rational quantum Yang-Baxter operator.
The $\hbar$-adic counterparts of these notions were introduced in \cite{Li-h-adic}.
In this framework, a quantum vertex operator algebra in sense of Etingof-Kazhdan is an $\hbar$-adic quantum vertex algebra whose classical limit is a vertex algebra.

In the very paper \cite{EK-qva}, Etingof and Kazhdan constructed quantum vertex operator algebras as formal deformations of
the universal affine vertex algebras
$V_{\hat{\mathfrak gl}_n}^\ell$ and $V_{\hat\ssl_n}^\ell$, by using the $R$-matrix type relations given in \cite{RS-RTT}.
Butorac, Jing and Ko\v{z}i\'{c} (\cite{BJK-qva-BCD}) extended Etingof-Kazhdan's construction
to type $B$, $C$ and $D$ rational $R$-matrices.
The quantum vertex operator algebras associated with trigonometric $R$-matrices of type $A$, $B$, $C$ and $D$ were constructed in \cite{Kozic-qva-tri-A, K-qva-phi-mod-BCD}.
In \cite{JKLT-Defom-va}, we developed a method to construct quantum vertex operator algebras by using vertex bialgebras.
By using this method, we constructed the quantum lattice vertex algebras, which is a family of quantum vertex operator algebras as deformations of lattice vertex algebras.
Moreover, based on the Drinfeld's quantum affinization construction (\cite{Dr-new,J-KM,Naka-quiver}), in \cite{K-Quantum-aff-va} we constructed the quantum affine vertex algebras $V_{\hat\g,\hbar}^\ell$ and $L_{\hat\g,\hbar}^\ell$ for all symmetric Kac-Moody Lie algebras $\g$.
When $\g$ is of finite type, we proved that $L_{\hat\g,\hbar}^\ell/\hbar L_{\hat\g,\hbar}^\ell\cong L_{\hat\g}^\ell$.

%In this paper, we study the quantization of parafermion vertex algebras, that is, the commutant $K_{\hat\g,\hbar}^\ell$ of quantum Heisenberg vertex algebra in $L_{\hat\g,\hbar}^\ell$.
%We give a set of generators of $K_{\hat\g,\hbar}^\ell$ in Section \ref{sec:qpara}.
%By using these generators, we prove that $K_{\hat\g,\hbar}^\ell/\hbar K_{\hat\g,\hbar}^\ell\cong K_{\hat\g}^\ell$.
%Furthermore, we show that $L_{\hat\g,\hbar}^\ell$ contains a quantum lattice vertex algebra $V_{\sqrt\ell Q_L}^{\eta_\ell}$ as a subalgebra,
%and prove that $K_{\hat\g}^\ell$ and $V_{\sqrt\ell Q_L}^{\eta_\ell}$ are commutants of each other in $L_{\hat\g,\hbar}^\ell$.
%
%In this paper, we study a quantization of parafermion vertex algebras, that is, the commutant $K_{\hat\g,\hbar}^\ell$ of quantum Heisenberg vertex algebra in $L_{\hat\g,\hbar}^\ell$.
%We give a set of generators of $K_{\hat\g,\hbar}^\ell$ in Section \ref{sec:qpara}, which are quantum analogues of the generators of the parafermion vertex operator algebra $K_{\hat \g}^\ell$ inside $L_{\hat\g}^\ell$.
%By using these generators, we prove that $K_{\hat\g,\hbar}^\ell/\hbar K_{\hat\g,\hbar}^\ell\cong K_{\hat\g}^\ell$.

In this paper, we study the commutant $K_{\hat\g,\hbar}^\ell$ of the quantum Heisenberg vertex algebra within $L_{\hat\g,\hbar}^\ell$.
This commutant admits a natural structure as an $\hbar$-adic nonlocal vertex subalgebra.
In Section \ref{sec:qpara}, we construct a set of generators for $K_{\hat\g,\hbar}^\ell$, representing quantum analogues of the generators of the parafermion vertex operator algebra $K_{\hat \g}^\ell$ contained in $L_{\hat\g}^\ell$.
Utilizing these generators, we prove that $K_{\hat\g,\hbar}^\ell/\hbar K_{\hat\g,\hbar}^\ell\cong K_{\hat\g}^\ell$.
Furthermore, we surprisingly find that the quantum Yang-Baxter operator associated with $L_{\hat\g,\hbar}^\ell$ acts trivially on $K_{\hat\g,\hbar}^\ell$.
Consequently, $K_{\hat\g,\hbar}^\ell$ possesses not only the structure of an $\hbar$-adic quantum vertex algebra but also that of an $\hbar$-adic vertex algebra.

The double commutant property of the parafermion vertex operator algebra $K_{\hat\g}^\ell$ and the lattice vertex algebra $V_{\sqrt\ell Q_L}$ within $L_{\hat\g}^\ell$ is fundamental in the study of parafermion vertex algebras.
Our next objective is to establish its quantum analogue (Theorem \ref{thm:qlattice-inj} and Theorem \ref{thm:double-comutant}).
Compared to the classical case, proving the existence of an embedding of the quantum lattice vertex algebra $V_{\sqrt\ell Q_L}^{\eta_\ell}$
into $L_{\hat\g,\hbar}^\ell$ (Theorem \ref{thm:qlattice-inj}) requires substantially more intricate techniques.
In \cite[Corollary 4.17]{K-q-lattice-va}, we established the existence of an embedding of a quantum lattice vertex algebra into an $\hbar$-adic nonlocal vertex algebra provided it contains a set of elements whose associated vertex operators satisfy certain relations. Applying this result requires proving that $L_{\hat\g,\hbar}^\ell$ contains a set of elements satisfying exactly the relations \eqref{A1}--\eqref{A7}.
The verification of relations \eqref{A1}--\eqref{A4} is straightforward. The proof for the remaining relations reduces to the case $\g=\ssl_2$. When $\ell=1$, the quantum affine vertex algebra $L_{\hat\ssl_2,\hbar}^1$ is isomorphic to the quantum lattice vertex algebra $V_{\Z\al_1}^{\eta_1}$, where $\Z\al_1$ denotes the root lattice of $\ssl_2$; consequently, the embedding's existence is immediate in this case. We then utilize the injection $\Delta: L_{\hat\ssl_2,\hbar}^{\ell+1} \to L_{\hat\ssl_2,\hbar}^{\ell} \hat\ot L_{\hat\ssl_2,\hbar}^1$ from \cite{K-Coproduct-q-aff-va} to establish the embedding for general $\ell$ by induction.
Finally, applying \cite[Corollary 4.18]{K-q-lattice-va}, the proof of the double commutant property (Theorem \ref{thm:double-comutant}) reduces to showing that $K_{\hat\g,\hbar}^\ell$ lies in the commutant of the quantum Heisenberg vertex algebra within $L_{\hat\g,\hbar}^\ell$—which holds by the very definition of $K_{\hat\g,\hbar}^\ell$.

The paper is organized as follows.
Section \ref{sec:qvas} presents the basics about $\hbar$-adic quantum vertex algebras and the theory of twisted tensor products.
Section \ref{sec:qaff-va} presents the construction of quantum affine vertex algebras as described in \cite{K-Quantum-aff-va},
along with their twisted tensor products and coproducts.
Section \ref{sec:qlattice} discusses the construction of quantum lattice vertex algebras, and demonstrates that
there exists an embedding from the quantum lattice vertex algebra $V_{\sqrt\ell Q_L}^{\eta_\ell}$ to $L_{\hat\g,\hbar}^\ell$.
The proof of the embedding can be reduced to the case $\g=\ssl_2$ with the proof provided in Section \ref{subsec:sl2-case}.
Section \ref{sec:qpara} focuses on identifying a generating subset of the quantum parafermion vertex algebra $K_{\hat\g,\hbar}^\ell$,  whose verification is given in Section \ref{subsec:pf-prop-W}. Additionally, Section \ref{sec:qpara} establishes the double commutant property of $K_{\hat\g,\hbar}^\ell$ and $V_{\sqrt\ell Q_L}^{\eta_\ell}$.

Throughout this paper, we denote by $\Z_+$ and $\N$ the set of positive and nonnegative integers, respectively.
For a vector space $W$ and $g(z)\in W[[z,z\inv]]$, we denote by $g(z)^+$ (resp. $g(z)^-$) the regular (singular) part of $g(z)$.

\section{Quantum vertex algebras}\label{sec:qvas}

In this section, we recall the we recall some basic results on vertex algebras, $\hbar$-adic nonlocal vertex algebras,
$\hbar$-adic quantum vertex algebras, $\hbar$-adic $n$-quantum vertex algebras and their twisted tensor products.

A \emph{vertex algebra} (VA) is a vector space $V$ together with a \emph{vacuum vector} $\vac\in V$ and a vertex operator map
\begin{align}
    Y(\cdot,z):&V\longrightarrow \E(V):=\Hom(V,V((z)));\quad
    v\mapsto Y(v,z)=\sum_{n\in\Z}v_nz^{-n-1},
\end{align}
such that
\begin{align}\label{eq:vacuum-property}
    Y(\vac,z)v=v,\quad Y(v,z)\vac\in V[[z]],\quad \lim_{z\to 0}Y(v,z)\vac=v,
    \quad\te{for }v\in V,
\end{align}
and that
\begin{align}\label{eq:Jacobi}
    &z_0\inv\delta\(\frac{z_1-z_2}{z_0}\)Y(u,z_1)Y(v,z_2)
    -z_0\inv\delta\(\frac{z_2-z_1}{-z_0}\)Y(v,z_2)Y(u,z_1)\\
    &\quad=z_1\inv\delta\(\frac{z_2+z_0}{z_1}\)Y(Y(u,z_0)v,z_2)
    \quad\quad\te{for }u,v\in V.\nonumber
\end{align}

A \emph{module} $W$ of a VA $V$ is a vector space $W$ together with a vertex operator map
\begin{align}
    Y_W(\cdot,z):&V\longrightarrow \E(W);\quad
    v\mapsto Y_W(v,z)=\sum_{n\in\Z}v_nz^{-n-1},
\end{align}
such that $Y_W(\vac,z)=1_W$ and
\begin{align*}
    &z_0\inv\delta\(\frac{z_1-z_2}{z_0}\)Y_W(u,z_1)Y_W(v,z_2)
    -z_0\inv\delta\(\frac{z_2-z_1}{-z_0}\)Y_W(v,z_2)Y_W(u,z_1)\\
    &\quad=z_1\inv\delta\(\frac{z_2+z_0}{z_1}\)Y_W(Y(u,z_0)v,z_2)
    \quad\quad\te{for }u,v\in V.\nonumber
\end{align*}

In this paper, we let $\hbar$ be a formal variable, and let $\C[[\hbar]]$ be the ring of formal power series in $\hbar$.
A $\C[[\hbar]]$-module $V$ is \emph{topologically free} if $V=V_0[[\hbar]]$
for some vector space $V_0$ over $\C$.
It is known that a $\C[[\hbar]]$-module $W$ is topologically free if and only if it is Hausdorff complete under the $\hbar$-adic topology and
\begin{align*}
  \hbar w=0\quad\te{implies}\quad w=0\quad\te{for any }w\in W
\end{align*}
(see \cite{Kassel-topologically-free}).
For another topologically free $\C[[\hbar]]$-module $U=U_0[[\hbar]]$, we recall the complete tensor
\begin{align*}
    U\wh\ot V=(U_0\ot V_0)[[\hbar]].
\end{align*}
The following useful result can be found in \cite{Kassel-topologically-free}.
\begin{lem}\label{lem:topo-free-inj-surj}
Let $U$ and $V$ be two topologically free $\C[[\hbar]]$-modules and let $f:U\to V$ be a $\C[[\hbar]]$-module map.
Denote by $f_0:U/\hbar U\to V/\hbar V$ the $\C$-linear map induced from $f$.
Then $f$ is injective (resp. surjective, bijective) if $f_0$ is injective (resp. surjective, bijective).
\end{lem}

We view a vector space as a $\C[[\hbar]]$-module by letting $\hbar=0$.
Fix a $\C[[\hbar]]$-module $W$. For $k\in\Z_+$, and some formal variables $z_1,\dots,z_k$, we define
\begin{align}
    \E^{(k)}(W;z_1,\dots,z_k)=\Hom_{\C[[\hbar]]}\(W,W((z_1,\dots,z_k))\)
\end{align}
We will denote $\E^{(k)}(W;z_1,\dots,z_k)$ by $\E^{(k)}(W)$ if there is no ambiguity,
and will denote $\E^{(1)}(W)$ by $\E(W)$.
An ordered sequence $(a_1(z),\dots,a_k(z))$ in $\E(W)$ is said to be \emph{quasi-compatible} (\cite{Li-nonlocal})
if there exists a nonzero polynomial $p(z_1,z_2)$ such that
\begin{align*}
    \(\prod_{1\le i<j\le k}p(z_i,z_j)\)a_1(z_1)\cdots a_k(z_k)\in\E^{(k)}(W).
\end{align*}
This ordered sequence is called \emph{compatible} if $p(z_1,z_2)=(z_1-z_2)^m$ for some $m\in\N$ (see \cite[Definition 3.5]{Li-G-phi}).

Now, we assume that $W=W_0[[\hbar]]$ is topologically free.
Define
\begin{align}
    \E_\hbar^{(k)}(W;z_1,\dots,z_k)=\Hom_{\C[[\hbar]]}\(W,W_0((z_1,\dots,z_k))[[\hbar]]\).
\end{align}
Similarly, we denote $\E_\hbar^{(k)}(W;z_1,\dots,z_k)$ by $\E_\hbar^{(k)}(W)$ if there is no ambiguity,
and denote $\E_\hbar^{(1)}(W)$ by $\E_\hbar(W)$ for short.
We note that $\E_\hbar^{(k)}(W)=\E^{(k)}(W_0)[[\hbar]]$ is topologically free.
For $n,k\in\Z_+$, the quotient map from $W$ to $W/\hbar^nW$ induces the following $\C[[\hbar]]$-module map
\begin{align*}
    \wt\pi_n^{(k)}:\End_{\C[[\hbar]]}(W)[[z_1^{\pm 1},\dots,z_k^{\pm 1}]]
    \to \End_{\C[[\hbar]]}(W/\hbar^nW)[[z_1^{\pm 1},\dots,z_k^{\pm 1}]].
\end{align*}
For $A(z_1,z_2),B(z_1,z_2)\in\Hom_{\C[[\hbar]]}(W,W_0((z_1))((z_2))[[\hbar]])$, we write $A(z_1,z_2)\sim B(z_2,z_1)$
if for each $n\in\Z_+$ there exists $k\in\N$, such that
\begin{align*}
    (z_1-z_2)^k\wt\pi_n^{(2)}(A(z_1,z_2))=(z_1-z_2)^k\wt\pi_n^{(2)}(B(z_2,z_1)).
\end{align*}

For each $k\in\Z_+$, the inverse system
\begin{align*}
    \xymatrix{
    0&W/\hbar W\ar[l]&W/\hbar^2W\ar[l]&W/\hbar^3W\ar[l]&\cdots\ar[l]
    }
\end{align*}
induces the following inverse system
\begin{align}\label{eq:E-h-inv-sys}
    \xymatrix{
    0&\E^{(k)}(W/\hbar W)\ar[l]&\E^{(k)}(W/\hbar^2W)\ar[l]&\cdots\ar[l]
    }
\end{align}
Then $\E_\hbar^{(k)}(W)$ is isomorphic to the inverse limit of \eqref{eq:E-h-inv-sys}.
The map $\wt \pi_n^{(k)}$ induces a $\C[[\hbar]]$-module $\pi_n^{(k)}:\E_\hbar^{(k)}(W)\to \E^{(k)}(W/\hbar^nW)$.
It is easy to verify that $\ker \pi_n^{(k)}=\hbar^n\E_\hbar^{(k)}(W)$.
We will denote $\pi_n^{(1)}$ by $\pi_n$ for short.
An ordered sequence $(a_1(z),\dots,a_r(z))$ in $\E_\hbar(W)$ is called \emph{$\hbar$-adically compatible} if for every $n\in\Z_+$,
the sequence $$(\pi_n(a_1(z)),\dots,\pi_n(a_r(z)))$$ in $\E(W/\hbar^nW)$ is compatible.
A subset $U$ of $\E_\hbar(W)$ is called \emph{$\hbar$-adically compatible} if every finite sequence in $U$ is $\hbar$-adically compatible.
Let $(a(z),b(z))$ in $\E_\hbar(W)$ be $\hbar$-adically compatible. That is, for any $n\in\Z_+$, we have
\begin{align*}
    (z_1-z_2)^{k_n}\pi_n(a(z_1))\pi_n(b(z_2))\in\E^{(2)}(W/\hbar^nW)\quad\te{for some }k_n\in\N.
\end{align*}
We recall the following vertex operator map (\cite{Li-h-adic}):
\begin{align}\label{eq:def-Y-E}
    &Y_\E(a(z),z_0)b(z)=\sum_{n\in\Z}a(z)_nb(z)z_0^{-n-1}\\
    =&\varprojlim_{n>0}z_0^{-k_n}\left.\((z_1-z)^{k_n}\pi_n(a(z_1))\pi_n(b(z))\)\right|_{z_1=z+z_0}.\nonumber
\end{align}

An \emph{$\hbar$-adic nonlocal VA} (\cite{Li-h-adic}, see also \cite[Definition 2.1 and Remark 2.3]{JKLT-Defom-va}) is a topologically free $\C[[\hbar]]$-module $V$ equipped with a vacuum vector $\vac$ and a vertex operator map $Y(\cdot,z):V\to \E_\hbar(V)$ such that the vacuum property \eqref{eq:vacuum-property} hold,
\begin{align*}
    \set{Y(u,z)}{u\in V}\subset\E_\hbar(V)\quad\te{is $\hbar$-adically compatible},
\end{align*}
and that
\begin{align}\label{eq:weak-asso}
    Y_\E(Y(u,z),z_0)Y(v,z)=Y(Y(u,z_0)v,z)\quad\te{for }u,v\in V.
\end{align}
We denote by $\partial$ the canonical derivation of $V$:
\begin{align}
    u\to\partial u=\lim_{z\to 0}\frac{d}{dz}Y(u,z)\vac.
\end{align}
Moreover, a \emph{$V$-module} is a topologically free $\C[[\hbar]]$-module $W$ equipped with a vertex operator map $Y_W(\cdot,z):V\to \E_\hbar(W)$, such that $Y_W(\vac,z)=1_W$,
\begin{align*}
    \set{Y_W(u,z)}{u\in V}\subset\E_\hbar(W)\quad\te{is $\hbar$-adically compatible},
\end{align*}
and that
\begin{align*}
    Y_\E(Y_W(u,z),z_0)Y_W(v,z)=Y_W(Y(u,z_0)v,z)\quad\te{for }u,v\in V.
\end{align*}

For a topologically free $\C[[\hbar]]$-module $V$ and a submodule $U\subset V$, we denote by $\bar U$ the closure of $U$ and set
\begin{align}
  [U]=\set{u\in V}{\hbar^n u\in U\,\,\te{for some }n\in\Z_+}.
\end{align}
Then $\overline{[U]}$ is the minimal closed submodule that is invariant under the operation $[\cdot]$.
Suppose further that $V$ is an $\hbar$-adic nonlocal VA, and $\vac\in U$.
We set
\begin{align}
  U'=\Span_{\C[[\hbar]]}\set{u_mv}{u,v\in U,\,m\in\Z},\quad\te{and}\quad \wh U=\overline{[U']}.
\end{align}
For a subset $S\subset V$, we define
\begin{align*}
  S^{(1)}=\overline{\left[\Span_{\C[[\hbar]]}S\cup\{\vac\}\right]},\quad\te{and}\quad
  S^{(n+1)}=\wh{S^{(n)}}\supset S^{(n)} \quad\te{for }n\ge 1.
\end{align*}
Set
\begin{align*}
  \<S\>=\overline{\left[\cup_{n\ge 1}S^{(n)}\right]}.
\end{align*}
Then $\<S\>$ is the unique minimal closed $\hbar$-adic nonlocal subVA of $V$, such that $[\<S\>]=\<S\>$.
In addition, we say that $V$ is generated by $S$ if $\<S\>=V$.

\begin{de}\label{de:qyb-mult}
Let $(V_i,Y_i,\vac)$ ($1\le i\le n$) be $\hbar$-adic nonlocal VAs.
We call a family of $\C[[\hbar]]$-module maps $S_{ij}(z):V_i\wh\ot V_j\to V_i\wh\ot V_j\wh\ot \C((z))[[\hbar]]$ ($1\le i,j\le n$)
\emph{quantum Yang-Baxter operators} of the $n$-tuple $(V_1,\dots, V_n)$, if
\begin{align}
  &S_{ij}(z)(v\ot \vac)=v\ot \vac\quad
   S_{ij}(z)(\vac\ot u)=\vac\ot u\quad\te{for }u\in V_j,\,\,v\in V_i,\label{eq:qyb-vac}\\
  &[\partial\ot 1,S_{ij}(z)]=-\frac{d}{dz}S_{ij}(z),
  \quad [1\ot\partial, S_{ij}(z)]=\frac{d}{dz}S_{ij}(z),\quad S_{ji}^{21}(z)S_{ij}(-z)=1,\label{eq:qyb-der-shift}\\
  &Y_i(u,z_1)Y_i(v,z_2)\sim Y_i(z_2)(1\ot Y_i(z_1))S_{ii}(z_2-z_1)(v\ot u),\quad\te{for }u,v\in V_i,\label{eq:qyb-locality}\\
  &S_{ij}(z_1)Y_i^{12}(z_2)=Y_i^{12}(z_2)S_{ij}^{23}(z_1)S_{ij}^{13}(z_1+z_2),\label{eq:qyb-hex1}\\
  &S_{ij}(z_1)Y_j^{23}(z_2)=Y_j^{23}(z_2)S_{ij}^{12}(z_1-z_2)S_{ij}^{13}(z_1),\label{eq:qyb-hex2}\\
  &S_{ij}^{12}(z_1)S_{ik}^{13}(z_1+z_2)S_{jk}^{23}(z_2)
  =S_{jk}^{23}(z_2)S_{ik}^{13}(z_1+z_2)S_{ij}^{12}(z_1)\quad\te{for } 1\le i,j,k\le n.\label{eq:qyb-qybeq}
\end{align}
We call the pair $((V_i)_{i=1}^n,(S_{ij}(z))_{i,j=1}^n)$ an \emph{$\hbar$-adic $n$-quantum VA}.
Moreover, let $$((V'_i)_{i=1}^n,(S'_{ij}(z))_{i,j=1}^n)$$ another $\hbar$-adic $n$-quantum VA, and let $f_i:V_i\to V'_i$ ($1\le i\le n$) be $\hbar$-adic nonlocal VA homomorphisms, that is, each $f_i$ preserves vertex operator maps and maps vacuum vectors to vacuum vectors.
We call $(f_1,\dots,f_n)$ an \emph{$\hbar$-adic $n$-quantum VA homomorphism} if
\begin{align}
  (f_i\ot f_j)\circ S_{ij}(z)=S'_{ij}(z)\circ(f_i\ot f_j)\quad\te{for }1\le i,j\le n.
\end{align}
\end{de}

\begin{rem}
The notion of $\hbar$-adic $1$-quantum VAs is identical to the notion of \emph{$\hbar$-adic quantum VAs}.
\end{rem}

\begin{rem}\label{rem:twisted-tensor}
Let $V_1$ and $V_2$ be two $\hbar$-adic nonlocal VAs, and let $S_{12}(z):V_1\wh\ot V_2\to V_1\wh\ot V_2\wh\ot \C((z))[[\hbar]]$ be a $\C[[\hbar]]$-module map satisfying the relations \eqref{eq:qyb-vac}, \eqref{eq:qyb-hex1} and \eqref{eq:qyb-hex2}.
Then $S_{ij}(z)\sigma$ is called a \emph{twisting operator} for the ordered pair $(V_1,V_2)$ (\cite{LS-twisted-tensor}).
It was proved in \cite{LS-twisted-tensor} that $V_1\wh\ot V_2$ carries an $\hbar$-adic nonlocal VA structure with vertex operator map defined as follows:
\begin{align*}
  Y_{\{12\}}(z)=Y_1^{12}(z)Y_2^{34}(z)S_{12}^{23}(-z)\sigma.
\end{align*}
\end{rem}

\begin{rem}\label{rem:iterated-twisted-tensor}
Let $V_1,V_2,\dots,V_n$ be $\hbar$-adic nonlocal VAs.
For each $1\le i<j\le n$, let $S_{ij}(z):V_i\wh\ot V_j\to V_i\wh\ot V_j\wh\ot \C((z))[[\hbar]]$ be $\C[[\hbar]]$-module maps satisfying the relations \eqref{eq:qyb-vac}, \eqref{eq:qyb-hex1}, \eqref{eq:qyb-hex2} and \eqref{eq:qyb-qybeq}.
Set $V_{\{1,2,\dots,n\}}=V_1\wh\ot V_2\wh\ot \cdots\wh\ot V_n$.
It was proved in \cite{S-iter-twisted-tensor} that $V_{\{1,2,\dots,n\}}$ carries an $\hbar$-adic nonlocal vertex algebra structure with vertex operator map defined as follows:
\begin{align*}
  &Y_{\{1,2,\dots,n\}}(z)=Y_1^{12}(z) Y_2^{34}(z)\cdots Y_n^{2n-1,2n}(z)\\
  \times&\prod_{\substack{a-b=-1\\ 1\le a,b\le n}}S_{a,b}^{2a,2b-1}(-z)
  \prod_{\substack{a-b=-2\\ 1\le a,b\le n}}S_{a,b}^{2a,2b-1}(-z)\cdots
  \prod_{\substack{a-b=1-k\\ 1\le a,b\le n}}S_{a,b}^{2a,2b-1}(-z)
  \sigma^{23}\sigma^{345}\cdots  \sigma^{n,n+1,\dots,2n-1},
\end{align*}
where $\sigma^{a,a+1,\dots,b}=\sigma^{a,a+1}\sigma^{a+1,a+2}\cdots\sigma^{b-1,b}$.
\end{rem}

Let $(V_i,Y_i,\vac_i)_{i=1}^n$ be an $\hbar$-adic $n$-quantum VA with quantum Yang-Baxter operators $(S_{ij}(z))_{i,j=1}^n$,
and let $K=\{i_1<i_2<\cdots<i_k\}$ be the subset of $\{1,2,\dots,n\}$.
It is obvious that $((V_i)_{i\in K},(S_{ij}(z))_{i,j\in K})$ is an $\hbar$-adic $k$-quantum VA.
According to Remark \ref{rem:iterated-twisted-tensor}, one has an iterated twisted tensor product $\hbar$-adic nonlocal VA $V_K$.
Let $L=\{j_1<j_2<\cdots<j_l\}$ be another subset of $\{1,2,\dots,n\}$. We define
\begin{align}\label{eq:def-S-Delta}
  S_{K,L}(z)=\prod_{\substack{a-b=k-1\\ 1\le a\le k\\ 1\le b\le l}}S_{i_a,j_b}^{a,k+b}(z)
    \prod_{\substack{a-b=k-2\\ 1\le a\le k\\ 1\le b\le l}}S_{i_a,j_b}^{a,k+b}(z)\cdots
    \prod_{\substack{a-b=1-l\\ 1\le a\le k\\ 1\le b\le l}}S_{i_a,j_b}^{a,k+b}(z).
\end{align}

The following result was proved in \cite{K-Coproduct-q-aff-va}.
\begin{prop}\label{prop:n-qva-to-k-qva}
Let $((V_i)_{i=1}^n,(S_{ij}(z))_{i,j=1}^n)$ be an $\hbar$-adic $n$-quantum VA, and $(K_1,\dots,K_k)$ a partition of the set $\{1,2,\dots,n\}$.
Then $((V_{K_i})_{i=1}^k,(S_{K_i,K_j}(z))_{i,j=1}^k)$ is an $\hbar$-adic $k$-quantum VA.
\end{prop}

\section{Quantum affine vertex algebras}\label{sec:qaff-va}

In this section, we recall the quantum affine VA introduced in \cite{K-Quantum-aff-va} and their twisted tensor products and coproducts given in \cite{K-Coproduct-q-aff-va}.

Let $A=(a_{ij})_{i,j\in I}$ be a Cartan matrix, and let $\g=\g(A)$ be the corresponding finite dimensional simple Lie algebra over $\C$.
We fix a realization $(\h,\Pi,\Pi^\vee)$ of $\g$, where $\h$ is a Cartan subalgebra of $\g$, $\Pi=\set{\al_i}{i\in I}\subset\h^\ast$ is the set of simple roots and
$\Pi^\vee=\set{h_i}{i\in I}\subset\h$ is the set of simple coroots.
Define
\begin{align}
    I_L=\set{i\in I}{\al_i\,\,\te{is a long root}},\quad
    I_S=\set{i\in I}{\al_i\,\,\te{is a short root}}.
\end{align}
And define
\begin{align}\label{eq:def-r-i}
    r=\begin{cases}
        1,&\mbox{if $\g$ is simply-laced},\\
        2,&\mbox{if $\g$ is of type $B_n,C_n,F_4$},\\
        3,&\mbox{if $\g$ is of type $G_2$},
    \end{cases}
    \quad r_i=\begin{cases}
        1,&\mbox{if }i\in I_S,\\
        r,&\mbox{if }i\in I_L.
    \end{cases}
\end{align}
Then we get from \cite[(6.22)]{Kac-book} and \cite[Table Aff]{Kac-book} that
\begin{align}\label{eq:sym}
    r_ia_{ij}=r_ja_{ji}\quad\te{for }i,j\in I.
\end{align}
Define a bilinear map $\<\cdot,\cdot\>$ on $\h^\ast$ by
\begin{align}\label{eq:bilinear-form}
    \<\al_i,\al_j\>=r_ia_{ij}/r\quad\te{for }i,j\in I.
\end{align}
Then $\<\cdot,\cdot\>$ is the normalized symmetric bilinear form on $\h^\ast$,
such that $\<\al,\al\>=2$ for any long root $\al$.
Since $\<\cdot,\cdot\>$ is non-degenerated, we can identify $\h^\ast$ with $\h$ via $\<\cdot,\cdot\>$. To be more precise, we identify $\al_i$ with $r_ih_i /r$ for $i\in I$.
Then
\begin{align}
    \<h_i,h_j\>=ra_{ij}/r_j\quad\te{for }i,j\in I.
\end{align}

Let $\hat\g=\g\ot\C[t,t\inv]\oplus\C c$ be the affinization of $\g$. One has that (\cite{Gar-loop-alg}):
\begin{prop}
The Lie algebra $\hat\g$ is isomorphic to the Lie algebra generated by
\begin{align*}
    \set{h_i(m),\,x_i^\pm(m)}{i\in I}
\end{align*}
and a central element $c$, subject to the relations written in terms of generating functions in $z$:
\begin{align*}
    h_i(z)=\sum_{m\in\Z}h_i(m)z^{-m-1},\quad x_i^\pm(z)=\sum_{m\in\Z}x_i^\pm(m)z^{-m-1},\quad i\in I.
\end{align*}
The relations are $(i,j\in I)$:
\begin{align}
    \tag{L1}\label{L1} &[h_i(z_1),h_j(z_2)]={r_i}{a_{ij}}rc\pd{z_2}z_1\inv\delta\(\frac{z_2}{z_1}\),\\
    \tag{L2}\label{L2} &[h_i(z_1),x_j^\pm(z_2)]=\pm r_ia_{ij}x_j^\pm(z_2)z_1\inv\delta\(\frac{z_2}{z_1}\),\\
    \tag{L3}\label{L3} &\left[x_i^+(z_1),x_j^-(z_2)\right]=\frac{\delta_{ij}}{r_i}\(h_i(z_2)z_1\inv\delta\(\frac{z_2}{z_1}\)+rc\pd{z_2}z_1\inv\delta\(\frac{z_2}{z_1}\)  \),\\
    \tag{L4}\label{L4}&(z_1-z_2)^{n_{ij}}\left[x_i^\pm(z_1),x_j^\pm(z_2)\right]=0,\\
    \tag{S}\label{S} &\left[ x_i^\pm(z_1),\left[ x_i^\pm(z_2),\dots,\left[ x_i^\pm(z_{m_{ij}}),x_j^\pm(z_0) \right]\cdots \right] \right]=0,\quad \te{if }a_{ij}\le 0,
\end{align}
where $n_{ij}=1-\delta_{ij}$ for $i,j\in I$ and $m_{ij}=1-a_{ij}$ for $i,j\in I$ with $a_{ij}\le 0$.
\end{prop}

Introduce a set $\mathcal B=\set{h_i,\,x_i^\pm}{i\in I}$, and defined a function $N:\mathcal B\times\mathcal B\to\N$ by
\begin{align*}
    N(h_i,h_j)=2,\quad N(h_i,x_j^\pm)=1,\quad N(x_i^\pm,x_j^\pm)=n_{ij},\quad N(x_i^+,x_j^-)=\delta_{ij}2.
\end{align*}
By utilizing the theory of free VAs introduced in \cite{R-free-conformal-free-va}, one gets a VA $V(\mathcal B,N)$, such that
for any VA $V$ containing $\mathcal B$ as a subset and
\begin{align}
  (z_1-z_2)^{N(a,b)}[Y(a,z_1),Y(b,z_2)]=0\quad \te{for }a,b\in\mathcal B,
\end{align}
there exists a VA homomorphism $f:V(\mathcal B,N)\to V$ uniquely determined by $f(a)=a$ for $a\in \mathcal B$.

\begin{de}\label{de:affVAs}
For $\ell\in \C$, we let $F_{\hat\g}^\ell$ be the quotient VA of $V(\mathcal B,N)$ modulo the ideal generated by
\begin{align*}
    (h_i)_0(h_j),\quad (h_i)_1(h_j)-\frac{a_{ij}}{r_j}r\ell\vac,\quad (h_i)_0(x_j^\pm)\mp a_{ij} x_j^\pm \quad \te{for }i,j\in I.
\end{align*}
Furthermore, let $V_{\hat\g}^\ell$ be the quotient VA of $F_{\hat\g}^\ell$ modulo the ideal generated by
\begin{align*}
    &(x_i^+)_0(x_j^-)-\delta_{ij}h_i,\quad (x_i^+)_1(x_j^-)-\delta_{ij}r\ell/r_i \vac\quad\te{for }i,j\in I,\\
    &\(x_i^\pm\)_0^{m_{ij}}(x_j^\pm)\quad\te{for }i,j\in I\,\,\te{with }a_{ij}\le 0.
\end{align*}
If $\ell\in\Z_+$, we define $L_{\hat\g}^\ell$ to be the quotient VA $V_{\hat\g}^\ell$ modulo the ideal generated by
\begin{align*}
    \(x_i^\pm\)_{-1}^{r\ell/r_i}(x_i^\pm)\quad\te{for }i\in I.
\end{align*}
\end{de}

\begin{rem}\label{rem:aff-vas}
Set $\hat\g_+=\g\ot \C[t]\oplus\C c$. For $\ell\in \C$, we let $\C_\ell:=\C$ be a $\hat\g_+$-module, with $\g\ot\C[t].\C_\ell=0$ and $c=\ell$. Then there are $\hat\g$-module structures on both $V_{\hat\g}^\ell$ and $L_{\hat\g}^\ell$,
such that
\begin{align*}
    h_i(z)=Y(h_i,z),\quad x_i^\pm(z)=Y(x_i^\pm,z),\quad i\in I.
\end{align*}
In addition, as $\hat\g$-modules, we have that
\begin{align*}
    V_{\hat\g}^\ell\cong \U(\hat\g)\ot_{\U(\hat\g_+)}\C_\ell.
\end{align*}
And $L_{\hat\g}^\ell$ is the unique simple quotient $\hat\g$-module of $V_{\hat\g}^\ell$, when $\ell\in\Z_+$.
\end{rem}

\begin{rem}\label{rem:M-epsilon}
In $F_{\hat\g}^\ell$, the fields $Y(h_i,z)$ and $Y(x_i^\pm,z)$ satisfy the relations \eqref{L1},
\eqref{L2}, \eqref{L4} with $h_i(z)$ (resp. $x_i^\pm(z)$) replaced by $Y(h_i,z)$ (resp. $Y(x_i^\pm,z)$),
and satisfy the following relation
\begin{align}
  \tag{L3.5}\label{L3.5}&(z_1-z_2)^{2\delta_{ij}}[Y(x_i^\pm,z_1),Y(x_j^\mp,z_2)]=0.
\end{align}
\end{rem}

Now, we recall the construction of quantum affine VAs introduced in \cite{K-Quantum-aff-va}.
Let $q=\exp \hbar$.
For $0\ne P(z)\in\C(z)[[\hbar]]$ and $g(q)=\sum_{k=1}^n a_kq^{m_k}\in\Z[q,q\inv]$, we let
\begin{align}
  P(z)^{g(q)}=\prod_{k=1}^n \left(q^{m_k\pd{z}}P(z)^{a_k}\right)=\prod_{k=1}^n \left( P(z+m_k\hbar)^{a_k} \right) \in \C(z)[[\hbar]].
\end{align}
It is straightforward to check that
\begin{align*}
  P(z)^{g_1(q)}P(z)^{g_2(q)}=P(z)^{g_1(q)+g_2(q)},\quad \(P(z)^{g_1(q)}\)^{g_2(q)}=P(z)^{g_1(q)g_2(q)}
\end{align*}
for any $g_1(q),g_2(q)\in\Z[q]$, and
\begin{align*}
  \(P(z_1)^{g(q)}\)|_{z_1=-z}=\(Q(z)\)^{g(q\inv)},\quad\te{with}\,\,Q(z)=P(-z).
\end{align*}

For $n\in\Z$ and any invertible element $\nu$ in an algebra over $\C$ with $\nu^2\ne 1$, we set
\begin{align*}
  [n]_\nu=\frac{\nu^n-\nu^{-n}}{\nu-\nu\inv}.
\end{align*}\
Furthermore, for nonnegative integers $s\le t$, define
\begin{align*}
  [s]_\nu!=[s]_\nu [s-1]_\nu \cdots [1]_\nu,\quad \binom{t}{s}_\nu=\frac{[t]_\nu!}{[s]_\nu![s-t]_\nu!}.
\end{align*}

Let $\ell$ be a complex number. We define following functions in $\C((z))[[\hbar]]$:
\begin{align}
  &\wh\ell_{ij}(z)=-[r_ia_{ij}]_{q^{\pd{z}}}[r\ell]_{q^{\pd{z}}}q^{r\ell\pd{z}}\pdiff{z}{2}\log f(z)
    -r_ia_{ij}r\ell z^{-2},\label{eq:tau-1}\\
  &\wh\ell_{ij}^{1,\pm}(z)=\wh\ell_{ji}^{2,\pm}(z)=[r_ia_{ij}]_{q^{\pd{z}}}q^{r\ell\pd{z}}\pd{z}\log f(z)
    -r_ia_{ij}z\inv,\label{eq:tau-2}\\
  %&\wh\ell_{ij}^{2,\pm}(z)=[a_{ji}]_{q^{r_j\pd{z}}}q^{r\ell\pd{z}}\pd{z}\log f(z)
%    -a_{ji}z\inv,\label{eq:tau-2-2}\\
  &\wh\ell_{ij}^{\pm,\pm}(z)=\begin{cases}
  f(z)^{q^{-r_ia_{ii}}-1},&\mbox{if }a_{ij}> 0,\\
  z\inv f(z)^{q^{-r_ia_{ij}}},&\mbox{if }a_{ij}\le 0,
  \end{cases}
  \label{eq:tau-3}\\
  &\wh\ell_{ij}^{+,-}(z)=z^{-\delta_{ij}}(z+2r\ell\hbar)^{\delta_{ij}},\label{eq:tau-4}\\
  &\wh\ell_{ij}^{-,+}(z)=z^{-\delta_{ij}}(z-2r\ell\hbar)^{\delta_{ij}}
    f(z)^{q^{r_ia_{ij}}-q^{-r_ia_{ij}}},\label{eq:tau-5}
\end{align}
for $i,j\in I$, where %$q_i=q^{r_i}$ and $g_{ij,\hbar}(z)=(1-q_i^{a_{ij}}e^{-z})/(q_i^{a_{ij}}-e^{-z})\in\C((z))[[\hbar]]$.
\begin{align}
  f(z)=e^{z/2}-e^{-z/2}\in\C[[z]].
\end{align}

\begin{rem}
Note that
\begin{align*}
  \pd z \log f(z)=\frac{1+e^{-z}}{2-2e^{-z}}\quad\te{and}\quad \pdiff z 2 \log f(z)=-\frac{e^{-z}}{(1-e^{-z})^2}.
\end{align*}
Then
\begin{align*}
  &\wh\ell_{ij}(z)=[r_ia_{ij}]_{q^{\pd{z}}}[r\ell]_{q^{\pd{z}}}q^{r\ell\pd{z}}\frac{e^{-z}}{(1-e^{-z})^2}
    -r_ia_{ij}r\ell z^{-2},\\
  &\wh\ell_{ij}^{1,\pm}(z)=\wh\ell_{ji}^{2,\pm}(z)=[r_ia_{ij}]_{q^{\pd{z}}}q^{r\ell\pd{z}}\frac{1+e^{-z}}{2-2e^{-z}}
    -r_ia_{ij}z\inv.
\end{align*}
Moreover,
\begin{align*}
  \wh\ell_{ij}^{\pm,\pm}(z)=&
  \begin{cases}
    f(z)\inv f(z-r_ia_{ii}\hbar),&\mbox{if }a_{ij}>0,\\
    z\inv f(z-r_ia_{ij}\hbar),&\mbox{if }a_{ij}\le 0
  \end{cases}\\
  =&\begin{cases}
    (e^{z/2}-e^{-z/2})\inv (q^{-r_i}e^{z/2}-q^{r_i}e^{-z/2}),&\mbox{if }a_{ij}>0,\\
    z\inv (q^{-r_ia_{ij}/2}e^{z/2}-q^{r_ia_{ij}/2}e^{-z/2}),&\mbox{if }a_{ij}\le 0,
  \end{cases}
\end{align*}
and
\begin{align*}
  \wh\ell_{ij}^{-,+}(z)=&z^{-\delta_{ij}}(z-2r\ell\hbar)^{\delta_{ij}}f(z+r_ia_{ij}\hbar)f(z-r_ia_{ij}\hbar)\inv\\
  =&z^{-\delta_{ij}}(z-2r\ell\hbar)^{\delta_{ij}}
  \frac{q^{r_ia_{ij}/2}e^{z/2}-q^{-r_ia_{ij}/2}e^{-z/2}}{ q^{-r_ia_{ij}/2}e^{z/2}-q^{r_ia_{ij}/2}e^{-z/2} }\\
  =&z^{-\delta_{ij}}(z-2r\ell\hbar)^{\delta_{ij}}
  \frac{q^{r_ia_{ij}}-e^{-z}}{ 1-q^{r_ia_{ij}}e^{-z} }.
\end{align*}
Therefore, the functions $\wh\ell_{ij}(z)$, $\wh\ell_{ij}^{1,\pm}(z)$, $\wh\ell_{ij}^{2,\pm}(z)$, $\wh\ell_{ij}^{\pm,\pm}(z)$,
$\wh\ell_{ij}^{+,-}(z)$ and $\wh\ell_{ij}^{-,+}(z)$ coincide with the functions $\tau_{ij}(z)$, $\tau_{ij}^{1,\pm}(z)$,
$\tau_{ij}^{2,\pm}(z)$, $\tau_{ij}^{\pm,\pm}(z)$, $\tau_{ij}^{+,-}(z)$ and $\tau_{ij}^{-,+}(z)$ given in
\cite[(6.1)-(6.5)]{K-Quantum-aff-va}, respectively.
\end{rem}

Recall from \cite{Li-smash} that a \emph{pseudo-endomorphism} of an $\hbar$-adic nonlocal VA $V$ is a $\C[[\hbar]]$-module map $A(z):V\to V\wh\ot\C((z))[[\hbar]]$, such that
\begin{align*}
  A(z)\vac=\vac\ot 1,\quad
  A(z_1)Y(u,z_2)v=Y(A(z_1-z_2)u,z_2)v\quad\te{for } u,v\in V.
\end{align*}
A \emph{pseudo-derivation} of $V$ is a $\C[[\hbar]]$-module map $D(z):V\to V\wh\ot\C((z))[[\hbar]]$, such that
\begin{align*}
  [D(z_1),Y(u,z_2)]=Y(D(z_1-z_2)u,z_2)\quad\te{for } u\in V.
\end{align*}
It was proved in \cite{K-Quantum-aff-va}:
\begin{lem}
For each $i\in I$, there are pseudo-derivations $\wh\ell_i(z)$ and pseudo-endomorphisms $\wh\ell_i^\pm(z)$ on
$F_{\hat\g}^\ell[[\hbar]]$ uniquely determined by ($i,j\in I$)
\begin{align*}
  &\wh\ell_i(z)h_j=\vac\ot \wh\ell_{ij}(z),\quad \wh\ell_i(z)x_j^\pm=\pm x_j^\pm\ot \wh\ell_{ij}^{1,\pm}(z),\\
  &\wh\ell_i^\pm(z)h_j=h_j\ot 1\mp \vac\ot \wh\ell_{ij}^{2,\pm}(z),\quad
  \wh\ell_i^\pm(z)x_j^\epsilon=x_j^\epsilon\ot \wh\ell_{ij}^{\pm,\epsilon}(z)\inv.
\end{align*}
\end{lem}

\begin{de}
Define $\mathcal M_\hbar^\ell(\g)$ to be the category consisting of topologically free $\C[[\hbar]]$-modules $W$ equipped with fields
$h_{i,\hbar}(z),x_{i,\hbar}^\pm(z)\in\E_\hbar(W)$ ($i\in I$) satisfying the following relations
\begin{align}
  &[h_{i,\hbar}(z_1),h_{j,\hbar}(z_2)]
  =[r_ia_{ij}]_{q^{\pd{z_2}}}[r\ell]_{q^{\pd{z_2}}}\tag{Q1}\label{eq:local-h-1}\\
  \times&
  \(\iota_{z_1,z_2}q^{-r\ell\pd{z_2}}-\iota_{z_2,z_1}q^{r\ell\pd{z_2}}\)
  \pd{z_1}\pd{z_2}\log f(z_1-z_2),\nonumber\\
  &[h_{i,\hbar}(z_1),x_{j,\hbar}^\pm(z_2)]
  =\pm x_{j,\hbar}^\pm(z_2)[r_ia_{ij}]_{q^{r_i\pd{z_2}}}\tag{Q2}\label{eq:local-h-2}\\
  \times&
  \(\iota_{z_1,z_2}q^{-r\ell\pd{z_2}}-\iota_{z_2,z_1}q^{r\ell\pd{z_2}}\)
  \pd{z_1}\log f(z_1-z_2),\nonumber\\
  &\iota_{z_1,z_2}f(z_1-z_2)^{\delta_{ij}+\delta_{ij}q^{2r\ell}}
  x_{i,\hbar}^+(z_1)x_{j,\hbar}^-(z_2)\tag{Q3.5}\label{eq:local-h-4}\\
    =&\iota_{z_2,z_1}f(z_1-z_2)^{\delta_{ij}+\delta_{ij}q^{2r\ell}+q^{-r_ia_{ij}}-q^{r_ia_{ij}}}
    x_{j,\hbar}^-(z_2)x_{i,\hbar}^+(z_1),\nonumber\\
  &\iota_{z_1,z_2}f(z_1-z_2)^{-\delta_{ij}+q^{-r_ia_{ij}}}
  x_{i,\hbar}^\pm(z_1)x_{j,\hbar}^\pm(z_2)\tag{Q4}\label{eq:local-h-3}\\
    =&\iota_{z_2,z_1} f(-z_2+z_1)^{-\delta_{ij}+q^{r_ia_{ij}}}
    x_{j,\hbar}^\pm(z_2)x_{i,\hbar}^\pm(z_1).\nonumber
\end{align}
\end{de}

The following result was proved in \cite[Lemma 5.2, Proposition 5.3 and Propositon 5.4]{K-Quantum-aff-va} (see also \cite[Proposition 6.1]{K-Coproduct-q-aff-va}):
\begin{prop}\label{prop:universal-M-tau}
There is a unique $\hbar$-adic nonlocal VA $(F_{\hat\g,\hbar}^\ell,Y_{\wh\ell},\vac)$, such that
\begin{itemize}
  \item[(1)] $F_{\hat\g,\hbar}^\ell$ is generated by $\set{h_i,\,x_i^\pm}{i\in I}$,
  and $$(F_{\hat\g,\hbar}^\ell,\{Y_{\wh\ell}(h_i,z)\}_{i\in I},\{Y_{\wh\ell}(x_i^\pm,z)\}_{i\in I})$$ is an object of $\mathcal M_\hbar^\ell(\g)$.

  \item[(2)] For any $\hbar$-adic nonlocal VA $(V,Y,\vac)$ containing $\bar h_i, \bar x_i^\pm$ ($i\in I$),
    such that $$(V,\{Y(\bar h_i,z)\}_{i\in I},\{Y(\bar x_i^\pm,z)\}_{i\in I})$$
    is an object of $\mathcal M_\hbar^\ell(\g)$,
    then there exists a unique $\hbar$-adic nonlocal VA homomorphism $\varphi:F_{\hat\g,\hbar}^\ell\to V$,
    such that $\varphi(h_i)=\bar h_i$ and $\varphi(x_i^\pm)=\bar x_i^\pm$ for $i\in I$.
\end{itemize}
%Moreover, $F_{\hat\g,\hbar}^\ell\cong F_{\hat\g}^\ell[[\hbar]]$ as $\C[[\hbar]]$-modules, the vertex operator map $Y_{\wh\ell}$ is determined by
%\begin{align*}
%  &Y_{\wh\ell}(h_i,z)=Y(h_i,z)+\wh\ell_i(z),\quad
%  Y_{\wh\ell}(x_i^\pm,z)=Y(x_i^\pm,z)\wh\ell_i^\pm(z)\quad\te{for }i\in I.
%\end{align*}
\end{prop}

The $\hbar$-adic nonlocal VA $F_{\hat\g,\hbar}^\ell$ can also be realized as a deformation of $F_{\hat\g}^\ell[[\hbar]]$.
An \emph{$\hbar$-adic (nonlocal) vertex bialgebra} \cite{Li-smash} is an $\hbar$-adic (nonlocal) VA $H$ equipped with a classical coalgebra structure $(\Delta,\varepsilon)$ such that (the coproduct) $\Delta:H\to H\wh\ot H$
and (the counit) $\varepsilon:H\to\C[[\hbar]]$ are homomorphisms of $\hbar$-adic nonlocal VAs.
An \emph{$H$-module (nonlocal) VA} \cite{Li-smash} is an $\hbar$-adic nonlocal VA $V$ equipped with a module structure $\tau$ on $V$ for $H$ viewed as an $\hbar$-adic nonlocal VA such that
\begin{align}
  &\tau(h,z)v\in V\ot \C((z)),\qquad
  \tau(h,z)\vac_V=\varepsilon(h)\vac_V,\label{eq:mod-va-for-vertex-bialg1-2}\\
  &\tau(h,z_1)Y(u,z_2)v=\sum Y(\tau(h_{(1)},z_1-z_2)u,z_2)\tau(h_{(2)},z_1)v
  \label{eq:mod-va-for-vertex-bialg3}
\end{align}
for $h\in H$, $u,v\in V$, where $\vac_V$ denotes the vacuum vector of $V$
and $\Delta(h)=\sum h_{(1)}\ot h_{(2)}$ is the coproduct in the Sweedler notation.
A \emph{(right) $H$-comodule nonlocal VA} (\cite[Definition 2.22]{JKLT-Defom-va}) is a nonlocal VA $V$ equipped with a homomorphism
$\rho:V\to V\wh\ot H$ of $\hbar$-adic nonlocal VAs such that
\begin{align}
  (\rho\ot 1)\rho=(1\ot \Delta)\rho,\quad (1\ot \varepsilon)\rho=\te{Id}_V.
\end{align}
We say that $\rho$ and $\tau$ are \emph{compatible} (\cite[Definition 2.24]{JKLT-Defom-va}) if
\begin{align}
  \rho(\tau(h,z)v)=(\tau(h,z)\ot 1)\rho(v)\quad \te{for }h\in H,\,v\in V.
\end{align}

\begin{de}
Let $V$ be an $\hbar$-adic nonlocal VA.
A \emph{deforming triple} is a triple $(H,\rho,\tau)$, where $H$ is a cocommutative $\hbar$-adic nonlocal vertex bialgebra, $(V,\rho)$ is an $H$-nonlocal VA and $(V,\tau)$ is an $H$-module nonlocal VA, such that $\rho$ and $\tau$ are compatible.
\end{de}

The following represents the direct $\hbar$-adic analogue of \cite[Theorem 2.25]{JKLT-Defom-va}.

\begin{thm}\label{thm:deform-va}
Let $V$ be an $\hbar$-adic nonlocal VA, and let $(H,\rho,\tau)$ be a deforming triple.
Define
\begin{align}
\mathfrak D_\tau^\rho (Y)(a,z)=\sum Y(a_{(1)},z)\tau(a_{(2)},z)\quad\te{for }a\in V,
\end{align}
where $\rho(a)=\sum a_{(1)}\ot a_{(2)}\in V\otimes H$.
Then $(V,\mathfrak D_\tau^\rho (Y),\vac)$ forms an $\hbar$-adic nonlocal VA,
which is denoted by $\mathfrak D_\tau^\rho (V)$.
\end{thm}

We recall the special commutative and cocommutative $\hbar$-adic vertex bialgebra $H$ given in \cite[Section 5]{K-Quantum-aff-va} as follows:
Let $H_0$ be the symmetric algebra of the following vector space, and let $H=H_0[[\hbar]]$:
\begin{align*}
  \bigoplus_{i\in I}\(\C[\partial] \(\C A_i\oplus\C B_i^+\oplus\C B_i^-\)\).
\end{align*}
The coproduct $\Delta$ and counit $\varepsilon$ are uniquely determined by
($i\in I$, $n\in\N$):
\begin{align*}
  &\Delta(\partial^n A_i)=\partial^n A_i\ot 1+1\ot \partial^n A_i,
  \quad\Delta(\partial^n B_i^\pm)=\sum_{k=0}^n\binom{n}{k} \partial^k B_i^\pm\ot \partial^{n-k} B_i^\pm,\\
  &\varepsilon(\partial^n A_i)=0,
  \quad\varepsilon(\partial^n B_i^\pm)=\delta_{n,0}.
\end{align*}
The vertex operator map $Y$ is defined by $Y(u,z)v=(e^{z\partial}u)v$, where $\partial$ is a derivation on $H$ determined by
\begin{align*}
  \partial (\partial^n A_i)=\partial^{n+1}A_i,\quad \partial(\partial^n B_i^\pm)=\partial^{n+1}B_i^\pm\quad\te{for } i\in I,\,n\in\N.
\end{align*}

Combining \cite[Proposition 6.3]{K-Coproduct-q-aff-va} with \cite[Remark 6.2]{K-Coproduct-q-aff-va}, we immediately get the following result.
\begin{lem}
There is a deforming triple $(H,\rho,\wh\ell)$ on $F_{\hat\g}^\ell[[\hbar]]$, where $\rho$ is the $H$-comodule nonlocal VA structure uniquely determined by
\begin{align*}
  \rho(h_i)=h_i\ot 1+\vac\ot A_i,\quad \rho(x_i^\pm)=x_i^\pm\ot B_i^\pm,
\end{align*}
and $\wh\ell$ is the $H$-module nonlocal VA structure uniquely determined by
\begin{align*}
  \wh\ell(A_i,z)=\wh\ell_i(z),\quad \wh\ell(B_i^\pm,z)=\wh\ell_i^\pm(z).
\end{align*}
\end{lem}

Combining this with Theorem \ref{thm:deform-va} and \cite[Proposition 5.12]{K-Quantum-aff-va} (see also \cite[Proposition 6.4]{K-Coproduct-q-aff-va}),
we immediately get the following result.
\begin{prop}\label{prop:desc-nonlocalVA-F}
$\mathfrak D_{\wh\ell}^\rho(F_{\hat\g}^\ell[[\hbar]])\cong F_{\hat\g,\hbar}^\ell$.
Moreover, the vertex operator map $\mathfrak D_{\wh\ell}^\rho(Y)$ of $\mathfrak D_{\wh\ell}^\rho(F_{\hat\g}^\ell[[\hbar]])$
is uniquely determined by
\begin{align*}
  &\mathfrak D_{\wh\ell}^\rho(Y)(h_i,z)=Y(h_i,z)+\wh\ell_i(z),\quad \mathfrak D_{\wh\ell}^\rho(Y)(x_i^\pm,z)=Y(x_i^\pm,z)\wh\ell_i^\pm(z)
  \quad\te{for }i\in I.
\end{align*}
\end{prop}

\begin{rem}
In the remainder of this paper, we identify the two $\hbar$-adic nonlocal VAs $\mathfrak D_{\wh\ell}^\rho(F_{\hat\g}^\ell[[\hbar]])$
and $F_{\hat\g,\hbar}^\ell$.
\end{rem}

Define
\begin{align}\label{eq:def-f-0}
   f_0(z)=\frac{f(z)}{z}=\frac{e^{z/2}-e^{-z/2}}{z}=\sum_{n\ge 0}\frac{z^{2n}}{4^n(2n+1)!}\in 1+z^2\C[[z^2]].
\end{align}

\begin{de}\label{de:V-tau}
For $\ell\in\C$, we let $R_{1}^\ell$ be the minimal closed ideal of
$F_{\hat\g,\hbar}^{\ell}$ such that $[R_{1}^\ell]=R_{1}^\ell$ and contains the following elements
\begin{align*}
  &\(x_i^+\)_0x_i^--\frac{1}{q^{r_i}-q^{-r_i}}\(\vac-E_\ell(h_i)\),\quad
  \(x_i^+\)_1x_i^-+\frac{2r\ell\hbar}{q^{r_i}-q^{-r_i}}E_\ell(h_i)\quad\te{for } i\in I,\\
  &\(x_i^\pm\)_0^{m_{ij}}x_j^\pm\quad\te{for } i,j\in I\,\,\te{with}\,\,a_{ij}\le 0,
\end{align*}
where
\begin{align}\label{eq:def-E-h}
  &E_\ell(h_i)=\(\frac{f_0(2r_i\hbar+2r\ell\hbar)}{f_0(2r_i\hbar-2r\ell\hbar)}\)^\half
  \exp\(\(-q^{-r\ell\partial}2\hbar f_0(2\partial\hbar) h_i\)_{-1}\)\vac.
\end{align}
Define
\begin{align*}
  V_{\hat\g,\hbar}^{\ell}=F_{\hat\g,\hbar}^{\ell}/R_{1}^\ell.
\end{align*}
\end{de}

\begin{de}\label{de:L-tau}
Let $\ell\in \Z_+$, and
let $R_{2}^\ell$ be the minimal closed ideal of $V_{\hat\g,\hbar}^{\ell}$ such that $[R_{2}^\ell]=R_{2}^\ell$ and contains the following elements
\begin{align*}
  & \(x_i^\pm\)_{-1}^{r\ell/r_i}x_i^\pm\quad\te{for } i\in I.
\end{align*}
Define
\begin{align*}
  L_{\hat\g,\hbar}^{\ell}=V_{\hat\g,\hbar}^{\ell}/R_{2}^\ell.
\end{align*}
\end{de}

\begin{rem}
For the notation $L_{\hat\g,\hbar}^{\ell}$. we always assume that $\ell\in\Z_+$.
\end{rem}

It was proved in \cite[Theorem 8.14]{K-Quantum-aff-va}:

\begin{thm}\label{thm:classical-limit-L}
For each $\ell\in\Z_+$, we have that $L_{\hat\g,\hbar}^\ell/\hbar L_{\hat\g,\hbar}^\ell\cong L_{\hat\g}^\ell$.
\end{thm}

%\begin{rem}
%\emph{
%Let $X=V,L$.
%The definition of $F_{\hat\g,\hbar}^\ell$ (resp. $X_{\hat\g,\hbar}^\ell$) is slightly different from the $\hbar$-adic quantum VA $F_{\wh\ell}(A,\ell)$ (resp. $X_{\hat\g,\hbar}(\ell,0)$) given in \cite{K-Quantum-aff-va}.
%Recall from \cite{K-Quantum-aff-va}*{Section 6} that $F_{\wh\ell}(A,\ell)$ (resp. $X_{\hat\g,\hbar}(\ell,0)$) is generated by $h_{i,\hbar}$, $x_{i,\hbar}^\pm$ for $i\in I$.
%Then there is a unique $\hbar$-adic quantum VA isomorphism from $F_{\hat\g,\hbar}^\ell$ (resp. $X_{\hat\g,\hbar}^\ell$) to $F_{\wh\ell}(A,\ell)$ (resp. $X_{\hat\g,\hbar}(\ell,0)$) such that
%\begin{align*}
%    h_i\mapsto [r_i]_{q^\partial}\inv h_{i,\hbar},\quad x_i^\pm\mapsto x_{i,\hbar}^\pm\quad \te{for }i\in I.
%\end{align*}
%}
%\end{rem}

Proposition \ref{prop:universal-M-tau} and Definitions \ref{de:V-tau}, \ref{de:L-tau}
provides the following result.
\begin{prop}\label{prop:universal-qaff}
Let $(V,Y,\vac)$ be an $\hbar$-adic nonlocal VA containing a subset $$\set{\bar h_i,\bar x_i^\pm}{i\in I},$$
such that
\begin{align*}
    (V,\{Y(\bar h_i,z)\}_{i\in I},\{Y(\bar x_i^\pm,z)\}_{i\in I})\in\obj \mathcal M_\hbar^\ell(\g).
\end{align*}
Suppose that
\begin{align*}
    &(\bar x_i^+)_0\bar x_i^-=\frac{1}{q^{r_i}-q^{-r_i}}\(\vac-E_\ell(\bar h_i)\),\quad
    (\bar x_i^+)_1\bar x_i^-=-\frac{2r\ell\hbar}{q^{r_i}-q^{-r_i}}E_\ell(\bar h_i)\quad\te{for }i\in I,\\
    &(\bar x_i^\pm)_0^{m_{ij}}\bar x_j^\pm=0\quad\te{for }i,j\in I\,\,\te{with}\,\,a_{ij}\le 0.
\end{align*}
Then the unique $\hbar$-adic nonlocal VA homomorphism $\varphi:F_{\hat\g,\hbar}^\ell\to V$ provided in Proposition \ref{prop:universal-M-tau} factors through $V_{\hat\g,\hbar}^\ell$.
Suppose further that $\ell\in\Z_+$ and
\begin{align*}
    (\bar x_i^\pm)_{-1}^{r\ell/r_i}\bar x_i^\pm=0\quad\te{for }i\in I.
\end{align*}
Then $\varphi$ also factors through $L_{\hat\g,\hbar}^\ell$.
\end{prop}

%We need the following generalization of Proposition \ref{prop:normal-ordering-rel}.
The following result was proved in \cite[Theorem 7.8]{K-Coproduct-q-aff-va}.
\begin{thm}\label{thm:quotient-algs}
Let $\ell_1,\ell_2,\dots,\ell_n\in\C$. Then $$((X_{\hat\g,\hbar}^{\ell_i})_{i=1}^n,(S_{\ell_i,\ell_j}(z))_{i,j=1}^n)$$
is an $\hbar$-adic $n$-quantum VA
for $X=F,V,L$,
where for any $\ell,\ell'\in\C$, the quantum Yang-Baxter operator $S_{\ell,\ell'}(z)$ is determined by ($i,j\in I$):
\begin{align}
  &S_{\ell,\ell'}(z)(h_j\ot h_i)=h_j\ot h_i+\vac\ot\vac\label{eq:S-twisted-1}
  \ot \pdiff{z}{2}\log f(z)^{[r_ia_{ij}]_{q}[r\ell]_{q} [r\ell']_q(q-q\inv) }
  ,\\
  %[a_{ij}]_{q^{r_i\pd{z}}}[r\ell/r_j]_{q^{r_j\pd{z}}}[r\ell']_{q^{\pd{z}}}
  %\(q^{\pd{z}}-q^{-\pd{z}}\)\pdiff{z}{2}\log f(z),\nonumber\\
  &S_{\ell,\ell'}(z)(x_j^\pm\ot h_i)=x_j^\pm\ot h_i\pm x_j^\pm\ot \vac \label{eq:S-twisted-2}
  \ot
  \pd{z}\log f(z)^{ [r_ia_{ij}]_{q}[r\ell']_q(q-q\inv) },\\
  %[a_{ij}]_{q^{r_i\pd{z}}}[r\ell']_{q^{\pd{z}}}
  %\(q^{\pd{z}}-q^{-\pd{z}}\)\pd{z}\log f(z),\nonumber\\
  &S_{\ell,\ell'}(z)(h_j\ot x_i^\pm)=h_j\ot x_i^\pm\mp\vac\ot x_i^\pm\label{eq:S-twisted-3}
  \ot
  \pd{z}\log f(z)^{ [r_ja_{ji}]_{q}[r\ell]_q(q-q\inv) },\\
  %[a_{ji}]_{q^{r_j\pd{z}}}[r\ell]_{q^{\pd{z}}}\(q^{\pd{z}}-q^{-\pd{z}}\)\pd{z}\log f(z),\nonumber\\
  &S_{\ell,\ell'}(z)(x_j^{\epsilon_1}\ot x_i^{\epsilon_2})=x_j^{\epsilon_1}\ot x_i^{\epsilon_2}\ot f(z)^{q^{-\epsilon_1\epsilon_2r_ia_{ij}}
    -q^{\epsilon_1\epsilon_2r_ia_{ij}}}.
    \label{eq:S-twisted-4}
\end{align}
\end{thm}

\begin{rem}
Let $X=F,V,L$, let $n\in\Z_+$ and let $\ell_1,\dots,\ell_n\in\C$. Denote by $X_{\hat\g,\hbar}^{\ell_1}\wh\ot X_{\hat\g,\hbar}^{\ell_2}\wh\ot\cdots\wh\ot X_{\hat\g,\hbar}^{\ell_n}$ the iterated twisted tensor product
$\hbar$-adic nonlocal VA associated with quantum Yang-Baxter operators $(S_{\ell_i,\ell_j}(z))_{i,j=1}^n$,
where the iterated twisted tensor product is defined in Remark \ref{rem:iterated-twisted-tensor}.
Moreover, Proposition \ref{prop:n-qva-to-k-qva} implies that
$$X_{\hat\g,\hbar}^{\ell_1}\wh\ot X_{\hat\g,\hbar}^{\ell_2}\wh\ot\cdots\wh\ot X_{\hat\g,\hbar}^{\ell_n}$$ is an $\hbar$-adic quantum VA.
Furthermore, for $1\le i<n$, applying Proposition \ref{prop:n-qva-to-k-qva} again yields that
$(X_{\hat\g,\hbar}^{\ell_1},\dots,X_{\hat\g,\hbar}^{\ell_{i-1}},
    X_{\hat\g,\hbar}^{\ell_i}\wh\ot X_{\hat\g,\hbar}^{\ell_{i+1}},
    X_{\hat\g,\hbar}^{\ell_{i+2}},\dots,
    X_{\hat\g,\hbar}^{\ell_n})$
is an $\hbar$-adic $(n-1)$-quantum VA.
\end{rem}

\begin{rem}\label{rem:twisted-tensor-classical-limit}
For each $g(q)\in\Z[q]$, we have that
\begin{align*}
  &f(z)^{g(q)}\equiv f(z)^{g(1)}\quad\te{and}\quad
  \pd z \log f(z)^{g(q)}\equiv \pd z \log f(z)^{g(1)}\mod \hbar\C[[z,z\inv,\hbar]].
\end{align*}
Then
\begin{align*}
  S_{\ell,\ell'}(z)(a_j\ot b_i)\equiv a_j\ot b_i\mod \hbar \left(X_{\hat\g,\hbar}^{\ell}\wh\ot X_{\hat\g,\hbar}^{\ell'}\wh\ot \C((z))[[\hbar]]\right)\quad\te{for }a,b=h,x^\pm.
\end{align*}
Since $\set{h_i,x_i^\pm}{i\in I}$ generates the whole $\hbar$-adic quantum VA (see Proposition \ref{prop:universal-M-tau}),
\begin{align*}
  S_{\ell,\ell'}(z)(v\ot u)\equiv v\ot u\mod \hbar \left(X_{\hat\g,\hbar}^{\ell}\wh\ot X_{\hat\g,\hbar}^{\ell'}\wh\ot \C((z))[[\hbar]]\right)\quad \te{for }v\in X_{\hat\g,\hbar}^{\ell},\,u\in X_{\hat\g,\hbar}^{\ell'}.
\end{align*}
Consequently, the twisted tensor product $X_{\hat\g,\hbar}^{\ell}\wh\ot X_{\hat\g,\hbar}^{\ell'}$ induces the usual tensor product
$X_{\hat\g,\hbar}^{\ell}/\hbar X_{\hat\g,\hbar}^{\ell} \ot X_{\hat\g,\hbar}^{\ell'}/\hbar X_{\hat\g,\hbar}^{\ell'}$.
\end{rem}

The following result was proved in \cite[Theorem 8.1]{K-Coproduct-q-aff-va}.

\begin{thm}\label{thm:coproduct}
Let $X=F,V,L$. For any $\ell,\ell'\in\C$, there is an $\hbar$-adic quantum VA homomorphism
$\Delta:X_{\hat\g,\hbar}^{\ell+\ell'}\to X_{\hat\g,\hbar}^{\ell}\wh\ot X_{\hat\g,\hbar}^{\ell'}$
defined by
\begin{align}
  &\Delta(h_i)=q^{-r\ell'\partial}h_i\ot\vac+\vac\ot q^{r\ell\partial}h_i,\tag{Co1}\label{eq:def-Delta-h}\\
  &\Delta(x_i^+)=x_i^+\ot\vac+q^{2r\ell\partial}E_\ell(h_i)\ot q^{2r\ell\partial}x_i^+,\tag{Co2}\label{eq:def-Delta-x+}\\
  &\Delta(x_i^-)=x_i^-\ot\vac+\vac\ot x_i^-,\label{eq:def-Delta-x-}\tag{Co3}
\end{align}
Moreover, let $n\in\Z_+$ and $\ell_1,\dots,\ell_n\in\C$. For $1\le i<n$, we have
\begin{equation*}
\begin{tikzcd}[column sep=1em, row sep=1.5em]
    \Big(X_{\hat\g,\hbar}^{\ell_1}\ar[d,equal]&\dots&X_{\hat\g,\hbar}^{\ell_{i-1}}\ar[d,equal]
    &X_{\hat\g,\hbar}^{\ell_i+\ell_{i+1}}\ar[d,"\Delta"]
    &X_{\hat\g,\hbar}^{\ell_{i+2}}\ar[d,equal]
    &\dots&X_{\hat\g,\hbar}^{\ell_n}\Big)\ar[d,equal]\\
    \Big(X_{\hat\g,\hbar}^{\ell_1}&\dots&X_{\hat\g,\hbar}^{\ell_{i-1}}&
    X_{\hat\g,\hbar}^{\ell_i}\wh\ot X_{\hat\g,\hbar}^{\ell_{i+1}}&
    X_{\hat\g,\hbar}^{\ell_{i+2}}&\dots&
    X_{\hat\g,\hbar}^{\ell_n}\Big)
\end{tikzcd}
\end{equation*}
is an $\hbar$-adic $(n-1)$-quantum VA homomorphism.
Furthermore, for $\ell,\ell',\ell''\in\C$, we have that
\begin{align*}
  (\Delta\ot 1)\circ\Delta=(1\ot\Delta)\circ\Delta:X_{\hat\g,\hbar}^{\ell+\ell'+\ell''}\to X_{\hat\g,\hbar}^{\ell}\wh\ot X_{\hat\g,\hbar}^{\ell'}\wh\ot X_{\hat\g,\hbar}^{\ell''}
\end{align*}
\end{thm}

\begin{rem}\label{rem:inj}
Let $\ell,\ell'\in\Z_+$.
By utilizing Theorem \ref{thm:classical-limit-L}, the $\C[[\hbar]]$-module map $\Delta$ induces the following $\C$-linear map
\begin{align*}
  \Delta_0:L_{\hat\g}^{\ell+\ell'}\cong L_{\hat\g,\hbar}^{\ell+\ell'}/\hbar L_{\hat\g,\hbar}^{\ell+\ell'}
  \to L_{\hat\g,\hbar}^\ell/\hbar L_{\hat\g,\hbar}^\ell\ot L_{\hat\g,\hbar}^{\ell'}/\hbar L_{\hat\g,\hbar}^{\ell'}
  \cong L_{\hat\g}^\ell\ot L_{\hat\g}^{\ell'},
\end{align*}
where both instances of ``$\ot$'' denote the classical tensor product, as established in Remark \ref{rem:twisted-tensor-classical-limit}.
From Remark \ref{rem:aff-vas}, $L_{\hat\g}^\ell$, $L_{\hat\g}^{\ell'}$ and $L_{\hat\g}^{\ell+\ell'}$ are all simple $\hat\g$-modules.
Recalling the definition of $E_\ell(h_i)$ (see \eqref{eq:def-E-h}), we get that $\Delta_0$ is exactly a $\hat\g$-module homomorphism with
$L_{\hat\g}^\ell\ot L_{\hat\g}^{\ell'}$ equipped with the standard tensor $\hat\g$-module structure.
Since $\Delta_0(\vac)=\vac\ot\vac\ne 0$ and $L_{\hat\g}^{\ell+\ell'}$ is simple, $\Delta_0$ must be injective.
Applying Lemma \ref{lem:topo-free-inj-surj}, the map
$\Delta:L_{\hat\g,\hbar}^{\ell+\ell'}\to L_{\hat\g,\hbar}^\ell\wh\ot L_{\hat\g,\hbar}^{\ell'}$ is also injective.
\end{rem}

Finally, we collect some formulas given in \cite{K-Coproduct-q-aff-va} that will be used later on.

\begin{prop}\label{prop:normal-ordering-rel-general}
For any $i\in I$ and $n\in\Z_+$, we set
\begin{align*}
  &\:Y_{\wh\ell}(x_i^\pm,z_1)\cdots Y_{\wh\ell}(x_i^\pm,z_n)\;
  =\(\prod_{1\le s<t\le n} f(z_s-z_t)^{-1+q^{-2r_i}}\)
  Y_{\wh\ell}(x_i^\pm,z_1)\cdots Y_{\wh\ell}(x_i^\pm,z_n).
\end{align*}
Then we have that
\begin{align}
  &\:Y_{\wh\ell}(x_{\sigma(1)}^\pm,z_{\sigma(1)})\cdots Y_{\wh\ell}(x_{\sigma(n)}^\pm,z_{\sigma(n)})\;
  =\:Y_{\wh\ell}(x_i^\pm,z_1)\cdots Y_{\wh\ell}(x_i^\pm,z_n)\;,\quad\te{for }\sigma\in S_n,\\
  &\:Y_{\wh\ell}(x_i^\pm,z_1)\cdots Y_{\wh\ell}(x_i^\pm,z_n)\;\in\E_\hbar^{(n)}(W),\\
  &Y_{\wh\ell}\(\(x_i^\pm\)_{-1}^n\vac,z\)
  =\prod_{a=1}^{n-1}f_0(2ar_i\hbar)\\
  &\quad\times  \:Y_{\wh\ell}(x_i^\pm,z+2(n-1)r_i\hbar)
    Y_{\wh\ell}(x_i^\pm,z+2(n-2)r_i\hbar)\cdots Y_{\wh\ell}(x_i^\pm,z)\;,\nonumber\\
  &\Rat_{z_1\inv,\dots,z_n\inv}Y_{\wh\ell}(x_i^\pm,z_1)\cdots Y_{\wh\ell}(x_i^\pm,z_n)\vac
    =\prod_{a=1}^n\frac{z_a}{z_a-2(n-a)r_i\hbar}\(x_i^\pm\)_{-1}^n\vac.
\end{align}
\end{prop}

\begin{lem}\label{lem:com-formulas}
For $i\in I$, we define
\begin{align*}
  &h_i^-(z)=Y(h_i,z)^-
  +\wh\ell_i(z),\quad
   h_i^+(z)=Y(h_i,z)^+,\\
  &\wt h_i^\pm(z)=-q^{-r\ell\pd{z}}2\hbar f_0\(2\hbar\pd{z}\)h_i^\pm(z).
\end{align*}
Then we have that
\begin{align}
  &[h_i^-(z_1),h_j^+(z_2)]\label{eq:com-formulas-1}
  %[a_{ij}]_{q^{r_i\pd{z_2}}}[r\ell/r_j]_{q^{r_j\pd{z_2}}}q^{-r\ell\pd{z_2}}
    =\pd{z_1}\pd{z_2}\log f(z_1-z_2)^{ [r_ia_{ij}]_{q}[r\ell]_{q}q^{r\ell} },\\
  &[h_i^-(z_1),Y_{\wh\ell}(x_j^\pm,z_2)]=\pm Y_{\wh\ell}(x_j^\pm,z_2)%[a_{ij}]_{q^{r_i\pd{z_2}}}q^{-r\ell\pd{z_2}}
  \pd{z_1}\log f(z_1-z_2)^{[r_ia_{ij}]_{q}q^{r\ell}},\label{eq:com-formulas-2}\\
%  &[h_i^+(z_1),Y_{\wh\ell}(x_j^\pm,z_2)]=\mp Y_{\wh\ell}(x_j^\pm,z_2)%[a_{ij}]_{q^{r_i\pd{z_2}}}q^{r\ell\pd{z_2}}
%  \pd{z_1}\log f(-z_2+z_1)^{[r_ia_{ij}]_{q}q^{-r\ell}}.\label{eq:com-formulas-3}
  &[\wt h_i^-(z_1),\wt h_j^+(z_2)]\label{eq:com-formulas-6}
  =\log f(z_1-z_2)^{( q^{2r\ell}-1 )( q^{-r_ia_{ij}}-q^{r_ia_{ij}} )},\\
  &[\wt h_i^-(z_1),Y_{\wh\ell}(x_j^\pm,z_2)]\label{eq:com-formulas-7}
  =\pm Y_{\wh\ell}(x_j^\pm,z_2)%\( q^{r_ia_{ij}\pd{z_2}}-q^{-r_ia_{ij}\pd{z_2}} \)
  \log f(z_1-z_2)^{q^{-r_ia_{ij}}-q^{r_ia_{ij}}},\\
  &\left[h_i^-(z_1),\exp\(\wt h_j^+(z_2)\)\right]
  =\exp\(\wt h_j^+(z_2)\) \label{eq:com-formulas-10}
  \pd{z_1}\log f(z_1-z_2)^{[r_ia_{ij}]_{q}[r\ell]_qq^{2r\ell} },\\
  &S_{\ell,\ell'}(z)\(E_\ell(h_j)\ot h_i\)\label{eq:S-E-1}
  =E_\ell(h_j)\ot h_i+E_\ell(h_j)\ot \vac\\
  &\qquad\qquad\ot \pd{z}\log f(z)^{-[r_ia_{ij}]_{q}[r\ell']_{q}[r\ell]_{q}
  \(q-q\inv\)^2q^{-r\ell}}.\nonumber
\end{align}
\end{lem}

\begin{prop}\label{prop:Y-E}
For each $i\in I$, we have that
\begin{align*}
  Y_{\wh\ell}(E_\ell(h_i),z)=\exp(\wt h_i^+(z))\exp(\wt h_i^-(z)).
\end{align*}
\end{prop}

\begin{prop}\label{prop:delta-int+}
For $i\in I$ and $k\ge 0$, we have that
\begin{align*}
  &\(\Delta\(x_i^+\)\)_{-1}^k(\vac\ot\vac)
  =\sum_{t=0}^k f_0(2tr_i\hbar)\binom{k-1}{t}_{q^{r_i}}
  \binom{k}{t}_{q^{r_i}}\binom{k-1}{t}\inv\\
%    \prod_{a=1}^{k-1}F(ar_i)\prod_{a=1}^{t-1}F(ar_i)\inv\prod_{a=1}^{k-t-1}F(ar_i)\inv\\
  \times&\prod_{a=0}^{k-t-1}\exp\(\wt h_i^+(2(ar_i+r\ell)\hbar)\)
  q^{2r_i(k-t)\partial}\(x_i^+\)_{-1}^t\vac
  \ot q^{2r\ell\partial}\(x_i^+\)_{-1}^{k-t}\vac.
\end{align*}
\end{prop}

\begin{prop}\label{prop:delta-int-}
For $i\in I$ and $k\ge 0$, we have that
\begin{align*}
  &\(\Delta(x_i^-)\)_{-1}^k(\vac\ot\vac)\\
  =&\sum_{t=0}^k(x_i^-)_{-1}^{k-t}\vac
  \ot q^{2(k-t)r_i\partial}(x_i^-)_{-1}^t\vac
  \ot f_0(2tr_i\hbar)\binom{k-1}{t}_{q^{r_i}}
  \binom{k}{t}_{q^{r_i}}\binom{k-1}{t}\inv.
\end{align*}
\end{prop}

\section{Quantum lattice vertex algebras}\label{sec:qlattice}

Let $\ell$ be a fixed positive integer, and let $Q_L$ be the lattice generated by long roots of $\g$.
In this section, we construct a quantum lattice VA $V_{\sqrt \ell Q_L}^{\eta_\ell}$ (see Theorem \ref{thm:qlatticeVA}) associated with the lattice $\sqrt \ell Q_L$ and a formal series $\eta_\ell(z)\in \C((z))[[\hbar]]$ (see \eqref{eq:def-eta}) following the framework in \cite{JKLT-Defom-va} (see also \cite[Section 4]{K-q-lattice-va}).
The main purpose of this section is to establish an embedding $V_{\sqrt\ell Q_L}^{\eta_\ell}\hookrightarrow L_{\hat\g,\hbar}^\ell$.

From \cite[\S 6.7]{Kac-book}, we get that
\begin{align}
    Q_L=\oplus_{i\in I}\Z r/r_i\al_i.
\end{align}
To enhance readability and convenience, we set
\begin{align}
  \beta_i=\sqrt\ell r/r_i\al_i\quad\te{for }i\in I.
\end{align}
From \eqref{eq:bilinear-form}, we get that
\begin{align}
  \<\beta_i,\beta_j\>=\ell r/r_j a_{ij}\quad\te{for }i\in I.
\end{align}
Then $\sqrt\ell Q_L$ is a positive definite even lattice.
We first recall the lattice VA $V_{\sqrt\ell Q_L}$ associated to the specific lattice $\sqrt\ell Q_L$ (see \cite[Section 6.4-6.5]{LL} for the construction of general lattice VAs).

Let $\hat\h$ be the Lie subalgebra of $\hat\g$ spanned by $\set{c,h_i(n)}{i\in I,n\in\Z}$,
and let $\h(n)$ be the subspace of $\hat\h$ spanned by $h_i(n)$ for $n\in\Z$.
Define the following two abelian Lie subalgebras
\begin{align*}
  \hat\h^\pm=\oplus_{n>0}\h(\pm n)
\end{align*}
and identify $\h$ with $\h(0)$. Set
\begin{align*}
  \hat\h'=\hat\h^+\oplus\hat\h^-\oplus\C c
\end{align*}
which is a Heisenberg algebra. Then $\hat\h=\hat\h'\oplus \h$, a direct sum of Lie algebras.

We fix a total order $<$ on the index set $I$. Let $\epsilon:\sqrt\ell Q_L\times \sqrt\ell Q_L\to \C^\times$
be the bimultiplicative map defined by
\begin{align}
    \epsilon(\beta_i,\beta_j)=\begin{cases}
        1,&\mbox{if }i\le j,\\
        \exp(\pi\<\beta_i,\beta_j\>\sqrt{-1}),&\mbox{if }i>j.
    \end{cases}
\end{align}
Then $\epsilon$ is a $2$-cocycle of $\sqrt\ell Q_L$ satisfying the following relations:
\begin{align*}
    \epsilon(\beta,\gamma)\epsilon(\gamma,\beta)\inv=(-1)^{\<\beta,\gamma\>},\quad \epsilon(\beta,0)=1=\epsilon(0,\beta)\quad \te{for }\beta,\gamma\in \sqrt\ell Q_L.
\end{align*}
Denote by $\C_\epsilon[\sqrt\ell Q_L]$ the $\epsilon$-twisted group algebra of $\sqrt\ell Q_L$, which by definition has a designated basis $\{e_\beta\,\mid\,\beta\in \sqrt\ell Q_L\}$ with relations
\begin{align}
    e_\beta\cdot e_\gamma=\epsilon(\beta,\gamma)e_{\beta+\gamma}\quad\te{for }\beta,\gamma\in \sqrt\ell Q_L.
\end{align}

Make $\C_{\epsilon}[\sqrt\ell Q_L]$ an $\hat \h$-module by letting $\hat\h'$
act trivially and letting $\h$ act by
\begin{eqnarray}
  h e_\be=\<h,\be\>e_\be \   \  \mbox{ for }h\in \h,\  \be\in \sqrt\ell Q_L.
\end{eqnarray}
Note that $S(\hat\h^-)$ is naturally an $\hat\h$-module of level $1$.
Define
\begin{eqnarray}
V_{\sqrt\ell Q_L}=S(\hat\h^-)\otimes \C_{\epsilon}[\sqrt\ell Q_L],
\end{eqnarray}
the tensor product of $\hat \h$-modules, which is an $\hat \h$-module of level $1$.
Set
$$\vac=1\ot e_0\in  V_{\sqrt\ell Q_L}.$$
Identify $\h$ and $\C_{\epsilon}[{\sqrt\ell Q_L}]$ as subspaces of $V_{\sqrt\ell Q_L}$ via the correspondence
$$h\mapsto h(-1)\otimes 1\   (h\in\h) \quad\te{and}\quad e_\al\mapsto 1\otimes e_\al\   (\al\in {\sqrt\ell Q_L}).$$
For $\al\in {\sqrt\ell Q_L}$ set
\begin{align}\label{eq:def-E}
  E^\pm(\al,z)=\exp\(\sum_{n\in\Z_+} \frac{\al(\mp n)}{\pm n}z^{\pm n} \)
\end{align}
on $V_{\sqrt\ell Q_L}$.
For $\al\in {\sqrt\ell Q_L}$,  define
$z^{\al}:\   \C_{\epsilon}[{\sqrt\ell Q_L}]\rightarrow \C_{\epsilon}[{\sqrt\ell Q_L}][z,z^{-1}]$ by
\begin{eqnarray}
z^\al\cdot e_\be=z^{\<\al,\be\>}e_\be\   \   \   \mbox{ for }\be\in {\sqrt\ell Q_L}.
\end{eqnarray}
Then there exists a VA structure on $V_{\sqrt\ell Q_L}$, which is uniquely determined by
the conditions that $\vac$ is the vacuum vector and that
%\begin{eqnarray}
%Y(\cdot,z):\  V_L\rightarrow (\te{End} V_{L^o})[[z,z^{-1}]]
%\end{eqnarray}
\begin{align*}
&Y_{\sqrt\ell Q_L}(h,z)=h(z),\quad
Y_{\sqrt\ell Q_L}(e_\al,z)=E^+(\al,z)E^-(\al,z)e_\al z^\al\quad\te{for } h\in\h,\,\,\al\in {\sqrt\ell Q_L}.
\end{align*}

To simplify notation, we adopt the following convention.
\begin{convention}\label{con:log}
We use the standard expansion
\begin{align*}
  \log f(z)/z=&\log \frac{e^{z/2}-e^{-z/2}}{z}
  =\log \left( 1+z^2\sum_{n=1}^\infty \frac{z^{2n-2}}{4^n(2n+1)!} \right)\\
  =&\sum_{m=1}^\infty \frac{(-1)^{m-1}}{m}\left(z^2\sum_{n=1}^\infty \frac{z^{2n-2}}{4^n(2n+1)!} \right)^m\in z^2\C[[z^2]].
\end{align*}
Additionally, for $g(q)=\sum_k a_kq^k\in\Z[q,q\inv]$, we employ the standard expansion
\begin{align*}
  &\log f(z)^{g(q)}-g(1)\log z=\sum_k a_k\log f(z+k\hbar)-g(1)\log z\\
  =&g(1)\log f(z)/z+\sum_k\sum_{n=1}^\infty a_k \frac{k^n\hbar^n}{n!}\pdiff z n \log f(z)\in\C((z))[[\hbar]].
\end{align*}
\end{convention}

Let $$D=\diag\set{r/r_i}{i\in I}.$$ Then $AD$ is symmetric.
Set
\begin{align*}
  A(z)=([a_{ij}]_{e^{r_iz}})_{i,j\in I},\quad D(z)=\diag\set{[r/r_i]_{e^{z}}}{i\in I}.
\end{align*}
Since both $A$ and $D$ are invertible, we define
\begin{align}\label{eq:def-C-ell-z}
  C^\ell(z)=(c_{ij}^\ell(z))_{i,j\in I}
  =\frac{1}{z}\(e^{r\ell z}[\ell]_{e^{rz}}A(z)D(z)D\inv A\inv \ell\inv -E\),
\end{align}
where $E$ stands for the identity matrix.
Define $\eta_\ell(\cdot,z):\h\to \h\ot \C((z))[[\hbar]]$ as follows
\begin{align}\label{eq:def-eta}
  \eta_\ell(\sqrt\ell r/r_i\al_i,z)=&\hbar\sum_{j\in I}\sqrt\ell r/r_j\al_j\ot
  c_{ij}^\ell\(\hbar\pd{z}\)\pd{z}\log f(z)\\
  &\nonumber+\sqrt\ell r/r_i\al_i\ot\log f(z)/z.
\end{align}
It is straightforward to verify that $\eta_\ell$ satisfies the following condition:
\begin{align}\label{eq:eta-neg-cond}
    &\eta_\ell(\beta_i,z)|_{\hbar=0}\in \h\ot z\C[[z]],\quad
    \eta_\ell(\beta_i,z)^-\in \h\ot \hbar\C[z\inv][[\hbar]]\quad\te{for }i\in I.
\end{align}
Moreover, we have
\begin{lem}\label{lem:eta-basic}
For $i,j\in I$, we have that
\begin{align}
  &\<\eta_\ell(\beta_i,z),\beta_j\>
  =\log f(z)^{ [a_{ij}]_{q^{r_i}}[r\ell/r_j]_{q^{r_j}}q^{r\ell} }-a_{ij}r\ell/r_j \log z,
    \label{eq:eta-sp}\\
  &e^{\<\eta_\ell(\beta_i,z),\beta_j\>}=f(z)^{[a_{ij}]_{q^{r_i}}[r\ell/r_j]_{q^{r_j}}q^{r\ell} }z^{-a_{ij}r\ell/r_j}.
    \label{eq:eta-sp-exp}
\end{align}
\end{lem}

\begin{proof}
Note that
\begin{align*}
  \left({ [a_{ij}]_{q^{r_i}}[r\ell/r_j]_{q^{r_j}}q^{r\ell} }\right)|_{q=1}=a_{ij}r\ell/r_j.
\end{align*}
The RHS of \eqref{eq:eta-sp} is a well-defined series in $\C((z))[[\hbar]]$ (see Convention \ref{con:log}).
Based on the definition of $C^\ell(z)$ as shown in \eqref{eq:def-C-ell-z}, we see that
\begin{align*}
  z\ell C^\ell(z)AD=e^{r\ell z}[\ell]_{e^{rz}}A(z)D(z)-\ell AD.
\end{align*}
Consequently, we have that
\begin{align*}
  &\<\eta_\ell(\beta_i,z),\beta_j\>
  =\hbar\sum_{k\in I}\ell  a_{kj}r/r_j c_{ik}^\ell\(\hbar\pd{z}\)\pd{z}\log f(z)
  +\ell r/r_j a_{ij}\log f(z)/z\nonumber\\
  =&\sum_{k\in I}\left(\hbar\pd z\right)\ell c_{ik}^\ell\left(\hbar\pd z \right)a_{kj}\ r/r_j\log f(z)+\ell a_{ij}\ r/r_j\log f(z)/z\\
  =&\left(e^{r\ell \hbar\pd z}[\ell]_{e^{r\hbar\pd z}}[a_{ij}]_{e^{r_i\hbar\pd z}}[r/r_j]_{e^{\hbar\pd z}}-\ell a_{ij}r/r_j\right)\log f(z)+\ell a_{ij}\ r/r_j\log f(z)/z\\
  %=&[a_{ij}]_{q^{r_i\pd z}}[r/r_j]_{q^{\pd z}}q^{r\ell \pd z}\log f(z)-a_{ij}r/r_j\log z\\
  %=&\log f(z)^{ q^{r\ell}[\ell]_{q^{r}}[a_{ij}]_{q^{r_i}}[r/r_j]_{q^{r_j}} }-r\ell/r_j a_{ij}\log z
  =&\log f(z)^{ [a_{ij}]_{q^{r_i}}[r\ell/r_j]_{q^{r_j}}q^{r\ell} }- a_{ij}r\ell/r_j\log z.
\end{align*}
This completes the proof of \eqref{eq:eta-sp}.
The relation \eqref{eq:eta-sp-exp} follows immediately from \eqref{eq:eta-sp}.
\end{proof}

For a $\C[[\hbar]]$-module $W$, and a linear map $\varphi(\cdot,z):\h\to W[[z,z\inv]]$,
%Let $W$ be a topologically free $\C[[\hbar]]$-module.
%For $h\in\h$, $f(z)\in\C((z))[[\hbar]]$ and a linear map $\varphi(\cdot,z):\h\to \E_\hbar(W)$,
we define (see \cite{EK-qva})
\begin{align}\label{eq:def-Phi}
  \Phi(h\ot f(z),\varphi)=\Res_{z_1}\varphi(h,z_1)f(z-z_1).
\end{align}
Define $\Phi(G(z),\varphi)$ for $G(z)\in \h\ot \C((z))[[\hbar]]$ in the obvious way.
And extend the bilinear form $\<\cdot,\cdot\>$ on $\h$ to a $\C((z))[[\hbar]]$-linear form on $\h\ot \C((z))[[\hbar]]$.
By using \cite[Theorem 4.11, Proposition 4.12]{K-q-lattice-va}, we have that

\begin{thm}\label{thm:qlatticeVA}
There exists an $\hbar$-adic quantum VA structure on $V_{\sqrt\ell Q_L}[[\hbar]]$ with the vertex operator map $Y_{\sqrt\ell Q_L}^{\eta_\ell}$ uniquely determined by $(i\in I)$:
\begin{align*}
    &Y_{\sqrt\ell Q_L}^{\eta_\ell}(\beta_i,z)=Y_{\sqrt\ell Q_L}(\beta_i,z)+\Phi(\eta_\ell'(\beta_i,z),Y_{\sqrt\ell Q_L}),\\
    &Y_{\sqrt\ell Q_L}^{\eta_\ell}(e_{\pm\beta_i},z)=Y_{\sqrt\ell Q_L}(e_{\pm\beta_i},z)\exp\(\pm \Phi(\eta_\ell(\beta_i,z),Y_{\sqrt\ell Q_L}) \),
\end{align*}
and the quantum Yang-Baxter operator $S_{\sqrt\ell Q_L}^{\eta_\ell}(z)$ uniquely determined by $(i,j\in I)$:
\begin{align*}
  &S_{\sqrt\ell Q_L}^{\eta_\ell}(z)(\beta_j\ot \beta_i)=\beta_j\ot \beta_i+\vac\ot\vac\ot \pdiff{z}{2} \log f(z)^{[a_{ij}]_{q^{r_i}}[r\ell/r_j]_{q^{r_j}}(q^{r\ell}-q^{-r\ell}) },\\
  &S_{\sqrt\ell Q_L}^{\eta_\ell}(z)(e_{\pm\beta_j}\ot \beta_i)=e_{\pm\beta_j}\ot \beta_i\pm e_{\pm\beta_j}\ot\vac\ot
  \pd{z}\log f(z)^{ [a_{ij}]_{q^{r_i}}[r\ell/r_j]_{q^{r_j}}(q^{r\ell}-q^{-r\ell}) },\\
  &S_{\sqrt\ell Q_L}^{\eta_\ell}(z)(\beta_j\ot e_{\pm\beta_i})=\beta_j\ot e_{\pm\beta_i}\mp\vac\ot e_{\pm\beta_i}\ot
  \pd{z}\log f(z)^{ [a_{ij}]_{q^{r_i}}[r\ell/r_j]_{q^{r_j}}(q^{r\ell}-q^{-r\ell}) },\\
  &S_{\sqrt\ell Q_L}^{\eta_\ell}(z)(e_{\epsilon_1\beta_j}\ot e_{\epsilon_2\beta_i})
  =e_{\epsilon_1\beta_j}\ot e_{\epsilon_2\beta_i}\ot f(z)^{\epsilon_1\epsilon_2[a_{ij}]_{q^{r_i}} [r\ell/r_j]_{q^{r_j}}(q^{-r\ell}-q^{r\ell}) }.
\end{align*}
We denote this $\hbar$-adic quantum VA by $V_{\sqrt\ell Q_L}^{\eta_\ell}$.
\end{thm}

To construct the embedding $V_{\sqrt{\ell} Q_L}^{\eta_\ell} \hookrightarrow L_{\hat{\mathfrak{g}},\hbar}^\ell$,
we first introduce the following special case of \cite[Definition 4.13]{K-q-lattice-va}.
\begin{de}
Define $\mathcal A_\hbar^{\eta_\ell}({\sqrt\ell Q_L})$ to be the category, where the objects are topologically free $\C[[\hbar]]$-modules $W$ equipped with fields $\al_{i,\hbar}(z),\,e_{i,\hbar}^\pm(z)\in\E_\hbar(W)$ ($i\in I$), satisfying the conditions that ($i,j\in I$):
\begin{align}
  \tag{A1}&[\beta_{i,\hbar}(z_1),\beta_{j,\hbar}(z_2)]
    =[a_{ij}]_{q^{r_i\pd{z_2}}}[r\ell/r_j]_{q^{r_j\pd{z_2}}}\label{A1}\\
  &\quad\times\(\iota_{z_1,z_2}q^{-r\ell\pd{z_2}}-\iota_{z_2,z_1}q^{r\ell\pd{z_2}}\)
  \pd{z_1}\pd{z_2}\log f(z_1-z_2),\nonumber\\
  \tag{A2}&[\beta_{i,\hbar}(z_1),e_{j,\hbar}^\pm(z_2)]=\pm e_{j,\hbar}^\pm(z_2)
    [a_{ij}]_{q^{r_i\pd{z_2}}}[r\ell/r_j]_{q^{r_j\pd{z_2}}}\label{A2}\\
  &\quad\times\(\iota_{z_1,z_2}q^{-r\ell\pd{z_2}}-\iota_{z_2,z_1}q^{r\ell\pd{z_2}}\)
    \pd{z_1}\log f(z_1-z_2),\nonumber\\
  \tag{A3}&\iota_{z_1,z_2}f(z_1-z_2)^{q^{-r_ia_{ij}}[r\ell/r_i]_{q^{r_i}}[r\ell/r_j]_{q^{r_j}}}
    e_{i,\hbar}^\pm(z_1)e_{j,\hbar}^\pm(z_2)\label{A3}\\
  &\quad=\iota_{z_2,z_1}f(z_1-z_2)^{q^{r_ia_{ij}}[r\ell/r_i]_{q^{r_i}}[r\ell/r_j]_{q^{r_j}}}
    e_{j,\hbar}^\pm(z_2)e_{i,\hbar}^\pm(z_1),\nonumber\\
  \tag{A4}&\iota_{z_1,z_2}f(z_1-z_2)^{(q^{r_ia_{ij}}+\delta_{ij}+\delta_{ij}q^{2r\ell})
    [r\ell/r_i]_{q^{r_i}}[r\ell/r_j]_{q^{r_j}}}
    e_{i,\hbar}^\pm(z_1)e_{j,\hbar}^\mp(z_2)\label{A4}\\
  &\quad=\iota_{z_2,z_1}f(z_1-z_2)^{(q^{-r_ia_{ij}}+\delta_{ij}+\delta_{ij}q^{2r\ell})
        [r\ell/r_i]_{q^{r_i}}[r\ell/r_j]_{q^{r_j}}}
    e_{j,\hbar}^\mp(z_2)e_{i,\hbar}^\pm(z_1),\nonumber\\
  \tag{A5}&\frac{d}{dz}e_{i,\hbar}^\pm(z)=\pm \beta_{i,\hbar}(z)^+e_{i,\hbar}^\pm(z)
    \pm e_{i,\hbar}^\pm(z)\beta_{i,\hbar}(z)^-
    %-\frac{e_{i,\hbar}^\pm(z)}{q^{r_i}-q^{-r_i}}
    \label{A5}\\
  &\quad-e_{i,\hbar}^\pm(z)\(\(f_0'(z)/f_0(z)\)^{[2]_{q^{r_i}}[r\ell/r_i]_{q^{r_i}}q^{r\ell}}\)|_{z=0}
    ,\nonumber\\
  \tag{A6}&\iota_{z_1,z_2}f(z_1-z_2)^{q^{r\ell}[2]_{q^{r_i}}[r\ell/r_i]_{q^{r_i}}}
    e_{i,\hbar}^+(z_1)e_{i,\hbar}^-(z_2)\label{A6}\\
  &\quad=\iota_{z_2,z_1}f(z_1-z_2)^{q^{-r\ell}[2]_{q^{r_i}}[r\ell/r_i]_{q^{r_i}}}
    e_{i,\hbar}^-(z_2)e_{i,\hbar}^+(z_1),\nonumber\\
  \tag{A7}&\left.\(\iota_{z_1,z_2}f(z_1-z_2)^{q^{r\ell}[2]_{q^{r_i}}[r\ell/r_i]_{q^{r_i}}}
    e_{i,\hbar}^+(z_1)e_{i,\hbar}^-(z_2)\)\right|_{z_2=z_1}=1.\label{A7}
\end{align}
%The morphisms between two objects $W_1$ and $W_2$
%%\begin{align*}
%%(W_1,\{\beta_{i,\hbar}(z)\}_{i\in I},\{e_{i,\hbar}^\pm(z)\}_{i\in I}),\quad
%%(W_2,\{\beta_{i,\hbar}(z)\}_{i\in I},\{e_{i,\hbar}^\pm(z)\}_{i\in I})
%%\end{align*}
%are $\C[[\hbar]]$-module maps $f:W_1\to W_2$ such that
%\begin{align*}
%    \al_{i,\hbar}(z)\circ f=f\circ \al_{i,\hbar}(z),\quad e_{i,\hbar}^\pm(z)\circ f=f\circ e_{i,\hbar}^\pm(z)\quad \te{for }i\in I.
%\end{align*}
\end{de}

The following is a special case of \cite[Proposition 4.14]{K-q-lattice-va}.

\begin{prop}\label{prop:qlatticeVA-mod-to-A-mod}
%There exists a functor $\mathfrak I$ from the category of $V_{\sqrt\ell Q_L}^{\eta_\ell}$-modules to the category $\mathcal A_\hbar^{\eta_\ell}({\sqrt\ell Q_L})$.
%More precisely, let $(W,Y_W^{\eta_\ell})$ be a $V_{\sqrt\ell Q_L}^{\eta_\ell}$-module. Then
%\begin{align*}
%    \(W,\{Y_W^{\eta_\ell}(\beta_i,z)\}_{i\in I},\{Y_W^{\eta_\ell}(e_{\pm \al_i},z)\}_{i\in I}\)\in \obj \mathcal A_\hbar^{\eta_\ell}({\sqrt\ell Q_L}).
%\end{align*}
%Moreover, let $W_1$ be another $V_{\sqrt\ell Q_L}^{\eta_\ell}$-module, and let $f:W\to W_1$ be a $V_{\sqrt\ell Q_L}^{\eta_\ell}$-module.
%Then $f$ is also a morphism in $\mathcal A_\hbar^{\eta_\ell}({\sqrt\ell Q_L})$.
Let $(W,Y_W^{\eta_\ell})$ be a
$V_{\sqrt\ell Q_L}^{\eta_\ell}$-module. Then
\begin{align*}
  (W,\{Y_W^{\eta_\ell}(\beta_i,z)\}_{i\in I},\{Y_W^{\eta_\ell}(e_{\pm\beta_i},z)\}_{i\in I})\in\obj \mathcal A_\hbar^{\eta_\ell}({\sqrt\ell Q_L}).
\end{align*}
\end{prop}

\begin{proof}
Combining Lemma \ref{lem:eta-basic} and \cite[Proposition 4.14]{K-q-lattice-va}, we get that the fields $Y_W^{\eta_\ell}(\beta_i,z)$,
$Y_W^{\eta_\ell}(e_{\pm\beta_i},z)$ ($i\in I$) satisfy the relations \eqref{A1}, \eqref{A2}, \eqref{A5}, \eqref{A6}, \eqref{A7} and the following relations:
\begin{align}
  &\tag{A3$'$}\label{A3'}\iota_{z_1,z_2}f(z_1-z_2)^{-q^{r\ell}[a_{ij}]_{q^{r_i}}[r\ell/r_j]_{q^{r_j}}}
  Y_W^{\eta_\ell}(e_{\pm\beta_i},z_1)Y_W^{\eta_\ell}(e_{\pm\beta_j},z_2)\\
  =&\nonumber \iota_{z_2,z_1}f(-z_2+z_1)^{-q^{-r\ell}[a_{ij}]_{q^{r_i}}[r\ell/r_j]_{q^{r_j}}}
  Y_W^{\eta_\ell}(e_{\pm\beta_j},z_2)Y_W^{\eta_\ell}(e_{\pm\beta_i},z_1),\\
%%%%%%%%%%%%%%%%%%%%%%%%%%%%%
  &\tag{A4$'$}\label{A4'}\iota_{z_1,z_2}f(z_1-z_2)^{q^{r\ell}[a_{ij}]_{q^{r_i}}[r\ell/r_j]_{q^{r_j}}}
  Y_W^{\eta_\ell}(e_{\pm\beta_i},z_1)Y_W^{\eta_\ell}(e_{\mp\beta_j},z_2)\\
  =&\nonumber \iota_{z_2,z_1}f(-z_2+z_1)^{q^{-r\ell}[a_{ij}]_{q^{r_i}}[r\ell/r_j]_{q^{r_j}}}
  Y_W^{\eta_\ell}(e_{\mp\beta_j},z_2)Y_W^{\eta_\ell}(e_{\pm\beta_i},z_1).
\end{align}
For $i,j\in I$, we let
\begin{align*}
  F_{ij}(z)=f(z)^{[r\ell/r_i+a_{ij}]_{q^{r_i}} [r\ell/r_j]_{q^{r_j}} },\quad
  G_{ij}(z)=f(z)^{ [r\ell/r_i-a_{ij}]_{q^{r_i}}[r\ell/r_j]_{q^{r_j}}+\delta_{ij}(1+q^{2r\ell})[r\ell/r_i]_{q^{r_i}}^2 }.
\end{align*}

From \eqref{eq:sym}, we have that
\begin{align*}
  a_{ij}<0 \Longrightarrow a_{ij}=\begin{cases}
           -r,&\mbox{if }i\in I_S,\,j\in I_L,\\
           -1,&\mbox{otherwise}.
         \end{cases}
\end{align*}
Since $r_i=1$ for all $i\in I_S$ (see \eqref{eq:def-r-i}) and $\ell\in\Z_+$, we have that $r\ell/r_i+a_{ij}\ge 0$ for any $i,j\in I$.
It implies that $F_{ij}(z)\in\C[[z,\hbar]]$.
On the other hand, the only possibility for $r\ell/r_i-a_{ij}$ to be negative is that
\begin{align*}
  i=j,\quad r=r_i,\quad\te{and}\quad\ell=1.
\end{align*}
However, in this case, we have that
\begin{align*}
  G_{ij}(z)=f(z)^{-1+1+q^{2r}}=f(z)^{q^{2r}}\in\C[[z,\hbar]].
\end{align*}
Therefore, we also have that $G_{ij}(z)\in\C[[z,\hbar]]$ for all $i,j\in I$.

Notice that
\begin{align*}
  &[r\ell/r_i\pm a_{ij}]_{q^{r_i}}[r\ell/r_j]_{q^{r_j}}\mp q^{r\ell}[a_{ij}]_{q^{r_i}}[r\ell/r_j]_{q^{r_j}}\\
  =&\frac{1}{q^{r_i}-q^{-r_i}}[r\ell/r_j]_{q^{r_j}}
  \(q^{r\ell\pm r_ia_{ij}}-q^{-r\ell\mp r_ia_{ij}}-q^{r\ell\pm r_ia_{ij}}+q^{r\ell\mp r_ia_{ij}}\)\\
  =&q^{\mp r_ia_{ij}}[r\ell/r_i]_{q^{r_i}}[r\ell/r_j]_{q^{r_j}},\\
  &[r\ell/r_i\pm a_{ij}]_{q^{r_i}}[r\ell/r_j]_{q^{r_j}}\mp q^{-r\ell}[a_{ji}]_{q^{r_j}}[r\ell/r_i]_{q^{r_i}}\\
  =&\frac{q^{r\ell}-q^{-r\ell}}{(q^{r_i}-q^{-r_i})(q^{r_j}-q^{-r_j})}
    \(q^{r\ell\pm r_ia_{ij}}-q^{-r\ell\mp r_ia_{ij}}-q^{-r\ell\pm r_ja_{ji}}+q^{-r\ell\mp r_ja_{ji}} \)\\
  =&q^{\pm r_ia_{ij}}[r\ell/r_i]_{q^{r_i}}[r\ell/r_j]_{q^{r_j}}.
\end{align*}
Then by multiplying $F_{ij}(z_1-z_2)$ on both hand sides of \eqref{A3'}, we get \eqref{A3},
and by multiplying $G_{ij}(z_1-z_2)$ on both hand sides of \eqref{A4'}, we get \eqref{A4}.
\end{proof}

%By using \cite[Theorem 4.15]{K-q-lattice-va}, we get that:
%
%\begin{thm}\label{thm:Undeform-qlattice}
%There exists a functor $\mathfrak U_{\eta_\ell}$ from the category $\mathcal A_\hbar^{\eta_\ell}({\sqrt\ell Q_L})$ to the category of $V_{\sqrt\ell Q_L}[[\hbar]]$-modules.
%Moreover, the three functors $\mathfrak U_{\eta_\ell}$, $\mathfrak I$, $\mathfrak D_{\eta_\ell}$ are category isomorphisms.
%\end{thm}

From \cite[Corollary 4.18]{K-q-lattice-va}, we immediately get the following result.

\begin{lem}\label{lem:qlattice-universal-property}
Let $(V,Y,\vac_V)$ be an $\hbar$-adic nonlocal VA containing a subset
$\set{\beta_i,e_i^\pm}{i\in I}$, such that
\begin{align*}
  (V,\{Y(\beta_i,z)\}_{i\in I},\{Y(e_i^\pm,z)\}_{i\in I})\in \obj\mathcal A_\hbar^{\eta_\ell}({\sqrt\ell Q_L}).
\end{align*}
Then there is a unique $\hbar$-adic nonlocal VA homomorphism
$\psi:V_{\sqrt\ell Q_L}^{\eta_\ell}\to V$, such that
\begin{align*}
  \psi(\beta_i)=\beta_i,\quad \psi(e_{\pm\beta_i})=e_i^\pm\quad\te{for }i\in I.
\end{align*}
Moreover, $\psi$ is injective.
\end{lem}

The following result is an immediate consequence of \cite[Corollary 4.19]{K-q-lattice-va}.

\begin{lem}\label{lem:h--to-e}
Let $(W,Y_W^{\eta_\ell})$ be a $V_{\sqrt\ell Q_L}^{\eta_\ell}$-module.
If there exists $w\in W$ such that
\begin{align*}
  Y_W^{\eta_\ell}(\beta_i,z)^-w=0\quad\te{for }i\in I,
\end{align*}
then
\begin{align*}
  &Y_W^{\eta_\ell}(e_{\pm\beta_i},z)w\in W[[z]]\quad\te{for }i\in I.
\end{align*}
\end{lem}

The following is the main result of this section, which is a quantum analogue of \cite[Proposition 4.1]{DW-para-structure-double-comm}.

\begin{thm}\label{thm:qlattice-inj}
There exists a unique $\hbar$-adic quantum VA embedding from $V_{\sqrt\ell Q_L}^{\eta_\ell}$
to $L_{\hat\g,\hbar}^\ell$, such that
\begin{align}\label{eq:qlattice-embedding}
  &\sqrt\ell r/r_i\al_i\mapsto [r_i]_{q^\partial}\inv h_i,\quad
  e_{\pm\sqrt\ell r/r_i\al_i}\mapsto \sqrt{c_{i,\ell}} q^{(-r\ell+r_i)\partial}\(x_i^\pm\)_{-1}^{r\ell/r_i}\vac
\end{align}
where
\begin{align}
  c_{i,\ell}=\frac{f_0(2r\ell \hbar)^\half f_0(2(r\ell-r_i) \hbar)^{-\half}}{([r\ell/r_i]_{q^{r_i}}!)^2 f_0(2r_i\hbar)^{r\ell/r_i-\half} }.
\end{align}
Here $f_0(z)\in 1+z^2\C[[z^2]]$ is defined in \eqref{eq:def-f-0} and for any $k\in\C$, $f_0(z)^k\in \C[[z]]$ denotes the standard expansion.
\end{thm}

The proof of Theorem \ref{thm:qlattice-inj} for $\g=\ssl_2$ is presented in Section \ref{subsec:sl2-case}.
Assuming this result holds, we prove Theorem \ref{thm:qlattice-inj} for general $\g$ in the remainder of this section.
To enhance readability and convenience, we define
\begin{align}
  &f_i^\pm=\sqrt{c_{i,\ell}} q^{(r_i-r\ell)\partial}\(x_i^\pm\)_{-1}^{r\ell/r_i}\vac
  \quad\te{for }i\in I.
\end{align}
To apply Lemma \ref{lem:qlattice-universal-property}, we must demonstrate that $L_{\hat\g,\hbar}^\ell$ equipped with the fields $Y_{\wh\ell}([r_i]_{q^\partial}\inv h_i,z)$, $Y_{\wh\ell}(f_i^\pm,z)$ ($i\in I$) is an object of $\mathcal A_\hbar^{\eta_\ell}(\sqrt\ell Q_L)$.
We first prove that it belongs to the following category containing $\mathcal A_\hbar^{\eta_\ell}(\sqrt\ell Q_L)$ as a full subcategory.

\begin{de}
Let $\wh{\mathcal A}_\hbar^{\eta_\ell}(\sqrt\ell Q_L)$ be the category whose objects are topologically free $\C[[\hbar]]$-modules $W$, equipped with fields $\beta_{i,\hbar}(z),e_{i,\hbar}^\pm(z)\in\E_\hbar(W)$ ($i\in I$) satisfying the relations
\eqref{A1}, \eqref{A2}, \eqref{A3} and \eqref{A4}.
\end{de}

\begin{lem}\label{lem:AQ-sp1-4}
Let $X=F,V,L$. Then $X_{\hat\g,\hbar}^\ell$ equipped with fields $Y_{\wh\ell}([r_i]_{q^\partial}\inv h_i,z)$, $Y_{\wh\ell}(f_i^\pm,z)$ ($i\in I$) is an object of $\wh{\mathcal A}_\hbar^{\eta_\ell}(\sqrt\ell Q_L)$.
\end{lem}

\begin{proof}
The equation \eqref{A1} follows immediate from the equation \eqref{eq:local-h-1}.
According to Proposition \ref{prop:normal-ordering-rel-general},
\begin{align}
  Y_{\wh\ell}(f_i^\pm&,z)
  \label{eq:pf-AQ-sp-temp1}
  =\sqrt{c_{i,\ell}}\prod_{a=1}^{r\ell/r_i}f_0(2ar_i\hbar)
  \:Y_{\wh\ell}(x_i^\pm,z+(r\ell-r_i)\hbar)\\
  &\times Y_{\wh\ell}(x_i^\pm,z+(r\ell-2r_i)\hbar)\cdots
    Y_{\wh\ell}(x_i^\pm,z-(r\ell-r_i)\hbar)\;.\nonumber
\end{align}
Then from \eqref{eq:local-h-2}, we have that
\begin{align*}
  &[Y_{\wh\ell}([r_i]_{q^\partial}\inv h_i,z_1),Y_{\wh\ell}(f_j^\pm,z_2)]=
  \sqrt{c_{j,\ell}}\prod_{a=1}^{r\ell/r_j}f_0(2ar_j\hbar)\\
  \times&[r_i]_{q^\pd{z_1}}\inv \bigg[Y_{\wh\ell}(h_i,z_1),
  \:Y_{\wh\ell}(x_j^\pm,z_2+(r\ell-r_j)\hbar)
  %Y_{\wh\ell}(x_{j,\hbar}^\pm,z_2+(r\ell-2r_j)\hbar)
  \cdots
    Y_{\wh\ell}(x_j^\pm,z_2-(r\ell-r_j)\hbar)\;\bigg]\\
  =&\pm \sqrt{c_{j,\ell}}\prod_{a=1}^{r\ell/r_j}f_0(2ar_j\hbar)
  \:Y_{\wh\ell}(x_j^\pm,z_2+(r\ell-r_j)\hbar)\\
  &\times Y_{\wh\ell}(x_j^\pm,z_2+(r\ell-2r_j)\hbar)\cdots
    Y_{\wh\ell}(x_j^\pm,z_2-(r\ell-r_j)\hbar)\;\\
  &\times [a_{ij}]_{q^{r_i\pd{z_2}}}[r\ell/r_j]_{q^{r_j\pd{z_2}}}
  \(\iota_{z_1,z_2}q^{-r\ell\pd{z_2}}-\iota_{z_2,z_1}q^{r\ell\pd{z_2}}\)
  \pd{z_1}\log f(z_1-z_2)\\
  =&\pm Y_{\wh\ell}(f_j^\pm,z_2)
    [a_{ij}]_{q^{r_i\pd{z_2}}}[r\ell/r_j]_{q^{r_j\pd{z_2}}}
    \(\iota_{z_1,z_2}q^{-r\ell\pd{z_2}}-\iota_{z_2,z_1}q^{r\ell\pd{z_2}}\)
    \pd{z_1}\log f(z_1-z_2),
\end{align*}
which proves \eqref{A2}.
Finally, \eqref{A3} and \eqref{A4} can be derived from \eqref{eq:pf-AQ-sp-temp1}, \eqref{eq:local-h-3},
\eqref{eq:local-h-4} and Proposition \ref{prop:normal-ordering-rel-general}.
\end{proof}

The proof of Theorem \ref{thm:qlattice-inj} for a general $\g$ is now ready to be completed.

\vspace{2mm}

\noindent\emph{Proof of Theorem \ref{thm:qlattice-inj}:}
Proposition \ref{prop:universal-qaff} yields a $\C[[\hbar]]$-linear quantum VA homomorphism
\begin{align*}
  \varsigma_i:L_{\hat\ssl_2,r_i\hbar}^{r\ell/r_i}
  \to L_{\hat\g,\hbar}^\ell.
\end{align*}
By utilizing Theorem \ref{thm:classical-limit-L}, the map $\varsigma_i$ induces the following $\C$-linear map
\begin{align*}
  \bar \varsigma_i: L_{\hat\ssl_2}^{r\ell/r_i}\cong L_{\hat\ssl_2,r_i\hbar}^{r\ell/r_i}/\hbar L_{\hat\ssl_2,r_i\hbar}^{r\ell/r_i}
  \to L_{\hat\g,\hbar}^\ell/\hbar L_{\hat\g,\hbar}^\ell\cong L_{\hat\g}^\ell.
\end{align*}
From Remark \ref{rem:aff-vas}, we have that $\bar\varsigma_i$ is a $\hat\ssl_2$-module map and $L_{\hat\ssl_2}^{r\ell/r_i}$ is a simple $\hat\ssl_2$-module.
Then $\bar\varsigma_i$ must be injective, since $\bar\varsigma_i(\vac)=\vac\ne 0$.
Applying Lemma \ref{lem:topo-free-inj-surj}, the map $\varsigma_i$ is also injective.
Let $\{h_1,x_1^\pm\}$ be the standard generators of $L_{\hat\ssl_2,r_i\hbar}^{r\ell/r_i}$,
and let $\Z\al_1$ be the root lattice of $\ssl_2$.
Considering the previously mentioned assumption that Theorem \ref{thm:qlattice-inj} holds true for $\g=\ssl_2$,
we can infer from Proposition \ref{prop:qlatticeVA-mod-to-A-mod} that
\begin{align*}
  &\(L_{\hat\ssl_2,r_i\hbar}^{r\ell/r_i}, Y_{\wh{r\ell/r_i}}(h_1,z),
  Y_{\wh{r\ell/r_i}}(f_1^\pm,z)\)\in\obj \mathcal A_{r_i\hbar}^{\eta_{r\ell/r_i}}(\sqrt{r\ell/r_i}\Z\al_1).
\end{align*}
As $\varsigma_i$ is an injective $\C[[\hbar]]$-linear quantum VA homomorphism,
the fields $Y_{\wh\ell}(h_i,z)$, $Y_{\wh\ell}(f_i^\pm,z)$ satisfy the relations \eqref{A5}, \eqref{A6} and \eqref{A7} on $L_{\hat\g,\hbar}^\ell$.
By combining Lemma \ref{lem:AQ-sp1-4}, we find that
\begin{align*}
  \(L_{\hat\g,\hbar}^\ell,\{Y_{\wh\ell}(h_i,z)\}_{i\in I},\{Y_{\wh\ell}(f_i^\pm,z)\}_{i\in I}\)\in \obj \mathcal A_\hbar^{\eta_\ell}(\sqrt\ell Q_L).
\end{align*}
Thus, by utilizing Lemma \ref{lem:qlattice-universal-property}, we complete the proof.

\section{Quantum parafermion vertex algebras}\label{sec:qpara}

In this section, we study the structure of quantum parafermion VA, that is, the commutant of $\h$ in a simple quantum affine VA $L_{\hat\g,\hbar}^\ell$.
We first present some basic results about commutant of $\hbar$-adic nonlocal VAs.
Then recall some fundamental results of parafermion VAs.
Our central objective is to prove quantum analogues of these classical results.

\begin{de}
Let $V$ be an ($\hbar$-adic) nonlocal VA, and let $U\subset V$ be a subset of $V$.
Define
\begin{align*}
  &\Com_V(U)=\set{v\in V}{[Y(u,z_1),Y(v,z_2)]=0\quad\te{for any }u\in U.}.
\end{align*}
\end{de}

\begin{rem}\label{rem:com-alt-def}
  For any $u\in U$ and $v\in \Com_V(U)$, we have that
  \begin{align*}
    Y(u,z_1)Y(v,z_2)=Y(v,z_2)Y(u,z_1)\in\E_\hbar^{(2)}(V).
  \end{align*}
  Then
  \begin{align*}
    Y(Y(u,z_0)v,z_2)=\(Y(u,z_1)Y(v,z_2)\)|_{z_1=z_2+z_0}\in \End(V)[[z_0,z_2,z_2\inv]].
  \end{align*}
  By applying on $\vac$ and taking $z_2\to 0$, we get that
  \begin{align*}
    Y(u,z_0)v\in V[[z_0]].
  \end{align*}
  Therefore,
  \begin{align}\label{eq:com-alt-rel}
    \Com_V(U)\subset\set{v\in V}{Y(u,z)^-v=0\quad\te{for any }u\in U}.
  \end{align}
\end{rem}

It is straightforward to verify the following result.
\begin{lem}\label{lem:Com-basic1}
Let $V$ be an $\hbar$-adic nonlocal VA, and let $U\subset V$ be a subset of $V$.
Then
\begin{align*}
  &\Com_V(\bar U)=\Com_V(U),\quad \Com_V([U])=\Com_V(U),\\
  &\overline{\Com_V(U)}=\Com_V(U),\quad \left[\Com_V(U)\right]=\Com_V(U).
\end{align*}
Moreover, $\Com_V(U)$ is an $\hbar$-adic nonlocal subVA of $V$.
\end{lem}

\begin{lem}\label{lem:Com-basic2}
Let $V$ be an $\hbar$-adic nonlocal VA, and let $U_1,U_2\subset V$ be closed $\C[[\hbar]]$-submodules of $V$, such that
$[U_i]=U_i$ for $i=1,2$.
Suppose further that $U_2\subset\Com_V(U_1)$
and
\begin{align*}
  U_2/\hbar U_2=\Com_{V/\hbar V}(U_1/\hbar U_1).
\end{align*}
Then
\begin{align*}
  U_2=\Com_V(U_1).
\end{align*}
\end{lem}

\begin{proof}
Notice that
\begin{align*}
  \Com_V(U_1)/\hbar \Com_V(U_1)\subset
  \Com_{V/\hbar V}(U_1/\hbar U_1).
\end{align*}
Then
\begin{equation*}
  \begin{tikzcd}
    U_2/\hbar U_2
    \arrow[d, phantom, sloped, "\subset"]
    \arrow[r, equal]
    &\Com_{V/\hbar V}(U_1/\hbar U_1)\\
    \Com_V(U_1)/\hbar \Com_V(U_1)
    \arrow[ur, phantom, sloped, "\subset"]&
  \end{tikzcd}
\end{equation*}
It implies that
\begin{align*}
  U_2/\hbar U_2= \Com_V(U_1)/\hbar \Com_V(U_1).
\end{align*}
It follows from Lemma \ref{lem:Com-basic1} that all the $\C[[\hbar]]$-submodules are topologically free.
Therefore, the injection $U_2\to \Com_V(U_1)$ must be an isomorphism.
\end{proof}

We identify $\h$ with a subspace of $L_{\hat\g}^\ell$ through the embedding that maps the simple coroots $h_i$ and the generators $h_i\in L_{\hat\g}^\ell$ (see Definition \ref{de:affVAs}).
For $\ell\in\Z_+$, we denote by
\begin{align*}
  K_{\hat\g}^\ell=\Com_{L_{\hat\g}^\ell}(\h)
\end{align*}
the parafermion VA (see \cite[Section 4]{DW-para-structure-gen}).
The following result was proved in \cite{DW-para-structure-gen} (see \cite{DR-para-rational}).
\begin{thm}\label{thm:para-classical-gen}
For each $i\in I$, define
\begin{align*}
  W_i^2=&(x_i^+)_{-1}x_i^--\frac{1}{2r_i}\partial h_i-\frac{1}{2r_ir\ell}(h_i)_{-1}h_i,\\
  W_i^3=&(x_i^+)_{-2}x_i^--(x_i^+)_{-1}\partial x_i^--\frac{2}{r\ell}(h_i)_{-1}(x_i^+)_{-1}x_i^-\\
  &+\frac{1}{6r_i}\partial^2h_i+\frac{1}{r_ir\ell}(h_i)_{-1}\partial h_i
  +\frac{2}{3r_ir^2\ell^2}(h_i)_{-1}^2h_i.\nonumber
\end{align*}
Then $K_{\hat\g}^\ell$ is generated by the set
\begin{align}
  \set{W_i^2,\,W_i^3}{i\in I}.
\end{align}
\end{thm}

The following double commutant property
was proved in \cite{DW-para-structure-double-comm}.
\begin{thm}\label{thm:classical-double-commutant}
$L_{\hat\g}^\ell$ has a subVA isomorphic to $V_{\sqrt\ell Q_L}$, and
\begin{align}
    \Com_{L_{\hat\g}^\ell}\(V_{\sqrt\ell Q_L}\)=K_{\hat\g}^\ell\quad\te{and}\quad
    \Com_{L_{\hat\g}^\ell}\(K_{\hat\g}^\ell\)=V_{\sqrt\ell Q_L}.
\end{align}
\end{thm}

The following result gives the $\hbar$-adic analogues of $W_i^2$ and $W_i^3$ in quantum affine VAs,
whose proof will be given in Section \ref{subsec:pf-prop-W}.

\begin{prop}\label{prop:W}
Let $X=V$ or $L$.
For each $i\in I$ and $\ell\in\C^\times$, we define
\begin{align}
  W_i(z)=&\exp\(\(\frac{1-e^{z\partial}}{\partial [r\ell]_{q^{\partial}}} h_i\)_{-1}\)Y_{\wh\ell}(x_i^+,z)x_i^--\frac{\vac}{(q^{r_i}-q^{-r_i})z}\\
  &
    +\frac{f_0(2(r_i+r\ell)\hbar)^\half
        f_0(2(r_i-r\ell)\hbar)^{-\half}
    \vac}{(q^{r_i}-q^{-r_i})(z+2r\ell\hbar)}.\nonumber
\end{align}
Then
\begin{align}
  W_i(z)\in X_{\hat\g,\hbar}^\ell[[z]],\label{eq:prop-W-well-defined}
\end{align}
and
\begin{align}
  Y_{\wh\ell}(h_i,z_1)^-W_j(z)=0\quad\te{in }X_{\hat\g,\hbar}^\ell.\label{eq:prop-W-parafermion}
\end{align}
Moreover, for another $\ell'\in\C$, we have that
\begin{align}
  S_{\ell,\ell'}(z_1)(W_i(z)\ot u)=W_i(z)\ot u\quad\te{for any }u\in X_{\hat\g,\hbar}^{\ell'},
  \label{eq:prop-W-S-invariant}
\end{align}
where $S_{\ell,\ell'}(z)$ is defined in Theorem \ref{thm:quotient-algs}.
\end{prop}

We identify $\h$ with a subspace of $L_{\hat\g,\hbar}^\ell$ through the embedding that maps the simple coroots $h_i$ and the generators $h_i\in L_{\hat\g,\hbar}^\ell$ (see Definition \ref{de:L-tau}).
By utilizing Lemma \ref{lem:Com-basic1}, we define the following $\hbar$-adic nonlocal subVAs of $L_{\hat\g,\hbar}^\ell$.
\begin{de}
Define
\begin{align}
  K_{\hat\g,\hbar}^\ell=\Com_{L_{\hat\g,\hbar}^\ell}(\h).
\end{align}
\end{de}

The following result show that $K_{\hat\g,\hbar}^\ell$ is a quantization of the parafermion VA $K_{\hat\g}^\ell$, by presenting a generating subset for $K_{\hat\g,\hbar}^\ell$--a quantum analogue of Theorem \ref{thm:para-classical-gen}.
%With the help of this, we show that $K_{\hat\g,\hbar}^\ell$ is a quantization of the parafermion VA $K_{\hat\g}^\ell$,
%that is, $K_{\hat\g,\hbar}^\ell$ is an $\hbar$-adic quantum VA and the classical limit $K_{\hat\g,\hbar}^\ell/\hbar K_{\hat\g,\hbar}^\ell\cong K_{\hat\g}^\ell$.
%Notably, we find that the quantum Yang-Baxter operator of $K_{\hat\g,\hbar}^\ell$ is trivial.
%Consequently, this $\hbar$-adic quantum VA reduces to an $\hbar$-adic VA.

\begin{thm}\label{thm:K-space-desc}
$K_{\hat\g,\hbar}^\ell$ is generated by the set
\begin{align}\label{eq:para-gen}
  \set{W_i(0),\,W_i(-2r\ell\hbar)}{i\in I}.
\end{align}
Moreover,
\begin{align}\label{eq:para-alt}
  K_{\hat\g,\hbar}^\ell=\set{u\in L_{\hat\g,\hbar}^\ell}{Y_{\wh\ell}(h_i,z)^-u=0\quad\te{for any }i\in I},
\end{align}
and for any $\ell'\in\Z_+$, any $u\in K_{\hat\g,\hbar}^\ell$ and any $v\in L_{\hat\g,\hbar}^{\ell'}$, we have that
\begin{align}\label{eq:para-S-inv}
  S_{\ell,\ell'}(z)(u\ot v)=u\ot v.
\end{align}
Additionally, $K_{\hat\g,\hbar}^\ell$ is an $\hbar$-adic quantum subVA of $L_{\hat\g,\hbar}^\ell$ and further an $\hbar$-adic VA.
Furthermore, we have that
\begin{align}\label{eq:para-classical-limit}
  K_{\hat\g,\hbar}^\ell/\hbar K_{\hat\g,\hbar}^\ell\cong K_{\hat\g}^\ell.
\end{align}
\end{thm}

\begin{proof}
Let $K_1=\<\set{W_i(0),\,W_i(-2r\ell\hbar)}{i\in I}\>$ and let
$K_2$ be the right hand side of \eqref{eq:para-alt}.
Then we have that
\begin{align}\label{eq:K-space-desc-temp1}
  K_1\subset K_{\hat\g,\hbar}^\ell\subset K_2,
\end{align}
where the first ``$\subset$'' follows from Proposition \ref{prop:W} and the second ``$\subset$'' follows from Remark \ref{rem:com-alt-def}.
For each $i\in I$, we denote
\begin{align*}
  \iota\(W_i^2\)=W_i(0),\quad \iota\(W_i^3\)=(r\ell\hbar)\inv (W_i(0)-W_i(-2r\ell\hbar))-\partial W_i(0).
\end{align*}
Then $K_1=\<\{\iota(W_i^2),\,\iota(W_i^3)\,|\,i\in I\}\>$.
Let $\pi$ be the composition map
\begin{align*}
  \xymatrix{
    L_{\hat\g,\hbar}^\ell\ar[r]& L_{\hat\g,\hbar}^\ell/\hbar L_{\hat\g,\hbar}^\ell\ar[r]& L_{\hat\g}^\ell.
  }
\end{align*}
Then one can straightforwardly verify that
\begin{align}\label{eq:W-mod-h}
  &\pi(\iota(W_i^2))=W_i^2,\quad \pi(\iota(W_i^3))=W_i^3\quad\te{for }i\in I.
\end{align}
Note that
\begin{align}\label{eq:pi-K2}
  \pi(K_2)\subset&\set{v\in L_{\hat\g}^\ell}{Y(h_i,z)^-v=0\quad\te{for any }i\in I}\\
  =&\set{v\in L_{\hat\g}^\ell}{[Y(h_i,z_1),Y(v,z_2)]=0\quad\te{for any }i\in I}
  =\Com_{L_{\hat\g}^\ell}(\h)=K_{\hat\g}^\ell,\nonumber
\end{align}
where the first equation follows from the Borcherds commutator formula (see \cite[Remark 3.1.11]{LL}).
For each $u\in K_2\setminus\hbar L_{\hat\g,\hbar}^\ell$, we have that $\pi(u)\in K_{\hat\g}^\ell$.
By using Theorem \ref{thm:para-classical-gen}, we get that $\pi(u)$ can be written as a finite sum of the form
\begin{align*}
  v^{(1)}_{n_1}\cdots v^{(k)}_{n_k},\quad\te{where }k\in\N,\,v^{(a)}\in\{W_i^2,\,W_i^3\,|\,i\in I\},\,n_a\in\Z,\,(1\le a\le k).
\end{align*}
Denote by $u'$ the corresponding finite sum with $W_i^j$ being replaced by $\iota(W_i^j)$ for $i\in I$, $j=2,3$.
Then $u'\in K_1$. Moreover, we deduce from \eqref{eq:W-mod-h} that
$\pi(u)=\pi(u')$.

Now, we fix an arbitrary $u\in K_2$. Define
\begin{align*}
  N(u)=\sup\sett{n\in\N}{u-v\in\hbar^nL_{\hat\g,\hbar}^\ell\,\,\te{for some }v\in K_1}.
\end{align*}
Suppose that $N(u)<+\infty$.
Then there exists $v\in K_1$ such that
\begin{align*}
  u-v\in K_2\cap\hbar^{N(u)}L_{\hat\g,\hbar}^\ell.
\end{align*}
Let $w=\hbar^{-N(u)}(u-v)$.
Then $w\in K_2\setminus\hbar L_{\hat\g,\hbar}^\ell$, as $[K_2]=K_2$.
From the argument above, we can choose $w'\in K_1$, such that $\pi(w)=\pi(w')$.
That is, $w-w'\in\hbar L_{\hat\g,\hbar}^\ell$.
Define
\begin{align*}
  v'=v+\hbar^{N(u)}w'\in K_1.
\end{align*}
Then
\begin{align*}
  u-v'=u-v-\hbar^{N(u)}w'=\hbar^{N(u)}(w-w')\in\hbar^{N(u)+1}L_{\hat\g,\hbar}^\ell,
\end{align*}
which contradict to the definition of $N(u)$.
Hence, $N(u)$ must be $+\infty$.
It implies that for each $n\in\N$, there exists $v_n\in K_1$, such that $u-v_n\in\hbar^n L_{\hat\g,\hbar}^\ell$.
Since $K_1$ is closed in $L_{\hat\g,\hbar}^\ell$, we get that
\begin{align*}
  u=\lim_{n\to\infty}v_n\in K_1.
\end{align*}
Therefore, $K_2\subset K_1$. We complete the proof of the first statement and \eqref{eq:para-alt}.
Combining these with \eqref{eq:prop-W-S-invariant}, we get the relation \eqref{eq:para-S-inv}.
Consequently, $S_{\ell,\ell}(z)$ acts trivially on $K_{\hat\g,\hbar}^\ell$.
It implies that the $\hbar$-adic nonlocal subVA $K_{\hat\g,\hbar}^\ell$ is not only an $\hbar$-adic quantum subVA of $L_{\hat\g,\hbar}^\ell$ but also an $\hbar$-adic VA.

Finally, since $\pi(\iota(W_i^j))=W_i^j$ for $i\in I$, $j=2,3$, we get from Theorem \ref{thm:para-classical-gen}
that
\begin{align*}
  K_{\hat\g,\hbar}^\ell/\hbar K_{\hat\g,\hbar}^\ell=K_{\hat\g,\hbar}^\ell/(K_{\hat\g,\hbar}^\ell\cap \hbar L_{\hat\g,\hbar}^\ell)=\pi(K_1)\supset K_{\hat\g}^\ell,
\end{align*}
where the first equation follows from \cite[Lemma 3.5]{Li-h-adic} and the fact that $[K_{\hat\g,\hbar}^\ell]=K_{\hat\g,\hbar}^\ell$.
By utilizing relation \eqref{eq:pi-K2}, we get the reverse inclusion:
\begin{align*}
  K_{\hat\g,\hbar}^\ell/\hbar K_{\hat\g,\hbar}^\ell=\pi(K_2)\subset K_{\hat\g}^\ell.
\end{align*}
Therefore, we complete the proof of \eqref{eq:para-classical-limit}.
\end{proof}

The following result is the quantum analogue of Theorem \ref{thm:classical-double-commutant}.
\begin{thm}\label{thm:double-comutant}
View $V_{\sqrt\ell Q_L}^{\eta_\ell}$ as a subalgebra of $L_{\hat\g,\hbar}^\ell$ (see Theorem \ref{thm:qlattice-inj}), we have that
\begin{align}\label{eq:double-comutant}
  \Com_{L_{\hat\g,\hbar}^\ell}\(V_{\sqrt\ell Q_L}^{\eta_\ell}\)=K_{\hat\g,\hbar}^\ell,\quad
  \Com_{L_{\hat\g,\hbar}^\ell}\(K_{\hat\g,\hbar}^\ell\)=V_{\sqrt\ell Q_L}^{\eta_\ell}.
\end{align}
\end{thm}

\begin{proof}
Let $u\in K_{\hat\g,\hbar}^\ell$.
From \eqref{eq:para-alt}, we have that
\begin{align*}
  Y_{\wh\ell}(h_i,z)^-u=0\quad\te{for }i\in I.
\end{align*}
View $L_{\hat\g,\hbar}^\ell$ as a $V_{\sqrt\ell Q_L}^{\eta_\ell}$-module.
By utilizing Lemma \ref{lem:h--to-e}, we get that
\begin{align*}
  Y_{\wh\ell}(f_i^\pm,z)^-u=0\quad\te{for }i\in I.
\end{align*}
Combining this with \eqref{eq:para-S-inv}, \cite[Remark 3.1]{K-Coproduct-q-aff-va} and the fact that $V_{\sqrt\ell Q_L}^{\eta_\ell}$ is generated by the set $\set{h_i,f_i^\pm}{i\in I}$, we get that
\begin{align*}
  [Y_{\wh\ell}(u,z_1),Y_{\wh\ell}(v,z_2)]=Y_{\wh\ell}(Y_{\wh\ell}(u,z_1-z_2)^-v-Y_{\wh\ell}(u,-z_2+z_1)^-v,z_2)=0
\end{align*}
for any $u\in V_{\sqrt\ell Q_L}^{\eta_\ell}$ and $v\in K_{\hat\g,\hbar}^\ell$.
Consequently,
\begin{align*}
  \Com_{L_{\hat\g,\hbar}^\ell}\(V_{\sqrt\ell Q_L}^{\eta_\ell}\)\supset K_{\hat\g,\hbar}^\ell\quad\te{and}\quad
  \Com_{L_{\hat\g,\hbar}^\ell}\(K_{\hat\g,\hbar}^\ell\)
  \supset V_{\sqrt\ell Q_L}^{\eta_\ell}.
\end{align*}
Since
\begin{align*}
  L_{\hat\g,\hbar}^\ell/\hbar L_{\hat\g,\hbar}^\ell\cong L_{\hat\g}^\ell,\quad
  K_{\hat\g,\hbar}^\ell/\hbar K_{\hat\g,\hbar}^\ell\cong K_{\hat\g}^\ell,\quad
  V_{\sqrt\ell Q_L}^{\eta_\ell}/\hbar V_{\sqrt\ell Q_L}^{\eta_\ell}\cong V_{\sqrt\ell Q_L},
\end{align*}
we complete the proof by using Theorem \ref{thm:classical-double-commutant} and Lemma \ref{lem:Com-basic2}.
\end{proof}

\section{Proof of Theorem \ref{thm:qlattice-inj} for $\g=\ssl_2$}\label{subsec:sl2-case}

Let $\g=\ssl_2$.
Then $I=\{1\}$, $r=r_1=1$, $Q_L=\Z\al_1$ and the Cartan matrix $A=(2)$.
In this section, we prove Theorem \ref{thm:qlattice-inj} using induction on $\ell$. First, we prove that $L_{\hat\ssl_2,\hbar}^1\cong V_{\Z\al_1}^{\eta_1}$. Then we show that if the case for $\ell$ holds true, it implies that the case for $\ell+1$ also holds true by using the following $\hbar$-adic quantum VA injection (see Remark \ref{rem:inj}) provided in Theorem \ref{thm:coproduct}:
\begin{align}\label{eq:inj}
  \Delta:L_{\hat\ssl_2,\hbar}^{\ell+1}\hookrightarrow L_{\hat\ssl_2,\hbar}^\ell\wh\ot L_{\hat\ssl_2,\hbar}^1.
\end{align}

We need the following analogue of $E_\ell(h_i)$ (see \eqref{eq:def-E-h}) in quantum lattice VA:
\begin{align*}
    &E_\ell(\beta_1)=\( \frac{f_0(2\hbar+2\ell\hbar)}{f_0(2\hbar-2\ell\hbar)} \)^\half
    \exp\( \(-q^{-\ell\partial}2\hbar f_0(2\hbar\partial)\beta_1\)_{-1} \)\vac\in V_{\sqrt\ell \Z\al_1}^{\eta_\ell}.
\end{align*}
The following result rewrite $E_\ell(\beta_1)$ in term of classical lattice VA operators (see \eqref{eq:def-E}).
\begin{lem}\label{lem:E-beta}
We have that
\begin{align*}
    E_\ell(\beta_1)=E^+(\beta_1,-\hbar-\ell\hbar)
    E^+(-\beta_1,\hbar-\ell\hbar)\vac.
\end{align*}
\end{lem}

\begin{proof}
Set
\begin{align*}
    &\wt \beta_1^-(z)=-q^{-\ell\pd{z}}2\hbar f_0\(2\hbar\pd{z}\)\( Y_{\sqrt\ell \Z\al_1}(\beta_1,z)^-+\Phi(\eta_\ell'(\beta_1,z),Y_{\sqrt\ell \Z\al_1})\),\\
    &\wt \beta_1^+(z)=-q^{-\ell\pd{z}}2\hbar f_0\(2\hbar\pd{z}\) Y_{\sqrt\ell \Z\al_1}(\beta_1,z)^+.
\end{align*}
From a straightforward calculation, we have that
\begin{align*}
  [\wt \beta_1^-(z_1),\wt \beta_1^+(z_2)]
  =\gamma(z_2-z_1),\quad
  \te{where }
  \gamma(z)=\log f(-z)^{ (q^{-2}-q^{2})(q^{2\ell}-1) }.
\end{align*}
Notice that
\begin{align*}
  &\Res_zz\inv\gamma(-z)
  =\Res_zz\inv\log f(z)^{ (q^{2}-q^{-2})(q^{-2\ell}-1) }\\
  =&\Res_zz\inv\log f_0(z)^{ (q^{2}-q^{-2})(q^{-2\ell}-1) }
  =\log\frac{f_0(2\hbar-2\ell\hbar)}{f_0(2\hbar+2\ell\hbar)}.
\end{align*}
Then \cite[Lemma 8.9]{K-Coproduct-q-aff-va} provides that
\begin{align*}
  &Y_{\sqrt\ell \Z\al_1}^{\eta_\ell}(E_\ell(\beta_1),z)
  =\(\frac{f_0(2\hbar+2\ell\hbar)}{f_0(2\hbar-2\ell\hbar)}\)^\half Y_{\sqrt\ell \Z\al_1}^{\eta_\ell}\(\exp\(\(-q^{-\ell\partial}2\hbar f_0(2\partial\hbar) \beta_1\)_{-1}\)\vac,z\)\\
  =&\(\frac{f_0(2\hbar+2\ell\hbar)}{f_0(2\hbar-2\ell\hbar)}\)^\half  \exp(\wt \beta_1^+(z))\exp(\wt \beta_1^-(z))
  \(\frac{f_0(2\hbar-2\ell\hbar)}{f_0(2\hbar+2\ell\hbar)}\)^\half\\
  =&\exp(\wt \beta_1^+(z))\exp(\wt \beta_1^-(z)).
\end{align*}
Notice that
\begin{align*}
    \exp\(\wt \beta_1^+(z)\)=E^+(\beta_1,z-\hbar-\ell\hbar)
        E^+(-\beta_1,z+\hbar-\ell\hbar).
\end{align*}
Then
\begin{align*}
    &E_\ell(\beta_1)=\lim_{z\to 0}Y_{\sqrt\ell \Z\al_1}^{\eta_\ell}(E_\ell(\beta_1),z)\vac
    =\lim_{z\to 0}\exp\(\wt \beta_1^+(z)\)\exp\(\wt \beta_1^-(z)\)\vac\\
    =&\lim_{z\to 0}E^+(\beta_1,z-\hbar-\ell\hbar)E^+(-\beta_1,z+\hbar-\ell\hbar)\vac
    =E^+(\beta_1,-\hbar-\ell\hbar)E^+(-\beta_1,\hbar-\ell\hbar)\vac,
\end{align*}
which complete the proof of lemma.
\end{proof}

When $\ell=1$, Theorem \ref{thm:qlattice-inj} is equivalent to the following result.
\begin{prop}\label{prop:level-1-ssl2-case}
There exists a unique $\hbar$-adic quantum VA isomorphism from $L_{\hat\ssl_2,\hbar}^1$ to $V_{\Z\al_1}^{\eta_1}$  such that
\begin{align*}
  h_1\mapsto \al_1,\quad x_1^\pm\mapsto e_{\pm\al_1}.
\end{align*}
\end{prop}

\begin{proof}
From Proposition \ref{prop:qlatticeVA-mod-to-A-mod}, we have that
\begin{align*}
    \(V_{\Z\al_1}^{\eta_1},Y_{\Z\al_1}^{\eta_1}(\al_1,z),Y_{\Z\al_1}^{\eta_1}(e_{\pm\al_1},z)\)\in \obj \mathcal A_\hbar^{\eta_1}(\Z\al_1).
\end{align*}
Comparing the relations \eqref{A1}, \eqref{A2}, \eqref{A3} and \eqref{A4}
with the relations \eqref{eq:local-h-1}, \eqref{eq:local-h-2}, \eqref{eq:local-h-3} and \eqref{eq:local-h-4},
we get that
\begin{align*}
    \(V_{\Z\al_1}^{\eta_1},Y_{\Z\al_1}^{\eta_1}(\al_1,z),Y_{\Z\al_1}^{\eta_1}(e_{\pm\al_1},z)\)\in \obj \mathcal M_{\wh 1}(\ssl_2).
\end{align*}
By utilizing Proposition \ref{prop:universal-M-tau}, we obtain a unique $\hbar$-adic nonlocal VA homomorphism $\psi:F_{\hat\ssl_2,\hbar}^1\to V_{\Z\al_1}^{\eta_1}$ such that
\begin{align*}
  \psi(h_1)=\al_1,\quad \psi(x_1^\pm)=e_{\pm\al_1}.
\end{align*}

From \eqref{eq:eta-sp-exp}, we get that
\begin{align*}
  &Y_{\Z\al_1}^{\eta_1}(e_{\al_1},z)e_{-\al_1}
  =f(z)^{-1-q^2}E^+(\al_1,z)\vac.
\end{align*}
Then
\begin{align*}
  Y_{\Z\al_1}^{\eta_1}(e_{\al_1},z)^-e_{\al_1}=&0\quad\te{and}\quad
  Y_{\Z\al_1}^{\eta_1}(e_{\al_1},z)^-e_{-\al_1}
  =\frac{1}{q-q\inv}
    \(\frac{\vac}{z}-\frac{E^+(\al_1,-2\hbar)\vac}
        {z+2\hbar}\).
\end{align*}
By recalling Lemma \ref{lem:E-beta}, we know that
\begin{align*}
  &E^+(\al_1,-2\hbar)\vac=E_1(\al_1).
\end{align*}
Using \eqref{eq:eta-sp-exp}, it is also derived that
\begin{align*}
  &Y_{\Z\al_1}^{\eta_1}(e_{\al_1},z)e_{\al_1}
  =f(z)^{1+q^2}E^+(\al_1,z)e_{2\al_1}.
\end{align*}
It implies that
\begin{align*}
    &(e_{\al_1})_{-1}e_{\al_1}=\Res_zz\inv Y_{\Z\al_1}^{\eta_1}(e_{\al_1},z)e_{\al_1}
    =\Res_zz\inv f(z)^{1+q^2}E^+(\al_1,z)e_{2\al_1}=0.
\end{align*}
By applying Proposition \ref{prop:universal-qaff}, $\psi$ gives rise to a unique $\hbar$-adic nonlocal VA homomorphism $V_{\ssl_2,\hbar}^1\to L_{\Z\al_1}^{\eta_1}$.
Notice that the $\C[[\hbar]]$-module map $\psi$ induces a $\C$-isomorphism
\begin{align*}
  L_{\ssl_2,\hbar}^1/\hbar L_{\ssl_2,\hbar}^1\cong L_{\ssl_2}^1\cong V_{\Z\al_1}\cong V_{\Z\al_1}^{\eta_1}/\hbar V_{\Z\al_1}^{\eta_1}.
\end{align*}
Therefore, $\psi$ must be an $\hbar$-adic nonlocal VA isomorphism.
Finally, through a comparison of the action of quantum Yang-Baxter operators on generators (refer to Theorems \ref{thm:quotient-algs} and \ref{thm:qlatticeVA}), it is concluded that $\psi$ is an $\hbar$-adic quantum VA isomorphism.
\end{proof}

In the remainder of this section, we prove the induction step, that is, if Theorem \ref{thm:qlattice-inj} holds for $\ell$,
then it holds for $\ell+1$.
Let $Y_\Delta$ be the vertex operator map of the twisted tensor product $L_{\hat\ssl_2,\hbar}^\ell\wh\ot L_{\hat\ssl_2,\hbar}^1$.
By Lemma \ref{lem:AQ-sp1-4} and the injection \eqref{eq:inj}, the proof of Theorem \ref{thm:qlattice-inj} for $\ell+1$
reduces to the following two results, whose proofs will be presented in Subsections \ref{subsec:pf-AQ-sp5}
and \ref{subsec:pf-AQ-sp6-7}.

%Since the $\hbar$-adic quantum VA homomorphism $\Delta:L_{\hat\ssl_2,\hbar}^{\ell+1}\to L_{\hat\ssl_2,\hbar}^\ell\wh\ot L_{\hat\ssl_2,\hbar}^1$ is
%injective,
%the proposition below follows immediately from Lemmas \ref{lem:AQ-sp5-alt}, \ref{lem:Delta-AQ-sp-5+} and \ref{lem:Delta-AQ-sp-5-}.

\begin{prop}\label{prop:AQ-sp5}
Let us assume that Theorem \ref{thm:qlattice-inj} holds for $\ell$.
Then the relation \eqref{A5} holds for $\beta_{1,\hbar}(z)=Y_\Delta(\Delta(h_1),z)$ and $e_{1,\hbar}^\pm(z)=Y_\Delta(\Delta(f_1^\pm),z)$.
\end{prop}

\begin{prop}\label{prop:AQ-sp6-7}
Let us assume that Theorem \ref{thm:qlattice-inj} holds for $\ell$.
Then relations \eqref{A6} and \eqref{A7} hold for $\beta_{1,\hbar}(z)=Y_\Delta(\Delta(h_1),z)$ and $e_{1,\hbar}^\pm(z)=Y_\Delta(\Delta(f_1^\pm),z)$.
\end{prop}

The proofs of the above two propositions rely on explicit expressions of $\Delta(f_1^\pm)$.
Recalling the definitions of $h_1^\pm(z)$ and $\wt h_1^\pm(z)$ given in Lemma \ref{lem:com-formulas},
we obtain the following expressions directly from Propositions \ref{prop:delta-int+} and \ref{prop:delta-int-}.

\begin{lem}\label{lem:Delta-f}
In $L_{\hat\ssl_2,\hbar}^\ell\wh\ot L_{\hat\ssl_2,\hbar}^1$, we have that
\begin{align}
    \Delta(f_1^+)\label{eq:Delta-f+}
    =&q^{\partial}\exp\(\wt h_1^+((\ell-1)\hbar)\)f_1^+\ot q^{\ell\partial}f_1^+\\
    &\nonumber\times\frac{f_0(2(\ell-1)\hbar)^{\frac14}f_0(2\ell\hbar)^\half f_0(2(\ell+1)\hbar)^{\frac 14}}{f_0(2\hbar)^\half},\\
    \Delta(f_1^-)\label{eq:Delta-f-}
    =&q^{-\partial}f_1^-\ot q^{\ell\partial}f_1^-
    \frac{f_0(2(\ell-1)\hbar)^{\frac14}f_0(2\ell\hbar)^\half f_0(2(\ell+1)\hbar)^{\frac 14}}{f_0(2\hbar)^\half}.
\end{align}
\end{lem}

\subsection{Proof of Proposition \ref{prop:AQ-sp5}}\label{subsec:pf-AQ-sp5}
The main purpose of this subsection is to prove Proposition \ref{prop:AQ-sp5}.
We need the following equivalent conditions for \eqref{A5}.

\begin{lem}\label{lem:AQ-sp5-alt}
In the $\hbar$-adic quantum VA $L_{\hat\ssl_2,\hbar}^\ell$,
the equation \eqref{A5} holds for $\ell$ if and only if the following equation holds true:
\begin{align}\label{eq:AQ-sp5-alt1}
  &\frac{d}{dz}Y_{\wh\ell}(f_1^\pm,z)=\pm h_1^+(z)Y_{\wh\ell}(f_1^\pm,z)\pm Y_{\wh\ell}(f_1^\pm,z)h_1^-(z)
\end{align}
which can also be expressed as the equation:
\begin{align}\label{eq:AQ-sp5-alt2}
  \partial f_1^\pm=\pm(h_1)_{-1}f_1^\pm
  -f_1^\pm\left.\(\pd{z}\log f_0(z)^{[2]_q[\ell]_qq^\ell}\)\right|_{z=0}
\end{align}
\end{lem}

\begin{proof}
From Proposition \ref{prop:normal-ordering-rel-general} and the definition of $f_1^\pm$, we have that
\begin{align*}
  &Y_{\wh\ell}(f_1^\pm,z)=\sqrt{c_{1,\ell}} Y_{\wh\ell}\(\(x_1^\pm\)_{-1}^\ell\vac,z+(1-\ell)\hbar\)\\
  =&c\:Y_{\wh\ell}(x_1^\pm,z+(\ell-1)\hbar)Y_{\wh\ell}(x_1^\pm,z+(\ell-3)\hbar)
    \cdots Y_{\wh\ell}(x_1^\pm,z+(1-\ell)\hbar)\;
\end{align*}
for some $c\in\C[[\hbar]]$.
By utilizing \eqref{eq:com-formulas-2}, we get that
\begin{align}
  &\left[h_1^-(z_1),Y_{\wh\ell}\(f_1^\pm,z_2\)\right]\nonumber\\
  =&c[h_1^+(z_1),\:Y_{\wh\ell}(x_1^\pm,z_2+(\ell-1)\hbar)\cdots
    Y_{\wh\ell}(x_1^\pm,z_2+(1-\ell)\hbar)\;]\nonumber\\
  =&\pm c\:Y_{\wh\ell}(x_1^\pm,z_2+(\ell-1)\hbar)\cdots
    Y_{\wh\ell}(x_1^\pm,z_2+(1-\ell)\hbar)\;\pd{z_1}\log
  f(z_1-z_2)^{[2]_qq^\ell [\ell]_q}\nonumber\\
  =&\pm Y_{\wh\ell}\(f_1^\pm,z_2\)\pd{z_1}\log f(z_1-z_2)^{[2]_qq^\ell[\ell]_q}.\label{eq:com-h--f}
\end{align}
Taking $\Rat_{z_1}$ on the both hand sides of \eqref{eq:com-h--f}, we get that
\begin{align*}
  &[h_1^-(z_1)^+,Y_{\wh\ell}(f_1^\pm,z_2)]
  =\pm Y_{\wh\ell}(f_1^\pm,z_2)\pd{z_1}\log f(z_1-z_2)^{[2]_qq^\ell[\ell]_q}\\
  =&\pm Y_{\wh\ell}(f_1^\pm,z_2)\Rat_{z_1}\pd{z_1}\log f(z_1-z_2)^{[2]_qq^\ell[\ell]_q}\\
  =&\pm Y_{\wh\ell}(f_1^\pm,z_2)\pd{z_1}\log f_0(z_1-z_2)^{[2]_qq^\ell[\ell]_q},
\end{align*}
where $\Rat_{z_1}$ denotes the rational part of a series in variable $z_1$.
Hence, the RHS of \eqref{A5} equals to
\begin{align*}
  &\pm h_1^+(z)Y_{\wh\ell}(f_1^\pm,z)\pm h_1^-(z)^+Y_{\wh\ell}(f_1^\pm,z)\pm Y_{\wh\ell}(f_1^\pm,z)h_1^-(z)^-\\
  &\quad-Y_{\wh\ell}(f_1^\pm,z)\(\pd{z}\log f_0(z)^{[2]_q[\ell]_qq^\ell}\)|_{z=0}\\
  =&\pm h_1^+(z)Y_{\wh\ell}(f_1^\pm,z)\pm Y_{\wh\ell}(f_1^\pm,z)h_1^-(z).
\end{align*}
Therefore, the first assertion is valid.

Let the both hand sides of \eqref{eq:com-h--f} act on $\vac$.
Letting $z_2=0$, we get that
\begin{align}\label{eq:h--f}
  h_1^-(z)f_1^\pm
  =\pm f_1^\pm\ot \pd{z}\log f(z)^{[2]_q[\ell]_qq^\ell}.
\end{align}
Then we have that
\begin{align}\label{eq:h--1-f}
  &(h_1)_{-1}f_1^\pm=\Res_zz\inv Y_{\wh\ell}(h_1,z)f_1^\pm
  =\Res_zz\inv(h_1^+(z)+h_1^-(z))f_1^\pm\\
  =&h_1^+(0)f_1^\pm
  \pm\Res_zz\inv \pd{z}\log f(z)^{[2]_q[\ell]_qq^\ell}
  =h_1^+(0)f_1^\pm \pm\left.\(   \pd{z}\log
  f(z)^{[2]_q[\ell]_qq^\ell}\)\right|_{z=0}.\nonumber
\end{align}
Suppose that the relation \eqref{eq:AQ-sp5-alt1} holds true.
Letting the both hand sides of \eqref{eq:AQ-sp5-alt1} act on $\vac$ and taking $z=0$, we get that
\begin{align}\label{eq:der-f}
  \partial f_1^\pm=\pm h_1^+(0)f_1^\pm.
\end{align}
Combining this with \eqref{eq:h--1-f}, we complete the proof of
\eqref{eq:AQ-sp5-alt2}.

On the other hand, notice that
\begin{align*}
  &Y_{\wh\ell}(h_1,z_1)Y_{\wh\ell}(f_1^\pm,z_2)
  =(h_1^+(z_1)+h_1^-(z_1))Y_{\wh\ell}(f_1^\pm,z_2)\\
  =&h_1^+(z_1)Y_{\wh\ell}(f_1^\pm,z_2)
  +Y_{\wh\ell}(f_1^\pm,z_2)h_1^-(z_1)
  \pm Y_{\wh\ell}(f_1^\pm,z_2)
  \pd{z_1}\log f(z_1-z_2)^{[2]_q[\ell]_qq^\ell},
\end{align*}
where the last equation follows from \eqref{eq:com-h--f}.
Then
\begin{align*}
  &Y_\E\(Y_{\wh\ell}(h_1,z),z_0\)Y_{\wh\ell}(f_1^\pm,z)\\
  =&h_1^+(z+z_0)Y_{\wh\ell}(f_1^\pm,z)
  +Y_{\wh\ell}(f_1^\pm,z)h_1^-(z+z_0)
  \pm Y_{\wh\ell}(f_1^\pm,z)\pd{z_0}\log f(z_0)^{[2]_q[\ell]_qq^\ell}.
\end{align*}
Taking $\Res_{z_0}z_0\inv$ on both hand sides, we get that
\begin{align*}
  &Y_{\wh\ell}\((h_1)_{-1}f_1^\pm,z\)\\
  =&h_1^+(z)Y_{\wh\ell}(f_1^\pm,z)+Y_{\wh\ell}(f_1^\pm,z)h_1^-(z)
  \pm Y_{\wh\ell}(f_1^\pm,z)\left.\(\pd{z_0}\log f(z_0)^{[2]_q[\ell]_qq^\ell}\)\right|_{z_0=0}.
\end{align*}
Therefore, \eqref{eq:AQ-sp5-alt1} can be derived from \eqref{eq:AQ-sp5-alt2}.
\end{proof}

To compute $\partial\Delta(f_1^\pm)$ and $\Delta(h_1)_{-1}\Delta(f_1^\pm)$, we need the following five technical results.

\begin{lem}\label{lem:S-xk-h}
We have that
\begin{align*}
  &S_{\ell,1}(z)\(f_1^\pm\ot h_1\)
  =f_1^\pm\ot h_1\pm f_1^\pm\ot \vac
  \ot \pd{z}\log
  f(z)^{[2]_q[\ell]_q(q-q\inv)}.
\end{align*}
\end{lem}

\begin{proof}
We can utilize Proposition \ref{prop:normal-ordering-rel-general} to derive the following expression:
\begin{align}\label{eq:S-xk-temp1}
  &\Rat_{z_1\inv,\dots z_\ell\inv}Y_{\wh\ell}(x_1^\pm,z_1)\cdots Y_{\wh\ell}(x_1^\pm,z_\ell)\vac
  =\prod_{a=1}^\ell\frac{z_a}{z_a-2(\ell-a)\hbar}(x_1^\pm)_{-1}^\ell\vac.
\end{align}
By combining equations \eqref{eq:qyb-hex1} and \eqref{eq:S-twisted-2}, we obtain the following result:
\begin{align*}
  &S_{\ell,1}(z)\(Y_{\wh\ell}(x_1^\pm,z_1)\cdots Y_{\wh\ell}(x_1^\pm,z_\ell)\vac\ot h_1\)\\
  =&Y_{\wh\ell}(x_1^\pm,z_1)\cdots Y_{\wh\ell}(x_1^\pm,z_\ell)\vac\ot h_1\nonumber\\
  &\pm Y_{\wh\ell}(x_1^\pm,z_1)\cdots Y_{\wh\ell}(x_1^\pm,z_\ell)\vac\ot \vac
  \ot [2]_{q^{\pd{z}}}\(q^{\pd{z}}-q^{-\pd{z}}\)\sum_{a=1}^\ell e^{z_a\pd{z}}\pd{z}\log f(z)\\
  =&Y_{\wh\ell}(x_1^\pm,z_1)\cdots Y_{\wh\ell}(x_1^\pm,z_\ell)\vac\ot h_1\nonumber\\
  &\pm Y_{\wh\ell}(x_1^\pm,z_1)\cdots Y_{\wh\ell}(x_1^\pm,z_\ell)\vac\ot \vac
  \ot [2]_{q^{\pd{z}}}\(q^{\pd{z}}-q^{-\pd{z}}\)\sum_{a=1}^\ell e^{z_a\pd{z}}\pd{z}\log f(z).
\end{align*}
Combining this with \eqref{eq:S-xk-temp1}, we get that
\begin{align*}
  &S_{\ell,1}(z) \(f_1^\pm\ot h_1\)= S_{\ell,1}(z)\(\sqrt{c_{1,\ell}} q^{(1-\ell)\partial}(x_1^\pm)_{-1}^\ell\vac\ot h_1\)\\
  =&\sqrt{c_{1,\ell}}
  \Res_{z_1,\dots,z_\ell}z_1\inv\cdots z_\ell\inv
  S_{\ell,1}(z)
  \(q^{(1-\ell)\partial}Y_{\wh\ell}(x_1^\pm,z_1)\cdots Y_{\wh\ell}(x_1^\pm,z_\ell)\vac\ot h_1\)\\
  =&\sqrt{c_{1,\ell}} \Res_{z_1,\dots,z_\ell}z_1\inv\cdots z_\ell\inv q^{(1-\ell)\partial}
  Y_{\wh\ell}(x_1^\pm,z_1)\cdots Y_{\wh\ell}(x_1^\pm,z_\ell)\vac\ot h_1\\
  &\pm \sqrt{c_{1,\ell}} \Res_{z_1,\dots,z_\ell}z_1\inv\cdots z_\ell\inv
    q^{(1-\ell)\partial}Y_{\wh\ell}(x_1^\pm,z_1)\cdots Y_{\wh\ell}(x_1^\pm,z_\ell)\vac\ot \vac\\
  &\ot [2]_{q^{\pd{z}}}\(q^{\pd{z}}-q^{-\pd{z}}\)q^{(1-\ell)\pd{z}}\sum_{a=1}^\ell e^{z_a\pd{z}}\pd{z}\log f(z)\\
  =&\Res_{z_1,\dots,z_\ell}\prod_{a=1}^\ell\frac{1}{z_a-2(\ell-a)\hbar}f_1^\pm\ot h_1
  \pm \Res_{z_1,\dots,z_\ell}\prod_{a=1}^\ell\frac{1}{z_a-2(\ell-a)\hbar}f_1^\pm\ot \vac\\
  &\quad\ot [2]_{q^{\pd{z}}}\(q^{\pd{z}}-q^{-\pd{z}}\)q^{(1-\ell)\pd{z}}\sum_{a=1}^\ell e^{z_a\pd{z}}\pd{z}\log f(z)\\
  =&f_1^\pm\ot h_1\pm f_1^\pm\ot \vac\ot
   \pd{z}\log f(z)^{[2]_q[\ell]_q(q-q\inv)}.
\end{align*}
This completes the proof of the lemma.
\end{proof}

\begin{lem}\label{lem:S-Eu-h}
For $u\in L_{\hat\ssl_2,\hbar}^\ell$, $h\in L_{\hat\ssl_2,\hbar}^1$ such that
\begin{align*}
  &\exp(\wt h_1^-(z))u=u\ot g(z),\quad
  S_{\ell,1}(z)(u\ot h)=u\ot h+u\ot \vac\ot f_1(z),\\
  &\quad\te{and}\quad S_{\ell,1}(z)(E_\ell(h_1)\ot h)=E_\ell(h_1)\ot h+E_\ell(h_1)\ot\vac\ot f_2(z),
\end{align*}
for some $f_1(z),f_2(z),g(z)\in \C((z))[[\hbar]]$ with $g(z)$ invertible.
Then
\begin{align*}
  &S_{\ell,1}(z)\(\exp(\wt h_1^+(z_1))u\ot h\)\\
  =&\exp(\wt h_1^+(z_1))u\ot h
  +\exp(\wt h_1^+(z_1))u\ot \vac\ot (f_1(z)+f_2(z+z_1)).
\end{align*}
\end{lem}

\begin{proof}
By utilizing Proposition \ref{prop:Y-E}, we get that
\begin{align}\label{eq:S-Eu-h-temp1}
  &Y_{\wh\ell}(E_\ell(h_1),z_1)u
  =\exp\(\wt h_1^+(z_1)\)\exp\(\wt h_1^-(z_1)\)u
  =\exp\(\wt h_1^+(z_1)\)u \ot g(z_1).
\end{align}
In addition, we get from \eqref{eq:qyb-hex1} that
\begin{align*}
  &S_{\ell,1}(z)\(Y_{\wh\ell}(E_\ell(h_1),z_1)u\ot h\)\\
  =&Y_{\wh\ell}(z_1)S_{\ell,1}^{23}(z)S_{\ell,1}^{13}(z+z_1)(E_\ell(h_1)\ot u\ot h)\\
  =&Y_{\wh\ell}(E_\ell(h_1),z_1)u\ot h+Y_{\wh\ell}(E_\ell(h_1),z_1)u\ot\vac\ot\( f_1(z)+f_2(z+z_1) \).
\end{align*}
Combining this with \eqref{eq:S-Eu-h-temp1}, we can conclude that
\begin{align*}
  &S_{\ell,1}(z)\(\exp\(\wt h_1^+(z_1)\)u\ot h\)\ot g(z_1)\\
  =&\exp\(\wt h_1^+(z_1)\)u\ot h\ot g(z_1)
  +\exp\(\wt h_1^+(z_1)\)u\ot \vac\ot g(z_1)(f_1(z)+f_2(z+z_1)).
\end{align*}
Since $g(z)$ is invertible, we complete the proof of the lemma.
\end{proof}

\begin{lem}\label{lem:com-formulas3}
We have that
\begin{align}
  &\left[\wt h_1^-(z_1),Y_{\wh\ell}(f_1^\pm,z_2)\right]
  =\pm Y_{\wh\ell}(f_1^\pm,z_2)
    \log f(z_1-z_2)^{q^{-\ell-1}+q^{1-\ell}-q^{\ell-1}-q^{\ell+1}},\nonumber\\
  &\exp\(\wt h_1^-(z_1)\)Y_{\wh\ell}(f_1^\pm,z_2)
  \label{eq:com-formulas-15}
  =Y_{\wh\ell}(f_1^\pm,z_2)\exp\(\wt h_1^-(z_1)\)\\
  &\qquad\qquad\times  f(z_1-z_2)^{\pm q^{-\ell-1}\pm q^{1-\ell}\mp q^{\ell-1}\mp q^{\ell+1}}.\nonumber
\end{align}
\end{lem}

\begin{proof}
From Proposition \ref{prop:normal-ordering-rel-general} and the definition of $f_1^\pm$, we have that
\begin{align*}
  &Y_{\wh\ell}(f_1^\pm,z)=\sqrt{c_{1,\ell}} Y_{\wh\ell}\(\(x_1^\pm\)_{-1}^\ell\vac,z+(1-\ell)\hbar\)\\
  =&c\:Y_{\wh\ell}(x_1^\pm,z+(\ell-1)\hbar)Y_{\wh\ell}(x_1^\pm,z+(\ell-3)\hbar)
    \cdots Y_{\wh\ell}(x_1^\pm,z+(1-\ell)\hbar)\;
\end{align*}
for some $c\in\C[[\hbar]]$.
By utilizing \eqref{eq:com-formulas-7}, we get that
\begin{align*}
  &\left[\wt h_1^-(z_1),Y_{\wh\ell}\(f_1^\pm,z_2\)\right]
  =c[\wt h_1^+(z_1),\:Y_{\wh\ell}(x_1^\pm,z_2+(\ell-1)\hbar)\cdots
    Y_{\wh\ell}(x_1^\pm,z_2+(1-\ell)\hbar)\;]\\
  =&\pm c\:Y_{\wh\ell}(x_1^\pm,z_2+(\ell-1)\hbar)\cdots
    Y_{\wh\ell}(x_1^\pm,z_2+(1-\ell)\hbar)\;\log
  f(z_1-z_2)^{(q^{-2}-q^2)[\ell]_q}\\
  =&\pm Y_{\wh\ell}\(f_1^\pm,z_2\)\log f(z_1-z_2)^{q^{-\ell-1}+q^{1-\ell}-q^{\ell-1}-q^{\ell+1}}.
\end{align*}
The relation \eqref{eq:com-formulas-15} follows immediate from the preceding relation.
\end{proof}

\begin{lem}\label{lem:S-wt-h-f+-h}
We have that
\begin{align*}
  &S_{\ell,1}(z)\(\exp(\wt h_1^+((\ell-1)\hbar)f_1^+\ot h_1\)\\
  =&\exp(\wt h_1^+((\ell-1)\hbar)f_1^+\ot h_1
  +\exp(\wt h_1^+((\ell-1)\hbar)f_1^+\ot\vac
  \ot \pd{z}\log f(z)^{[2]_q[\ell]_q(q\inv-q^{-3})}.
\end{align*}
\end{lem}

\begin{proof}
Recall Lemma \ref{lem:S-xk-h} which states that
\begin{align*}
  &S_{\ell,1}(z)(f_1^+\ot h_1)
  =f_1^+\ot h_1+f_1^+\ot\vac\ot \pd{z}\log f(z)^{[2]_q[\ell]_q(q-q\inv) }.
\end{align*}
From \eqref{eq:S-E-1}, we get that
\begin{align*}
  &S_{\ell,1}(z)\(E_\ell(h_1)\ot h_1\)\\
  =&E_\ell(h_1)\ot h_1
  -E_\ell(h_1)\ot \vac
  \ot [2]_{q^{\pd{z}}}
    [\ell]_{q^{\pd{z}}}\(q^{\pd{z}}-q^{-\pd{z}}\)^2
  q^{-\ell\pd{z}}\pd{z}\log f(z)\\
  =&E_\ell(h_1)\ot h_1  +E_\ell(h_1)\ot \vac\ot
  \pd{z}\log f(z)^{-[2]_q[\ell]_q(q-q\inv)^2q^{-\ell}}.
\end{align*}
Applying these two equations to Lemma \ref{lem:S-Eu-h}, we complete the proof of the lemma.
\end{proof}

\begin{lem}\label{lem:h--exp-wt-h-f+}
In the $\hbar$-adic quantum VA $L_{\hat\ssl_2,\hbar}^\ell$, we have that
\begin{align*}
  &h_1^-(z-2\hbar)\exp\(\wt h_1^+((\ell-1)\hbar)\)f_1^+
  =\exp\(\wt h_1^+((\ell-1)\hbar)\)f_1^+
  \ot \pd{z}\log f(z)^{[2]_q[\ell]_qq^\ell}
\end{align*}
\end{lem}

\begin{proof}
The following equation is a special case of \eqref{eq:com-formulas-10}:
\begin{align*}
  \left[h_1^-(z_1),\exp\(\wt h_1^+(z_2)\)\right]=\exp\(\wt h_1^+(z_2)\)
  \pd{z_1}\log f(z_1-z_2)^{ [2]_q(q^\ell-q^{-\ell})q^{2\ell} }.
\end{align*}
It implies that
\begin{align*}
  &\left[h_1^-(z-2\hbar),\exp\(\wt h_1^+((\ell-1)\hbar)\)\right]
  =\exp\(\wt h_1^+((\ell-1)\hbar)\)
  \pd{z}\log f(z)^{ [2]_q(q^\ell-q^{-\ell})q^{\ell-1} }.
\end{align*}
Then we have that
\begin{align*}
  &h_1^-(z-2\hbar)\exp\(\wt h_1^+((\ell-1)\hbar)\)f_1^+\\
  =&\left[h_1^-(z-2\hbar),\exp\(\wt h_1^+((\ell-1)\hbar)\)\right]f_1^+
  +\exp\(\wt h_1^+((\ell-1)\hbar)\)
  h_1^-(z-2\hbar)f_1^+\\
  =&\exp\(\wt h_1^+((\ell-1)\hbar)\)f_1^+
  \ot \pd{z}\log f(z)^{[2]_q[\ell]_q(q^\ell-q^{\ell-2}+q^{\ell-2})}\\
  =&\exp\(\wt h_1^+((\ell-1)\hbar)\)f_1^+
  \ot \pd{z}\log f(z)^{[2]_q[\ell]_qq^\ell},
\end{align*}
where the second equation follows from \eqref{eq:h--f}.
\end{proof}

The following proves the ``$+$'' case of Proposition \ref{prop:AQ-sp5}.

\begin{lem}\label{lem:Delta-AQ-sp-5+}
Let us assume that Theorem \ref{thm:qlattice-inj} holds true for $\ell$.
In $L_{\hat\ssl_2,\hbar}^\ell\wh\ot L_{\hat\ssl_2,\hbar}^1$, we have that
\begin{align*}
  &\Delta(\partial f_1^+)=\Delta(h_1)_{-1}\Delta(f_1^+)
  -\Delta(f_1^+)\left.\(\pd{z}\log f(z)^{[2]_q[\ell+1]_qq^{\ell+1}}\)\right|_{z=0}.
\end{align*}
\end{lem}

\begin{proof}
From \eqref{eq:Delta-f+}, we have that
\begin{align*}
  \Delta(f_1^+)=c\(q^{\partial}\exp\(\wt h_1^+((\ell-1)\hbar)\)f_1^+\ot q^{\ell\partial}f_1^+\)
\end{align*}
for some invertible $c\in\C[[\hbar]]$.
Then we have that
\begin{align*}
  &Y_\Delta\(\Delta(h_1),z\)\Delta(f_1^+)\\
  =&cY_\Delta\(q^{-\partial}h_1\ot \vac+\vac\ot q^{\ell\partial}h_1,z\)
  \(q^{\partial}\exp\(\wt h_1^+((\ell-1)\hbar)\)f_1^+\ot q^{\ell\partial}f_1^+\)\\
  =&cY_{\wh\ell}^{12}(z)Y_{\wh 1}^{34}(z)S_{\ell,1}^{23}(-z)
  \Big(
    q^{-\partial}h_1\ot q^{\partial}\exp\(\wt h_1^+((\ell-1)\hbar)\)f_1^+
    \ot\vac\ot q^{\ell\partial}f_1^+\\
  &\quad+\vac\ot q^{\partial}\exp\(\wt h_1^+((\ell-1)\hbar)\)f_1^+
  \ot q^{\ell\partial}h_1\ot q^{\ell\partial}f_1^+\Big)\\
  =&cY_{\wh\ell}^{12}(z)Y_{\wh 1}^{34}(z)
  \Big(
    q^{-\partial}h_1\ot q^{\partial}\exp\(\wt h_1^+((\ell-1)\hbar)\)f_1^+\ot\vac\ot q^{\ell\partial}f_1^+\\
  &\quad+\vac\ot q^{\partial}\exp\(\wt h_1^+((\ell-1)\hbar)\)f_1^+
  \ot q^{\ell\partial}h_1\ot q^{\ell\partial}f_1^+\\
  &\quad+\vac\ot q^{\partial}\exp\(\wt h_1^+((\ell-1)\hbar)\)f_1^+
  \ot \vac\ot q^{\ell\partial}f_1^+\\
  &\quad\quad\ot \pd{(-z)}\log f(-z)^{[2]_q[\ell]_q(q\inv-q^{-3})q^{1-\ell}}
  \Big)\\
  =&cY_{\wh\ell}^{12}(z)Y_{\wh 1}^{34}(z)
  \Big(
    q^{-\partial}h_1\ot q^{\partial}\exp\(\wt h_1^+((\ell-1)\hbar)\)f_1^+\ot\vac\ot q^{\ell\partial}f_1^+\\
  &\quad+\vac\ot q^{\partial}\exp\(\wt h_1^+((\ell-1)\hbar)\)f_1^+
  \ot q^{\ell\partial}h_1\ot q^{\ell\partial}f_1^+\\
  &\quad-\vac\ot q^{\partial}\exp\(\wt h_1^+((\ell-1)\hbar)\)f_1^+
  \ot \vac\ot q^{\ell\partial}f_1^+
  \ot \pd{z}\log f(z)^{[2]_q[\ell]_q(q^{\ell}-q^{\ell+2})}
  \Big)\\
  =&cq^\partial Y_{\wh\ell}(h_1,z-2\hbar)\exp\(\wt h_1^+((\ell-1)\hbar)\)f_1^+\ot q^{\ell\partial}f_1^+\\
  &+cq^\partial \exp\(\wt h_1^+((\ell-1)\hbar)\)f_1^+\ot q^{\ell\partial}Y_{\wh 1}(h_1,z)f_1^+\\
  &-\Delta(f_1^+)\ot \pd{z}\log f(z)^{[2]_q[\ell]_q(q^{\ell}-q^{\ell+2})}\\
%%%%%%%%%%%%%%%%
  =&cq^\partial \(h_1^+(z-2\hbar)+h_1^-(z-2\hbar)\)\exp\(\wt h_1^+((\ell-1)\hbar)\)f_1^+\ot q^{\ell\partial}f_1^+\\
  &+cq^\partial \exp\(\wt h_1^+((\ell-1)\hbar)\)f_1^+\ot q^{\ell\partial}(h_1^+(z)+h_1^-(z))f_1^+\\
  &-\Delta(f_1^+)\ot \pd{z}\log f(z)^{[2]_q[\ell]_q(q^{\ell}-q^{\ell+2})}\\
%%%%%%%%%%%%%%%%
  =&\(h_1^+(z-\hbar)\ot 1+1\ot h_1^+(z+\ell\hbar)\)\Delta(f_1^+)+\Delta(f_1^+)
  \ot \pd{z}\log f(z)^{[2]_q[\ell]_qq^{\ell+2}+[2]_qq},
\end{align*}
where the first equation follows from \eqref{eq:Delta-f+}, the third equation follows from Lemma \ref{lem:S-wt-h-f+-h}
and the last equation follows from \eqref{eq:h--f} and Lemma \ref{lem:h--exp-wt-h-f+}.
Applying $\Res_zz\inv$, we get that
\begin{align}\label{eq:Delta-h--1-Delta-f+-temp}
  &\Delta(h_1)_{-1}\Delta(f_1^+)\nonumber\\
  =&\(h_1^+(-\hbar)\ot 1+1\ot h_1^+(\ell\hbar)\)\Delta(f_1^+)\nonumber\\
  &+\Delta(f_1^+)
  \ot \Res_zz\inv\pd{z}\log f(z)^{[2]_q[\ell]_qq^{\ell+2}+[2]_qq}\nonumber\\
  =&\(h_1^+(-\hbar)\ot 1+1\ot h_1^+(\ell\hbar)\)\Delta(f_1^+)\\
  &\nonumber+\Delta(f_1^+)
  \ot \left.\(\pd{z}\log f(z)^{[2]_q[\ell+1]_qq^{\ell+1}}\)\right|_{z=0}.
\end{align}

On the other hand,
we notice that
\begin{align*}
  [\partial, \wt h_1^+(z)]
  =\pd{z}\wt h_1^+(z)
  =h_1^+(z-(\ell+1)\hbar)-h_1^+(z+(1-\ell)\hbar).
\end{align*}
Then we have that
\begin{align}\label{eq:der-Delta-f+-temp}
  &\Delta(\partial f_1^+)=(\partial\ot 1+1\ot\partial)\Delta(f_1^+)\nonumber\\
  =&c(\partial\ot 1+1\ot\partial)q^\partial\exp\(\wt h_1^+((\ell-1)\hbar)\)f_1^+\ot q^{\ell\partial}f_1^+\nonumber\\
  =&cq^\partial \(h_1^+(-2\hbar)-h_1^+(0)\)\exp\(\wt h_1^+((\ell-1)\hbar)\)f_1^+\ot q^{\ell\partial}f_1^+\nonumber\\
  &+cq^\partial \exp\(\wt h_1^+((\ell-1)\hbar)\)h_1^+(0)f_1^+\ot q^{\ell\partial}f_1^+\nonumber\\
  &+cq^\partial\exp\(\wt h_1^+((\ell-1)\hbar)\)f_1^+\ot q^{\ell\partial}h_1^+(0)f_1^+\nonumber\\
  =&(h_1^+(-\hbar)\ot 1+1\ot h_1^+(\ell\hbar))
  \Delta(f_1^+),
\end{align}
where the second equation follows from \eqref{eq:der-f}.
Combining \eqref{eq:Delta-h--1-Delta-f+-temp} and
\eqref{eq:der-Delta-f+-temp}, we complete the proof of the lemma.
\end{proof}

The following proves the ``$-$'' case of Proposition \ref{prop:AQ-sp5}.

\begin{lem}\label{lem:Delta-AQ-sp-5-}
Let us assume that Theorem \ref{thm:qlattice-inj} holds true for $\ell$.
In $L_{\hat\ssl_2,\hbar}^\ell\wh\ot L_{\hat\ssl_2,\hbar}^1$, we have that
\begin{align*}
  &\Delta(\partial f_1^-)=-\Delta(h_1)_{-1}\Delta(f_1^-)
  -\Delta(f_1^-)\left.\(\pd{z}\log f(z)^{[2]_q[\ell+1]_qq^{\ell+1}}\)\right|_{z=0}.
\end{align*}
\end{lem}

\begin{proof}
Let $Y_\Delta$ be the vertex operator map of the twisted tensor product $L_{\hat\ssl_2,\hbar}^\ell\wh\ot L_{\hat\ssl_2,\hbar}^1$.
From \eqref{eq:Delta-f-}, we have that
\begin{align*}
  \Delta(f_1^-)=c\(q^{-\partial}f_1^-\ot q^{\ell\partial}f_1^-\)
\end{align*}
for some invertible $c\in\C[[\hbar]]$.
Then
\begin{align*}
  &Y_\Delta(\Delta(h_1),z)\Delta(f_1^-)\\
  =&cY_\Delta\(q^{-\partial}h_1\ot\vac+\vac\ot q^{\ell\partial}h_1,z\)\(q^{-\partial}f_1^-\ot q^{\ell\partial}f_1^-\)\\
  =&cY_{\wh\ell}^{12}(z)Y_{\wh 1}^{34}(z)S_{\ell,1}^{23}(-z)
  \(q^{-\partial}h_1\ot q^{-\partial}f_1^-\ot \vac\ot q^{\ell\partial}f_1^-
  +\vac\ot q^{-\partial}f_1^-\ot q^{\ell\partial}h_1\ot q^{\ell\partial}f_1^-\)\\
  =&cY_{\wh\ell}^{12}(z)Y_{\wh 1}^{34}(z)
  \Big(
    q^{-\partial}h_1\ot q^{-\partial}f_1^-\ot \vac\ot q^{\ell\partial}f_1^-
    +
    \vac\ot q^{-\partial}f_1^-\ot q^{\ell\partial}h_1\ot q^{\ell\partial}f_1^-\\
  &\quad
    -\vac\ot q^{-\partial}f_1^-\ot \vac\ot q^{\ell\partial}f_1^-
    \ot \pd{(-z)}\log f(-z)^{[2]_q[\ell]_q(q-q\inv)q^{-\ell-1}}
  \Big)\\
  =&cq^{-\partial}Y_{\wh\ell}(h_1,z)f_1^-\ot q^{\ell\partial}f_1^-
  +cq^{-\partial}f_1^-\ot q^{\ell\partial}Y_{\wh 1}(h_1,z)f_1^-\\
  &-\Delta(f_1^-)\ot \pd{(-z)}\log f(-z)^{[2]_q[\ell]_q(q-q\inv)q^{-\ell-1}}\\
  =&cq^{-\partial}(h_1^+(z)+h_1^-(z))f_1^-\ot q^{\ell\partial}f_1^-
  +cq^{-\partial}f_1^-\ot q^{\ell\partial}(h_1^+(z)+h_1^-(z))f_1^-\\
  &+\Delta(f_1^-)\ot \pd{z}\log f(z)^{[2]_q[\ell]_q(q\inv-q)q^{\ell+1}}\\
  =&\(h_1^+(z-\hbar)\ot 1+1\ot h_1^+(z+\ell\hbar)\)\Delta(f_1^-)
  +\Delta(f_1^-)\ot \pd{z}\log f(z)^{[2]_q[\ell]_q(q^\ell-q^{\ell+2}-q^\ell)-[2]_qq}\\
  =&\(h_1^+(z-\hbar)\ot 1+1\ot h_1^+(z+\ell\hbar)\)\Delta(f_1^-)
  -\Delta(f_1^-)\ot \pd{z}\log f(z)^{[2]_q[\ell+1]_qq^{\ell+1}},
\end{align*}
where the third equation follows from Lemma \ref{lem:S-xk-h} and relations \cite[Lemma 4.13]{K-Coproduct-q-aff-va}
and the sixth equation follows from \eqref{eq:h--f}.
Applying $\Res_zz\inv$, we get that
\begin{align}\label{eq:Delta-h--1-Delta-f--temp}
  &\Delta(h_1)_{-1}\Delta(f_1^-)\nonumber\\
  =&\(h_1^+(-\hbar)\ot 1+1\ot h_1^+(\ell\hbar)\)\Delta(f_1^-)
  -\Delta(f_1^-)
  \ot \Res_zz\inv\pd{z}\log f(z)^{[2]_q[\ell+1]_qq^{\ell+1}}\nonumber\\
  =&\(h_1^+(-\hbar)\ot 1+1\ot h_1^+(\ell\hbar)\)\Delta(f_1^-)
  -\Delta(f_1^-)
  \ot \left.\(\pd{z}\log f(z)^{[2]_q[\ell+1]_qq^{\ell+1}}\)\right|_{z=0}.
\end{align}

On the other hand, from \eqref{eq:der-f} we have that
\begin{align*}
  &\(\partial\ot 1+1\ot\partial\)\Delta(f_1^-)
  =cq^{-\partial}\partial f_1^-\ot q^{\ell\partial}f_1^-
  +cq^{-\partial}f_1^-\ot q^{\ell\partial}\partial f_1^-\\
  =&-cq^{-\partial}h_1^+(0)f_1^-\ot q^{\ell\partial}f_1^-
  -cq^{-\partial}f_1^-\ot q^{\ell\partial}h_1^+(0)f_1^-
  =-(h_1^+(-\hbar)\ot 1+1\ot h_1^+(\ell\hbar))f_1^-.
\end{align*}
Combining this with \eqref{eq:Delta-h--1-Delta-f--temp}, we complete the proof of lemma.
\end{proof}

\subsection{Proof of Proposition \ref{prop:AQ-sp6-7}}\label{subsec:pf-AQ-sp6-7}
The main purpose of this subsection is to prove Proposition \ref{prop:AQ-sp6-7}.
Recall the operator $h_1^+(z)$ given in Lemma \ref{lem:com-formulas}.
Set
\begin{align*}
  h_1^+(z)=\sum_{n\in\Z_+}h_1(-n)z^{n-1}.
\end{align*}
We define the following analogue of the classical lattice VA operator $E^+(\beta_1,z)$ in $L_{\hat\g,\hbar}^\ell$:
\begin{align*}
    E^+(h_1,z)=\exp\( \sum_{n\in\Z_+}\frac{h_1(-n)}{n}z^n \).
\end{align*}
From the definition of $\wt h_1^+(z)$, we have that
\begin{align}\label{eq:exp-wt-h=E-E}
    &\exp\(\wt h_1^+(z)\)=E^+(h_1,z-(\ell+1)\hbar)E^+(-h_1,z+(1-\ell)\hbar).
\end{align}
The following result implies that relations \eqref{A6} and \eqref{A7} can be established through the subsequent relation:
\begin{align}\label{A8}
  \tag{A8} Y_{\wh\ell}(f_1^+,z)f_1^-=E^+(h_1,z)\vac f(z)^{-[2]_q[\ell]_qq^\ell}.
\end{align}
\begin{lem}\label{lem:A8}
If Theorem \ref{thm:qlattice-inj} holds true for $\ell$, then the relation \eqref{A8} holds true.
On the other hand, the relation \eqref{A8} implies the relations \eqref{A6} and \eqref{A7}.
\end{lem}

\begin{proof}
The first statement follows from the last relation of \cite[Theorem 4.11]{K-q-lattice-va} and \eqref{eq:eta-sp-exp}.
On the other hand, note that the relation \eqref{A4} holds true (see Lemma \ref{lem:AQ-sp1-4}). That is,
\begin{align*}
  &\iota_{z_1,z_2}f(z_1-z_2)^{(q^2+1+q^{2\ell}) [\ell]_q^2 }Y_{\wh\ell}(f_1^+,z_1)Y_{\wh\ell}(f_1^-,z_2)\\
  =&\iota_{z_2,z_1}f(z_1-z_2)^{(q^{-2}+1+q^{2\ell}) [\ell]_q^2 }Y_{\wh\ell}(f_1^-,z_2)Y_{\wh\ell}(f_1^+,z_1).
\end{align*}
%From \eqref{eq:qyb-locality} we get that
%\begin{align*}
%  S_{\ell,\ell}(-z)(f_1^-\ot f_1^+)=f_1^-\ot f_1^+\ot f(z)^{(q^{-2}-q^2)[\ell]_q^2}.
%\end{align*}
Combining this with \eqref{A8} and \cite[(2.9)]{Li-h-adic}, we get that
\begin{align*}
  &z_0\inv\delta\left(\frac{z_1-z_2}{z_0}\right)Y_{\wh\ell}(f_1^+,z_1)Y_{\wh\ell}(f_1^-,z_2)\\
  &\quad-z_0\inv\delta\left(\frac{-z_2+z_1}{z_0}\right)f(-z_2+z_1)^{[2]_q[\ell]_q(q^{-\ell}-q^\ell)}Y_{\wh\ell}(f_1^-,z_2)Y_{\wh\ell}(f_1^+,z_1)\\
  =&z_1\inv\delta\left(\frac{z_2+z_0}{z_1}\right)f(z_0)^{-[2]_q[\ell]_qq^\ell}Y_{\wh\ell}(E^+(h_1,z_0)\vac,z_2).
\end{align*}
Multiplying $f(z_0)^{[2]_q[\ell]_qq^\ell}$ on both hand sides and taking $\Res_{z_0}$ we get the relation \eqref{A6}.
Then we have that
\begin{align*}
  &\left.\left(\iota_{z_1,z_2}f(z_1-z_2)^{[2]_q[\ell]_qq^\ell}Y_{\wh\ell}(f_1^+,z_1)Y_{\wh\ell}(f_1^-,z_2)\right)\right|_{z_1=z_2}\\
  =&\lim_{z_0\to 0}\left(f(z_0)^{[2]_q[\ell]_qq^\ell}Y_\E(Y_{\wh\ell}(f_1^+,z_2),z_0)Y_{\wh\ell}(f_1^-,z_2)\right)\\
  =&\lim_{z_0\to 0}\left(f(z_0)^{[2]_q[\ell]_qq^\ell}Y_{\wh\ell}(Y_{\wh\ell}(f_1^+,z_0)f_1^-,z_2)\right)\\
  =&\lim_{z_0\to 0}Y_{\wh\ell}(E^+(h_1,z_0)\vac,z_2)=1,
\end{align*}
where the first equation follows from \cite[Proposition 4.12]{Li-h-adic},
the second equation follows from \eqref{eq:weak-asso},
the third equation follows from \eqref{A8} and the last equation follows from the definition of $E^+(h_1,z)$.
Therefore, the relation \eqref{A7} holds true.
\end{proof}

Lemma \ref{lem:A8} shows that Proposition \ref{prop:AQ-sp6-7} is an immediate consequence of the following result.
\begin{prop}\label{prop:A8}
Let us assume that Theorem \ref{thm:qlattice-inj} holds for $\ell$.
Then
\begin{align*}
  Y_\Delta(\Delta(f_1^+),z)\Delta(f_1)^-=\Delta(E^+(h_1,z)\vac)f(z)^{-[2]_q[\ell+1]_qq^{\ell+1}}.
\end{align*}
\end{prop}

The rest of this subsection is devoted to proving Proposition \ref{prop:A8}.
We first compute the explicit expression for $Y_\Delta(\Delta(f_1^+),z)\Delta(f_1^-)$.
It requires the following three technical results.

\begin{lem}\label{lem:Y-exp-wt-h-x}
We have that
\begin{align*}
  &Y_{\wh\ell}\(\exp\(\wt h_1^+((\ell-1)\hbar)\)f_1^\pm,z\)\\
  =&\exp\(\wt h_1^+(z+(\ell-1)\hbar)\)Y_{\wh\ell}(f_1^\pm,z)
  \exp\(\wt h_1^-(z+(\ell-1)\hbar)\)\\
  &\qquad\times \(\frac{f_0(2(\ell-1)\hbar)}{f_0(2(\ell+1)\hbar)}\)^\half
  \(\frac{f_0(2\hbar)}{f_0(2(\ell-1)\hbar) f_0(2\ell\hbar)}\)^{\pm 1}.
\end{align*}
\end{lem}

\begin{proof}
From \eqref{eq:com-formulas-6}, we have that
\begin{align*}
  [\wt h_1^-(z_1+(\ell-1)\hbar),\wt h_1^+(z_2+(\ell-1)\hbar)]
  =\gamma(z_2-z_1),
\end{align*}
where
\begin{align*}
  &\gamma(z)
  =\(q^{2\pd{z}}-q^{-2\pd{z}}\)
  \(q^{-2\ell\pd{z}}-1 \)\log f(-z).
\end{align*}
From Lemma \ref{lem:com-formulas3}, we have that
\begin{align*}
  &[\wt h_1^-(z_1+(\ell-1)\hbar),Y_{\wh\ell}(f_1^\pm,z_2)]
  =Y_{\wh\ell}(f_1^\pm,z_2)\iota_{z_1,z_2}\gamma_1(z_1-z_2),
\end{align*}
where
\begin{align*}
  \gamma_1(z)=\pm
    \log f(z)^{q^{-2} +1-q^{2\ell-2}-q^{2\ell} }.
\end{align*}
Notice that
\begin{align*}
  &\Res_zz\inv\gamma(-z)
  =\Res_zz\inv\(q^{-2\pd{z}}-q^{2\pd{z}}\)\(q^{2\ell\pd{z}}-1\)
  \log f_0(z)
  =\log\frac{f_0(2(\ell-1)\hbar)}{f_0(2(\ell+1)\hbar)},
\end{align*}
and that
\begin{align*}
  \Res_zz\inv \gamma_1(z)
  =&\pm\Res_zz\inv \( q^{-2\pd{z}}+1-q^{2(\ell-1)\pd{z}}-q^{2\ell\pd{z}} \)
    \log f(z)\\
  =&\pm\Res_zz\inv \( q^{-2\pd{z}}+1-q^{2(\ell-1)\pd{z}}-q^{2\ell\pd{z}} \)
    \log f_0(z)\\
  =&\pm\log\frac{f_0(2\hbar)f_0(0)}{f_0(2(\ell-1)\hbar) f_0(2\ell\hbar)}
  =\pm\log\frac{f_0(2\hbar)}{f_0(2(\ell-1)\hbar) f_0(2\ell\hbar)}.
\end{align*}
The by utilizing \cite[Lemma 8.9]{K-Coproduct-q-aff-va}, we complete the proof.
\end{proof}

\begin{lem}\label{lem:Y-E-wt-h-x+-x-}
We have that
\begin{align*}
    &Y_{\wh\ell}\(\exp\(\wt h_1^+((\ell-1)\hbar)\)f_1^+,z\)f_1^-\\
    =&\exp\(\wt h_1^+(z+(\ell-1)\hbar)\)Y_{\wh\ell}(f_1^+,z)f_1^-
    \ot f(z)^{q^{2\ell-2}+q^{2\ell}-q^{-2}-1}\\
    &\qquad\times \frac{f_0(2\hbar)}{f_0(2(\ell-1)\hbar)^\half f_0(2\ell\hbar)f_0(2(\ell+1)\hbar)^\half}.
\end{align*}
\end{lem}

\begin{proof}
From Lemma \ref{lem:Y-exp-wt-h-x}, we get that
\begin{align}\label{eq:Y-E-wt-h-x+-x--temp}
    &Y_{\wh\ell}\(\exp\(\wt h_1^+((\ell-1)\hbar)\)f_1^+,z\)f_1^-
    =\exp\(\wt h_1^+(z+(\ell-1)\hbar)\)Y_{\wh\ell}(f_1^+,z)
    \\
  &\qquad\exp\(\wt h_1^-(z+(\ell-1)\hbar)\)f_1^-
  \frac{f_0(2\hbar)}{f_0(2(\ell-1)\hbar)^\half f_0(2\ell\hbar)f_0(2(\ell+1)\hbar)^\half}.\nonumber
\end{align}
From Lemma \ref{lem:com-formulas3}, we get that
\begin{align*}
    &\exp\(\wt h_1^-(z_1+(\ell-1)\hbar)\)Y_{\wh\ell}(f_1^-,z_2)\vac
    =Y_{\wh\ell}(f_1^-,z_2)\vac\ot f(z_1-z_2)^{q^{2\ell-2}+q^{2\ell}-q^{-2}-1}.
\end{align*}
Taking $z_2=0$, we get that
\begin{align*}
    &\exp\(\wt h_1^-(z+(\ell-1)\hbar)\)f_1^-
    =f_1^-\ot f(z)^{q^{2\ell-2}+q^{2\ell}-q^{-2}-1}.
\end{align*}
Combining with \eqref{eq:Y-E-wt-h-x+-x--temp}, we complete the proof of lemma.
\end{proof}

\begin{lem}\label{lem:S-x-ell-x+}
We have that
\begin{align*}
    &S_{\ell,1}(z)(f_1^-\ot f_1^+)
    =f_1^-\ot f_1^+\ot f(z)^{q^{\ell+1}+q^{\ell-1}-q^{1-\ell}-q^{-\ell-1}}.
\end{align*}
\end{lem}

\begin{proof}
From \eqref{eq:qyb-hex1} and \eqref{eq:S-twisted-4}, we have that
\begin{align*}
  &S_{\ell,1}(z)\(Y(x_1^-,z_1)Y(x_1^-,z_2)\cdots Y(x_1^-,z_\ell)\vac\ot x_1^+\)\\
  =&Y(x_1^-,z_1)Y(x_1^-,z_2)\cdots Y(x_1^-,z_\ell)\vac\ot x_1^+
  \ot \prod_{a=1}^\ell
  e^{z_a\pd{z}}f(z)^{q^{2}-q^{-2}}.
\end{align*}
Recall from Proposition \ref{prop:normal-ordering-rel-general}, we have that
\begin{align*}
    &\Rat_{z_1\inv,\dots,z_\ell\inv} Y(x_1^-,z_1)Y(x_1^-,z_2)\cdots Y(x_1^-,z_\ell)\vac
    =\prod_{a=1}^\ell\frac{z_a}{z_a-2(\ell-a)\hbar}\(x_1^-\)_{-1}^\ell\vac.
\end{align*}
Then
\begin{align*}
    &S_{\ell,1}(z)(f_1^-\ot f_1^+)=\sqrt{c_{1,\ell}}
    S_{\ell,1}(z)\(q^{(1-\ell)\partial}\(x_1^-\)_{-1}^\ell\vac\ot x_1^+\)\\
    =&\Res_{z_1,\dots,z_\ell}z_1\inv\cdots z_\ell\inv \sqrt{c_{1,\ell}}
    S_{\ell,1}(z)\(q^{(1-\ell)\partial}Y(x_1^-,z_1)Y(x_1^-,z_2)\cdots Y(x_1^-,z_\ell)\vac\ot x_1^+\)\\
    =&f_1^-\ot f_1^+\ot \prod_{a=1}^\ell
        f(z)^{(q^2-q^{-2})q^{2(\ell-a)}q^{1-\ell}}
    =f_1^- \ot f_1^+\ot f(z)^{q^{\ell+1}+q^{\ell-1}-q^{1-\ell}-q^{-\ell-1}}.
\end{align*}
We complete the proof of the lemma.
\end{proof}

The following is the explicit expression of $Y_\Delta(\Delta(f_1^+),z)\Delta(f_1^-)$.

\begin{lem}\label{lem:Delta-f+Delta-f-}
We have
\begin{align*}
  &Y_\Delta(\Delta(f_1^+),z)\Delta(f_1^-)
  =q^{-\partial}E^+(h_1,z)\vac\ot q^{\ell\partial}E^+(h_1,z)\vac
    \ot f(z)^{-[2]_q[\ell+1]_qq^{\ell+1}}.
\end{align*}
\end{lem}

\begin{proof}
From Lemma \ref{lem:Delta-f}, we have that
\begin{align*}
    &Y_\Delta(\Delta(f_1^+),z)\Delta(f_1^-)
    =Y_{\wh\ell}^{12}(z)Y_{\wh 1}^{34}(z)S_{\ell,1}^{23}(-z)\\
    &\quad\( q^{\partial}\exp\(\wt h_1^+((\ell-1)\hbar)\)f_1^+   \ot q^{-\partial}f_1^-
        \ot q^{\ell\partial}f_1^+
        \ot q^{\ell\partial}f_1^-\)\\
    &\quad\times
        \frac{f_0(2(\ell-1)\hbar)^\half f_0(2\ell\hbar)f_0(2(\ell+1)\hbar)^\half}{f_0(2\hbar)}
        \\
    =&Y_{\wh\ell}^{12}(z)Y_{\wh 1}^{34}(z)
    \( q^{\partial}\exp\(\wt h_1^+((\ell-1)\hbar)\)f_1^+   \ot q^{-\partial}f_1^-
        \ot q^{\ell\partial}f_1^+
        \ot q^{\ell\partial}f_1^-\)\\
    &\quad
        \ot f(-z)^{1+q^{2}-q^{2\ell}-q^{2+2\ell}}
        \frac{f_0(2(\ell-1)\hbar)^\half f_0(2\ell\hbar)f_0(2(\ell+1)\hbar)^\half}{f_0(2\hbar)}\\
    =&q^{-\partial}Y_{\wh\ell}\( \exp\(\wt h_1^+((\ell-1)\hbar)\)f_1^+,z+2\hbar \)
        f_1^-\ot q^{\ell\partial}Y_{\wh 1}\(f_1^+,z\)f_1^- \\
    &\quad \ot f(z)^{1+q^{2}-q^{2\ell}-q^{2+2\ell}}
    \frac{f_0(2(\ell-1)\hbar)^\half f_0(2\ell\hbar)f_0(2(\ell+1)\hbar)^\half}{f_0(2\hbar)}\\
    =&q^{-\partial}\exp\(\wt h_1^+(z+(\ell+1)\hbar)\)Y_{\wh\ell}( f_1^+,z+2\hbar  )f_1^-
    \ot q^{\ell\partial}Y_{\wh 1}\(f_1^+,z\)f_1^-,
\end{align*}
where the second equation follows from Lemma \ref{lem:S-x-ell-x+} and \cite[Lemma 4.13]{K-Coproduct-q-aff-va},
and the last equation follows from Lemma \ref{lem:Y-E-wt-h-x+-x-}.
Proposition \ref{prop:level-1-ssl2-case} provides that
\begin{align*}
    &Y_{\wh 1}(f_1^+,z)f_1^-=E^+(h_1,z)\vac\ot f(z)^{-1-q^2}.
\end{align*}
From the induction assumption, we have that
\begin{align*}
    &Y_{\wh\ell}( f_1^+,z+2\hbar  )f_1^-
    =E^+(h_1,z+2\hbar)\vac
        f(z)^{-q^{\ell+2}(q+q\inv)[\ell]_q}.
\end{align*}
Combining these equations, we get that
\begin{align*}
    &Y_\Delta(\Delta(f_1^+),z)\Delta(f_1^-)\\
    =&q^{-\partial}\exp\(\wt h_1^+(z+(\ell+1)\hbar)\)Y_{\wh\ell}( f_1^+,z+2\hbar  )f_1^-
    \ot q^{\ell\partial}Y_{\wh 1}\(f_1^+,z\)f_1^-\\
    =&q^{-\partial}\exp\(\wt h_1^+(z+(\ell+1)\hbar)\)E^+(h_1,z+2\hbar)\vac
    \ot q^{\ell\partial}E^+(h_1,z)\vac
    \ot f(z)^{-1-q^2-q^{\ell+2}(q+q\inv)[\ell]_q}\\
    =&q^{-\partial}E^+(h_1,z)\vac\ot q^{\ell\partial}E^+(h_1,z)\vac
    \ot f(z)^{-[2]_q[\ell+1]_qq^{\ell+1}},
\end{align*}
where the last equation follows from \eqref{eq:exp-wt-h=E-E}.
\end{proof}

Next, we compute the explicit expression for $\Delta(E^+(h_1,z)\vac)$,
which relies on the following technical result.

\begin{lem}\label{lem:E+=exp-wh-h}
For $a\in\C$, we have that
\begin{align*}
    &\exp\(q^{a\partial}\( \frac{e^{z\partial}-1}{\partial}h_1 \)_{-1}\)\vac\\
    =&q^{a\partial}E^+(h_1,z)\vac
    f_0(z)^{-\half[2]_q[\ell]_q\(q^\ell+q^{-\ell}\)}
    \prod_{b=0}^{\ell-1}f_0(2b\hbar)f_0(2(b+1)\hbar).
\end{align*}
\end{lem}

\begin{proof}
Set
\begin{align*}
    &\wh h_1^\pm(z,w)=\frac{e^{w\pd{z}}-1}{\pd{z}}q^{a\pd{z}} h_1^\pm(z).
\end{align*}
Then we deduce from \eqref{eq:com-formulas-1} that
\begin{align*}
    &[\wh h_1^-(z_1,w),\wh h_1^+(z_2,w)]
    =[2]_{q^{\pd{z_2}}}[\ell]_{q^{\pd{z_2}}}
    q^{-\ell\pd{z_2}}\(e^{-w\pd{z_2}}-1\)\(e^{w\pd{z_2}}-1\)\log f(z_1-z_2).
\end{align*}
That is
\begin{align*}
    &[\wh h_1^-(z_1,w),\wh h_1^+(z_2,w)]
    =\gamma(z_2-z_1),
\end{align*}
where
\begin{align*}
    \gamma(z)=&[2]_{q^{\pd{z}}}[\ell]_{q^{\pd{z}}}q^{-\ell\pd{z}}\(2-e^{-w\pd{z}}-e^{w\pd{z}}\)\log f(-z).
\end{align*}
Notice that
\begin{align*}
    &\Res_zz\inv \gamma(-z)
    =\Res_zz\inv [2]_{q^{\pd{z}}}[\ell]_{q^{\pd{z}}}q^{\ell\pd{z}}\(2-e^{w\pd{z}}-e^{-w\pd{z}}\)\log f(z)\\
    =&\Res_zz\inv [2]_{q^{\pd{z}}}[\ell]_{q^{\pd{z}}}q^{\ell\pd{z}}\(2-e^{w\pd{z}}-e^{-w\pd{z}}\)\log f_0(z)\\
    =&2\Res_zz\inv [2]_{q^{\pd{z}}}[\ell]_{q^{\pd{z}}}q^{\ell\pd{z}}\log f_0(z)
    -\Res_zz\inv [2]_{q^{\pd{z}}}[\ell]_{q^{\pd{z}}}q^{\ell\pd{z}}\log f_0(z+w)\\
    &-\Res_zz\inv [2]_{q^{\pd{z}}}[\ell]_{q^{\pd{z}}}q^{\ell\pd{z}}\log f_0(w-z)\\
    =&2\Res_zz\inv [2]_{q^{\pd{z}}}[\ell]_{q^{\pd{z}}}q^{\ell\pd{z}}\log f_0(z)
    -\Res_zz\inv [2]_{q^{\pd{w}}}[\ell]_{q^{\pd{w}}}q^{\ell\pd{w}}\log f_0(z+w)\\
    &-\Res_zz\inv [2]_{q^{\pd{w}}}[\ell]_{q^{\pd{w}}}q^{-\ell\pd{w}}\log f_0(w-z)\\
    =&2\log\prod_{b=0}^{\ell-1}f_0(2b\hbar)f_0(2(b+1)\hbar)
    -\log f_0(w)^{[2]_q[\ell]_q\(q^\ell+q^{-\ell}\)}
\end{align*}
Then we get from \cite[Lemma 8.9]{K-Coproduct-q-aff-va} that
\begin{align*}
    &Y_{\wh\ell}\(\exp\(\( q^{a\partial}\frac{e^{w\partial}-1}{\partial}h_1 \)_{-1}\)\vac,z\)\vac\\
    =&\exp\(\(\wh h_1^+(z,w)+\wh h_1^-(z,w)\)_{-1}\)\vac
    =\exp\(\wh h_1^+(z,w)\)\exp\(\wh h_1^-(z,w)\)\vac\\
    &\quad\times f_0(w)^{-\half[2]_q[\ell]_q\(q^\ell+q^{-\ell}\)}
    \prod_{b=0}^{\ell-1}f_0(2b\hbar)f_0(2(b+1)\hbar)\\
    =&\exp\(\wh h_1^+(z,w)\)\vac
    f_0(w)^{-\half[2]_q[\ell]_q\(q^\ell+q^{-\ell}\)}
    \prod_{b=0}^{\ell-1}f_0(2b\hbar)f_0(2(b+1)\hbar) .
\end{align*}
Taking $z=0$, we get that
\begin{align}\label{eq:E+=exp-wh-h-temp}
    &\exp\(q^{a\partial}\( \frac{e^{w\partial}-1}{\partial}h_1 \)_{-1}\)\vac\\
    =&\exp\(\wh h_1^+(0,w)\)\vac
    f_0(w)^{-\half[2]_q[\ell]_q\(q^\ell+q^{-\ell}\)}
    \prod_{b=0}^{\ell-1}f_0(2b\hbar)f_0(2(b+1)\hbar)\nonumber\\
    =&E^+(h_1,w+a\hbar)E^+(-h_1,a\hbar)\vac
    f_0(w)^{-\half[2]_q[\ell]_q\(q^\ell+q^{-\ell}\)}
    \prod_{b=0}^{\ell-1}f_0(2b\hbar)f_0(2(b+1)\hbar),\nonumber
\end{align}
where the last equation follows from the following fact
\begin{align*}
    &\wh h_1^+(0,w)=q^{a\pd{z}}\frac{e^{w\pd{z}}-1}{\pd z}h_1^+(z)|_{z=0}\
    =\frac{e^{(w+a\hbar)\pd{z}}-1}{\pd{z}}h_1^+(z)|_{z=0}
    -\frac{e^{a\hbar\pd{z}}-1}{\pd{z}}h_1^+(z)|_{z=0}\\
    =&\sum_{n\in\Z_+}(w+a\hbar)^n \frac{1}{n!}\pdiff{z}{n-1}h_1^+(z)|_{z=0}
    -\sum_{n\in\Z_+}(a\hbar)^n \frac{1}{n!}\pdiff{z}{n-1}h_1^+(z)|_{z=0}\\
    =& \sum_{n\in\Z_+}\frac{h_1(-n)}{n}(w+a\hbar)^n
    -\sum_{n\in\Z_+}\frac{h_1(-n)}{n}(a\hbar)^n.
\end{align*}
Notice that
\begin{align*}
  &q^{a\partial}\sum_{n\in\Z_+}h_1(-n) z^{n-1}=q^{a\partial}Y(h_1,z)^+=Y(h_1,z+a\hbar)^+ q^{a\partial}\\
  =&\sum_{n\in\Z_+}h_1(-n)\sum_{k=0}^{n-1}\binom{n-1}{k}z^{n-1-k}(a\hbar)^kq^{a\partial}\\
  =&\sum_{n\in\Z_+}\sum_{k\in\N}h_1(-n-k)\binom{n+k-1}{k}z^{n-1}(a\hbar)^kq^{a\partial}.
\end{align*}
It shows that
\begin{align*}
  &q^{a\partial}h_1(-n)=\sum_{k\in\N}h_1(-n-k)\binom{n+k-1}{k}(a\hbar)^kq^{a\partial}.
\end{align*}
Then we have that
\begin{align*}
  &q^{a\partial}\sum_{n\in\Z_+}\frac{h_1(-n)}{n}z^n
  =\sum_{n\in\Z_+}
  \sum_{k\in\N}\frac{h_1(-n-k)}{n}\binom{n+k-1}{k}(a\hbar)^k z^n
  q^{a\partial}\\
  =&\sum_{n\in\Z_+}
  \sum_{k\in\N}\frac{h_1(-n-k)}{n+k}\binom{n+k}{k}(a\hbar)^k z^n
  q^{a\partial}
  =\sum_{n\in\Z_+}\frac{h_1(-n)}{n}
  \sum_{k=0}^{n-1}\binom{n}{k}(a\hbar)^k z^{n-k}
  q^{a\partial}\\
  =&\sum_{n\in\Z_+}\frac{h_1(-n)}{n}\((z+a\hbar)^n-(a\hbar)^n\)q^{a\partial}.
\end{align*}
It follows that
\begin{align*}
  &q^{a\partial}E^+(h_1,z)q^{-a\partial}=E^+(h_1,z+a\hbar)E^+(-h_1,a\hbar).
\end{align*}
Combining this with \eqref{eq:E+=exp-wh-h-temp}, we complete the proof.
\end{proof}

The following is the explicit expression for $\Delta(E^+(h_1,z)\vac)$.

\begin{lem}\label{lem:Delta-E+}
We have
\begin{align*}
  \Delta&\(E^+(h_1,z)\vac\)= q^{-\partial}E^+(h_1,z)\vac\ot q^{\ell\partial} E^+(h_1,z)\vac.
\end{align*}
\end{lem}

\begin{proof}
Set
\begin{align*}
  a=q^{-\partial}\frac{e^{z\partial}-1}{\partial}h_1\ot\vac,\quad b=\vac\ot q^{\ell\partial}\frac{e^{z\partial}-1}{\partial}h_1
\end{align*}
and $\wh h_1=a+b$.
Recall the definition of $S_\Delta(z)=S_{1,\ell}^{23}(z)S_{\ell,\ell}^{13}(z)S_{1,1}^{24}(z)S_{\ell,1}^{14}(z)$ (see \eqref{eq:def-S-Delta}).
Then we get that
\begin{align*}
  &S_\Delta(w)(b\ot a)\\
  =&S_{1,\ell}^{23}(w)S_{\ell,\ell}^{13}(w)S_{1,1}^{24}(w)S_{\ell,1}^{14}(w)
  \(\vac\ot q^{\ell\partial}\frac{e^{z\partial}-1}{\partial}h_1\ot q^{-\partial}\frac{e^{z\partial}-1}{\partial}h_1 \ot \vac\)\\
  =&\(1\ot q^{\ell\partial+\ell\pd{w}}\frac{e^{z\partial+z\pd{w}}-1}{\partial+\pd{w}}\ot q^{-\partial+\pd{w}}\frac{e^{z\partial-z\pd{w}}-1}{\partial-\pd{w}}\ot 1\)
   S_{1,\ell}^{23}(w)(\vac\ot h_1\ot h_1\ot\vac)\\
  =&b\ot a+\vac\ot\vac\ot\vac\ot\vac\\
  &\quad\ot q^{(\ell+1)\pd{w}}\frac{e^{z\pd{w}}-1}{\pd{w}}\frac{e^{-z\pd{w}}-1}{-\pd{w}}
   [2]_{q^{\pd{w}}}[\ell]_{q^{\pd{w}}}\(q^{\pd{w}}-q^{-\pd{w}}\)\pdiff{w}{2}\log f(w)\\
  =&b\ot a+\vac\ot\vac\ot\vac\ot\vac\\
  &\quad\ot
  [2]_{q^{\pd{w}}}[\ell]_{q^{\pd{w}}}q^{(\ell+1)\pd{w}}\(q^{\pd{w}}-q^{-\pd{w}}\)
  \(e^{z\pd{w}}+e^{-z\pd{w}}-2\)
  \log f(w),
\end{align*}
where the second equation follows from \cite[Lemma 4.13]{K-Coproduct-q-aff-va} and
the third equation follows from \eqref{eq:S-twisted-1}.
Then
\begin{align*}
  &[Y_\Delta(a,z_1),Y_\Delta(b,z_2)]w\\
  =&Y_\Delta(a,z_1)Y_\Delta(b,z_2)w-Y_\Delta(z_2)Y_\Delta^{23}(z_1)S_\Delta^{12}(z_2-z_1)(b\ot a\ot w)\\
  &+Y_\Delta(z_2)Y_\Delta^{23}(z_1)S_\Delta^{12}(z_2-z_1)(b\ot a\ot w)-Y_\Delta(b,z_2)Y_\Delta(a,z_1)w\\
  =&Y_\Delta(Y_\Delta(a,z_1-z_2)b-Y_\Delta(a,-z_2+z_1)b,z_2)\\
  &+Y_\Delta(z_2)Y_\Delta^{23}(z_1)S_\Delta^{12}(z_2-z_1)(b\ot a\ot w)-Y_\Delta(b,z_2)Y_\Delta(a,z_1)w\\
  =&Y_\Delta(z_2)Y_\Delta^{23}(z_1)\(S_\Delta^{12}(z_2-z_1)(b\ot a\ot w)-b\ot a\ot w\)\\
  =&w\ot \iota_{z_2,z_1}[2]_{q^{\pd{z_2}}}[\ell]_{q^{\pd{z_2}}}q^{(\ell+1)\pd{z_2}}\(q^{\pd{z_2}}-q^{-\pd{z_2}}\)
  \(e^{z\pd{z_2}}+e^{-z\pd{z_2}}-2\)
  \log f(z_2-z_1),
\end{align*}
where the second equation follows from \cite[(2.25)]{Li-h-adic}.
By applying $\Res_{z_1,z_2}z_1\inv z_2\inv$, we get
\begin{align*}
  &[a_{-1},b_{-1}]
  =\Res_{z_1,z_2}z_1\inv z_2\inv \iota_{z_2,z_1}
  [2]_{q^{\pd{z_2}}}[\ell]_{q^{\pd{z_2}}}q^{(\ell+1)\pd{z_2}}\(q^{\pd{z_2}}-q^{-\pd{z_2}}\)\\
  &\quad\times\(e^{z\pd{z_2}}+e^{-z\pd{z_2}}-2\)
  \log f(z_2-z_1)\\
  =&\Res_{z_2}z_2\inv [2]_{q^{\pd{z_2}}}[\ell]_{q^{\pd{z_2}}}q^{(\ell+1)\pd{z_2}}\(q^{\pd{z_2}}-q^{-\pd{z_2}}\)\(e^{z\pd{z_2}}+e^{-z\pd{z_2}}-2\)
  \log f(z_2)\\
  =&\Res_{z_2}z_2\inv [2]_{q^{\pd{z_2}}}[\ell]_{q^{\pd{z_2}}}q^{(\ell+1)\pd{z_2}}\(q^{\pd{z_2}}-q^{-\pd{z_2}}\)
  \log f_0(z_2+z)\\
  &+\Res_{z_2}z_2\inv [2]_{q^{\pd{z_2}}}[\ell]_{q^{\pd{z_2}}}q^{(\ell+1)\pd{z_2}}\(q^{\pd{z_2}}-q^{-\pd{z_2}}\)
  \log f_0(z_2-z)\\
  &-2\Res_{z_2}z_2\inv [2]_{q^{\pd{z_2}}}[\ell]_{q^{\pd{z_2}}}q^{(\ell+1)\pd{z_2}}\(q^{\pd{z_2}}-q^{-\pd{z_2}}\)
  \log f_0(z_2)\\
  =&\Res_{z_2}z_2\inv [2]_{q^{\pd{z}}}[\ell]_{q^{\pd{z}}}q^{(\ell+1)\pd{z}}\(q^{\pd{z}}-q^{-\pd{z}}\)
  \log f_0(z_2+z)\\
  &-\Res_{z_2}z_2\inv [2]_{q^{\pd{z}}}[\ell]_{q^{\pd{z}}}q^{-(\ell+1)\pd{z}}\(q^{\pd{z}}-q^{-\pd{z}}\)
  \log f_0(z-z_2)\\
  &-2\Res_{z_2}z_2\inv [2]_{q^{\pd{z_2}}}[\ell]_{q^{\pd{z_2}}}q^{(\ell+1)\pd{z_2}}\(q^{\pd{z_2}}-q^{-\pd{z_2}}\)
  \log f_0(z_2)\\
  =&2\log f_0(2\hbar)f_0(0)-2\log f_0(2(\ell+1)\hbar)f_0(2\ell\hbar)
  +\log f_0(z)^{ [2]_q\(q^\ell-q^{-\ell}\)\(q^{\ell+1}-q^{-\ell-1}\) }.
\end{align*}
From Baker-Campbell-Hausdorff formula, we have that
\begin{align*}
  &\Delta(E^+(h_1,z)\vac)
  =\Delta\(\exp\(\( \frac{e^{z\partial}-1}{\partial}h_1 \)_{-1}\)\vac\) \\
    &\quad\times f_0(z)^{\half[2]_q[\ell+1]_q\(q^{\ell+1}+q^{-\ell-1}\)}
    \prod_{a=0}^\ell f_0(2a\hbar)\inv f_0(2(a+1)\hbar)\inv \\
  =&\exp\(a_{-1}+b_{-1}\)\(\vac\ot\vac\)
     f_0(z)^{\half[2]_q[\ell+1]_q\(q^{\ell+1}+q^{-\ell-1}\)}
    \prod_{a=0}^\ell f_0(2a\hbar)\inv f_0(2(a+1)\hbar)\inv \\
  =&\exp\(a_{-1}\)\exp\(b_{-1}\) \exp\(-[a_{-1},b_{-1}]/2\)\(\vac\ot\vac\)\\
  &\quad\times
   f_0(z)^{\half[2]_q[\ell+1]_q\(q^{\ell+1}+q^{-\ell-1}\)}
    \prod_{a=0}^\ell f_0(2a\hbar)\inv f_0(2(a+1)\hbar)\inv \\
  =&\exp\(\( q^{-\partial}\frac{e^{z\partial}-1}{\partial}h_1 \)_{-1}\)\vac
    \ot \exp\(\( q^{\ell\partial}\frac{e^{z\partial}-1}{\partial}h_1 \)_{-1}\)\vac \\
  &\quad\times
  \frac{f_0(2\ell\hbar)f_0(2(\ell+1)\hbar)}{f_0(0)f_0(2\hbar)}
  f_0(z)^{ -\half[2]_q\(q^\ell-q^{-\ell}\)\(q^{\ell+1}-q^{-\ell-1}\) }
  \\
  &\quad\times
   f_0(z)^{\half[2]_q[\ell+1]_q\(q^{\ell+1}+q^{-\ell-1}\)}
    \prod_{a=0}^\ell f_0(2a\hbar)\inv f_0(2(a+1)\hbar)\inv \\
  =&\exp\(\( q^{-\partial}\frac{e^{z\partial}-1}{\partial}h_1 \)_{-1}\)\vac
    \ot \exp\(\( q^{\ell\partial}\frac{e^{z\partial}-1}{\partial}h_1 \)_{-1}\)\vac \\
  &\quad\times
  \frac{f_0(2\ell\hbar)f_0(2(\ell+1)\hbar)}{f_0(0)f_0(2\hbar)}
  f_0(z)^{ -\half[2]_q\(q^\ell-q^{-\ell}\)\(q^{\ell+1}-q^{-\ell-1}\) }\\
  &\quad\times
   f_0(z)^{\half[2]_q[\ell+1]_q\(q^{\ell+1}+q^{-\ell-1}\)}
    \prod_{a=0}^\ell f_0(2a\hbar)\inv f_0(2(a+1)\hbar)\inv \\
%%%%%%%%%%%%%%%%%%%%%%%%%%%%%%%%%%%%%%
  =&q^{-\partial}E^+(h_1,z)\vac\ot q^{\ell\partial}E^+(h_1,z)\vac
  f_0(z)^{-\half [2]_q[\ell]_q\(q^{\ell}+q^{-\ell}\)}
    \prod_{a=0}^{\ell-1} f_0(2a\hbar) f_0(2(a+1)\hbar) \\
  &\quad\times
  f_0(z)^{-\half [2]_q^2}f_0(0) f_0(2\hbar)
  \frac{f_0(2\ell\hbar)f_0(2(\ell+1)\hbar)}{f_0(0)f_0(2\hbar)}
  f_0(z)^{ -\half[2]_q\(q^\ell-q^{-\ell}\)\(q^{\ell+1}-q^{-\ell-1}\) }\\
  &\quad\times
   f_0(z)^{\half[2]_q[\ell+1]_q\(q^{\ell+1}+q^{-\ell-1}\)}
    \prod_{a=0}^\ell f_0(2a\hbar)\inv f_0(2(a+1)\hbar)\inv \\
%%%%%%%%%%%%%%%%%%%%%%%%%%%%%%%%%%%%%%
  =&q^{-\partial}E^+(h_1,z)\vac\ot q^{\ell\partial}E^+(h_1,z)\vac,
\end{align*}
where the first equation and the sixth equation follow from Lemma \ref{lem:E+=exp-wh-h}.
\end{proof}

Lemmas \ref{lem:Delta-f+Delta-f-} and \ref{lem:Delta-E+} collectively establish Proposition \ref{prop:A8}, which, when combined with Lemma \ref{lem:A8}, yields Proposition \ref{prop:AQ-sp6-7}.

\section{Proof of Proposition \ref{prop:W}}\label{subsec:pf-prop-W}

In this section, we prove the results of Theorem \ref{prop:W} sequentially.
For a vector space $W$ and $g(z)\in W[[z,z\inv]]$, we also denote the singular part $g(z)^-$ of $g(z)$ by $\Sing_zg(z)$ in this section.
The following result proves \eqref{eq:prop-W-well-defined}.

\begin{lem}\label{lem:W-well-defined}
For $i\in I$ and $\ell\in\C^\times$, we have that
\begin{align*}
  \Sing_zW_i(z)=0.
\end{align*}
\end{lem}

\begin{proof}
Notice that
\begin{align}
  &\Sing_z\exp\(\(\frac{1-e^{z\partial}}{\partial [r\ell]_{q^{\partial}}}h_i\)_{-1}\)Y_{\wh\ell}(x_i^+,z)x_i^-\nonumber\\
  =&\Sing_z\exp\(\(\frac{1-e^{z\partial}}{\partial [r\ell]_{q^{\partial}}}h_i\)_{-1}\)Y_{\wh\ell}(x_i^+,z)^-x_i^-\nonumber\\
  =&-\Sing_z\frac{1}{(q^{r_i}-q^{-r_i})(z+2r\ell\hbar)}\exp\(\(\frac{1-e^{z\partial}}{\partial [r\ell]_{q^{\partial}}}h_i\)_{-1}\)E_\ell(h_i)\nonumber\\
  &+\Sing_z\frac{1}{(q^{r_i}-q^{-r_i})z}\exp\(\(\frac{1-e^{z\partial}}{\partial [r\ell]_{q^{\partial}}}h_i\)_{-1}\)\vac\nonumber\\
  =&-\frac{1}{(q^{r_i}-q^{-r_i})(z+2r\ell\hbar)}\exp\(\(\frac{1-q^{-2r\ell\partial}}{\partial [r\ell]_{q^{\partial}}}h_i\)_{-1}\)E_\ell(h_i)
  +\frac{1}{(q^{r_i}-q^{-r_i})z}\vac\nonumber\\
  =&-\frac{\exp\(\(q^{-r\ell\partial}2\hbar f_0(2\partial\hbar)h_i\)_{-1}\)
    E_\ell(h_i)}
    {(q^{r_i}-q^{-r_i})(z+2r\ell\hbar)}+\frac{1}{(q^{r_i}-q^{-r_i})z}\vac\nonumber\\
  =&-\frac{1}{(q^{r_i}-q^{-r_i})(z+2r\ell\hbar)}
  \(\frac{f_0(2(r_i+r\ell)\hbar)}{f_0(2(r_i-r\ell)\hbar)}\)^\half\label{eq:Sing-W-temp1}
  \exp\(\(q^{-r\ell\partial}2\hbar f_0(2\partial\hbar)h_i\)_{-1}\)\\
  &\quad\times\exp\(\(-q^{-r\ell\partial}2\hbar f_0(2\partial\hbar)h_i\)_{-1}\)\vac
  +\frac{1}{(q^{r_i}-q^{-r_i})z}\vac\nonumber\\
  =&-\(\frac{f_0(2(r_i+r\ell)\hbar)}{f_0(2(r_i-r\ell)\hbar)}\)^\half
  \frac{\vac}{(q^{r_i}-q^{-r_i})(z+2r\ell\hbar)}
  +\frac{1}{(q^{r_i}-q^{-r_i})z}\vac\nonumber.
\end{align}
where the equation \eqref{eq:Sing-W-temp1} follows from the definition of $E_\ell(h_i)$ (see \eqref{eq:def-E-h}).
From the definition of $W_i(z)$, we complete the proof.
\end{proof}

The following result proves \eqref{eq:prop-W-parafermion}.
\begin{lem}\label{lem:W-parafermion}
For $i,j\in I$ and $\ell\in\C^\times$, we have that
\begin{align}
  &Y_{\wh\ell}(h_i,z)^- W_j(z_1)=0.\label{eq:W-parafermion-2}
\end{align}
\end{lem}

\begin{proof}
From \eqref{eq:local-h-2}, we have that
\begin{align*}
  &[Y_{\wh\ell}(h_i,z)^-,Y_{\wh\ell}(x_j^\pm,z_1)]
  =\pm \Sing_{z}Y_{\wh\ell}(x_j^\pm,z_1)[r_ia_{ij}]_{q^{\pd{z}}}q^{r\ell\pd{z}}\pd{z}\log f(z-z_1)\\
  =&\pm \Sing_{z}Y_{\wh\ell}(x_j^\pm,z_1)[r_ia_{ij}]_{q^{\pd{z}}}q^{r\ell\pd{z}}\frac{1+e^{-z+z_1}}{2-2e^{-z+z_1}}
  =\pm Y_{\wh\ell}(x_{j,\hbar}^\pm,z_1)[r_ia_{ij}]_{q^{\pd{z}}}q^{r\ell\pd{z}}\frac{1}{z-z_1}.
\end{align*}
By applying on $\vac$ and taking $z_1\to 0$, we get that
\begin{align*}
  Y_{\wh\ell}(h_i,z)^-x_j^\pm
  =\pm x_j^\pm [r_ia_{ij}]_{q^{\pd{z}}}q^{r\ell\pd{z}}\frac{1}{z}.
\end{align*}
Then
\begin{align}
  &Y_{\wh\ell}(h_i,z)^-Y_{\wh\ell}(x_j^+,z_1)x_j^-
  =[Y_{\wh\ell}(h_i,z)^-,Y_{\wh\ell}(x_j^\pm,z_1)]x_j^-
  +Y_{\wh\ell}(x_j^+,z_1)Y_{\wh\ell}(h_i,z)^-x_j^-\nonumber\\
  &\quad=Y_{\wh\ell}(x_j^+,z_1)x_j^-\ot [r_ia_{ij}]_{q^{\pd{z}}}q^{r\ell\pd{z}}\(\frac{1}{z-z_1}
  -\frac{1}{z}\).\label{eq:W-parafermion-1}
\end{align}
%which completes the proof of \eqref{eq:W-parafermion-1}.

From \eqref{eq:local-h-1}, we have that
\begin{align*}
  &\left[Y_{\wh\ell}(h_i,z)^-,
  \Res_{z_2}z_2\inv\frac{1-e^{z_1\pd{z_2}}}{\pd{z_2}[r\ell]_{q^{\pd{z_2}}}}Y_{\wh\ell}(h_j,z_2)\right]\nonumber\\
  =&\Sing_z[r_ia_{ij}]_{q^{\pd{z}}}q^{r\ell\pd{z}}\(1-e^{-z_1\pd{z}}\)\pd{z}\log f(z)\nonumber\\
  =&\Sing_z[r_ia_{ij}]_{q^{\pd{z}}}q^{r\ell\pd{z}}\(1-e^{-z_1\pd{z}}\)\frac{1+e^{-z}}{2-2e^{-z}}\nonumber\\
  =&[r_ia_{ij}]_{q^{\pd{z}}}q^{r\ell\pd{z}}\(1-e^{-z_1\pd{z}}\)\frac{1}{z}
  =[r_ia_{ij}]_{q^{\pd{z}}}q^{r\ell\pd{z}}\(\frac{1}{z}-\frac{1}{z-z_1}\).
\end{align*}
Then
\begin{align}
  &\left[Y_{\wh\ell}(h_i,z)^-,\exp\(\(\frac{1-e^{z_1\partial}}{\partial[r\ell]_{q^{\partial}}}h_j\)_{-1}\)\right]\nonumber\\
  =&\exp\(\(\frac{1-e^{z_1\partial}}{\partial[r\ell]_{q^{\partial}}}h_j\)_{-1}\)
  \ot [r_ia_{ij}]_{q^{\pd{z}}}q^{r\ell\pd{z}}\(\frac{1}{z}-\frac{1}{z-z_1}\).\label{eq:S-wtW-temp3}
\end{align}
Combining \eqref{eq:W-parafermion-1} and \eqref{eq:S-wtW-temp3}, we get that
\begin{align*}
  &Y_{\wh\ell}(h_i,z)^-\exp\(\(\frac{1-e^{z_1\partial}}{\partial[r\ell]_{q^{\partial}}}h_j\)_{-1}\)Y_{\wh\ell}(x_j^+,z_1)x_j^-
  =0.
\end{align*}
From the definition of $W_j(z)$ and the fact that
\begin{align*}
  Y_{\wh\ell}(h_i,z)^-\vac=0,
\end{align*}
we complete the proof of the \eqref{eq:W-parafermion-2}.
\end{proof}

We need the following technical result to
compute
\begin{align*}
  S_{\ell,\ell'}(z)(W_j(z_1)\ot h_i)\quad\te{and}\quad
  S_{\ell,\ell'}(z)(W_j(z_1)\ot x_i^\pm).
\end{align*}

\begin{lem}
For $i,j\in I$ and $\ell,\ell'\in\C$, we have that
\begin{align}
  &S_{\ell,\ell'}(z)\(Y_{\wh\ell}(x_j^+,z_1)x_j^-\ot h_i\)
  =Y_{\wh\ell}(x_j^+,z_1)x_j^-\ot h_i\label{eq:S-wtW-h}\\
  &\quad+Y_{\wh\ell}(x_j^+,z_1)x_j^-\ot \vac\ot
  \pd{z}\log \(f(z+z_1)/f(z)\)^{[r_ia_{ij}]_{q}[r\ell']_q(q-q\inv)},\nonumber\\
  &S_{\ell,\ell'}(z)\(Y_{\wh\ell}(x_j^+,z_1)x_j^-\ot x_i^\pm\)\label{eq:S-wtW-x}\\
  &\nonumber\quad=Y_{\wh\ell}(x_j^+,z_1)x_j^-\ot x_i^\pm \ot
  \(f(z+z_1)/f(z)\)^{\pm q^{-r_ia_{ij}}\mp q^{r_ia_{ij}}}.
  %\frac{f(z+z_1\mp r_ia_{ij}\hbar)f(z\pm r_ia_{ij}\hbar)}
  %  {f(z+z_1\pm r_ia_{ij}\hbar)f(z\mp r_ia_{ij}\hbar)}.\nonumber
\end{align}
\end{lem}

\begin{proof}
From \eqref{eq:qyb-hex1} and \eqref{eq:S-twisted-2}, we have that
\begin{align*}
  &S_{\ell,\ell'}(z)\(Y_{\wh\ell}(x_j^+,z_1)x_j^-\ot h_i\)
  =Y_{\wh\ell}^{12}(z_1)S_{\ell,\ell'}^{23}(z)S_{\ell,\ell'}^{13}(z+z_1)\(x_j^+\ot x_j^-\ot h_i\)\\
  =&Y_{\wh\ell}^{12}(z_1)\Big(x_j^+\ot x_j^-\ot h_i
    +x_j^+\ot x_j^-\ot \vac \ot
    \(e^{z_1\pd{z}}-1\)\pd{z}\log f(z)^{ [r_ia_{ij}]_{q}[r\ell']_{q}(q-q\inv) }  \Big)\\
  =&Y_{\wh\ell}(x_j^+,z_1)x_j^-\ot h_i
  +Y_{\wh\ell}(x_j^+,z_1)x_j^-\ot \vac \ot \pd{z}\log \(f(z+z_1)/f(z)\)^{ [r_ia_{ij}]_{q}[r\ell']_{q}(q-q\inv) },
\end{align*}
which proves \eqref{eq:S-wtW-h}.
From \eqref{eq:qyb-hex1} and \eqref{eq:S-twisted-4}, we have that
\begin{align*}
  &S_{\ell,\ell'}(z)\(Y_{\wh\ell}(x_j^+,z_1)x_j^-\ot x_i^\pm\)
  =Y_{\wh\ell}^{12}(z_1)S_{\ell,\ell'}^{23}(z)S_{\ell,\ell'}^{13}(z+z_1)\(x_j^+\ot x_j^-\ot x_i^\pm\)\\
  &\quad=Y_{\wh\ell}(x_j^+,z_1)x_j^-\ot x_i^\pm\ot
    \(f(z+z_1)/f(z)\)^{\pm q^{-r_ia_{ij}}\mp q^{r_ia_{ij}}},
\end{align*}
which proves \eqref{eq:S-wtW-x}.
\end{proof}

Employing Lemma \ref{lem:W-S-invariant}, the relations \eqref{eq:qyb-hex1}, \eqref{eq:qyb-hex2}
and the fact that both $V_{\hat\g,\hbar}^{\ell'}$ and $L_{\hat\g,\hbar}^{\ell'}$ are generated by the set
$\set{h_i,\,x_i^\pm}{i\in I}$, we conclude that \eqref{eq:prop-W-S-invariant} follows from the result below.

\begin{lem}\label{lem:W-S-invariant}
For $i,j\in I$ and $\ell,\ell'\in\C$ with $\ell\ne 0$, we have that
\begin{align}
  &S_{\ell,\ell'}(z)(W_j(z_1)\ot h_i)=W_j(z_1)\ot h_i,\label{eq:S-W-h}\\
  &S_{\ell,\ell'}(z)(W_j(z_1)\ot x_i^\pm)=W_j(z_1)\ot x_i^\pm.\label{eq:S-W-x}
\end{align}
\end{lem}

\begin{proof}
From \eqref{eq:S-twisted-1} and \cite[Lemma 4.13]{K-Coproduct-q-aff-va}, we have that
\begin{align*}
  &S_{\ell,\ell'}(z)\(\frac{1-e^{z_1\partial}}{\partial [r\ell]_{q^{\partial}}}h_j\ot h_i\)
  =\frac{1-e^{z_1\partial\ot1+z_1\pd{z}}}{(\partial\ot1+\pd{z}) [r\ell]_{q^{\partial\ot1+\pd{z}}}}S_{\ell,\ell'}(z)(h_j\ot h_i)\\
  =&\frac{1-e^{z_1\partial}}{\partial [r\ell]_{q^{\partial}}}h_j\ot h_i+\vac\ot\vac\ot [r_ia_{ij}]_{q^{\pd{z}}}[r\ell]_{q^{\pd{z}}}[r\ell']_{q^{\pd{z}}}
  \(q^{\pd{z}}-q^{-\pd{z}}\)\frac{1-e^{z_1\pd{z}}}{\pd{z}[r\ell]_{q^{\pd{z}}}}\pdiff{z}{2}\log f(z)\\
  =&\frac{1-e^{z_1\partial}}{\partial [r\ell]_{q^{\partial}}}h_j\ot h_i
  +\vac\ot\vac
  \ot  \(1-e^{z_1\pd{z}}\)\pd{z}\log f(z)^{[r_ia_{ij}]_{q}[r\ell']_q(q-q\inv) }.
\end{align*}
Combining this with \eqref{eq:S-wtW-h} and \cite[Lemma 4.16]{K-Coproduct-q-aff-va}, we get that
\begin{align*}
  &S_{\ell,\ell'}(z)\(W_j(z_1)\ot h_i\)\\
  =&S_{\ell,\ell'}(z)\(\exp\(\(\frac{1-e^{z_1\partial}}{\partial [r\ell]_{q^{\partial}}}h_j\)_{-1}\)Y_{\wh\ell}(x_j^+,z_1)x_j^-\ot h_i\)\\
  &-\frac{\vac\ot h_i}{(q^{r_i}-q^{-r_i})z}
    +\(\frac{f_0(2(r_i+r\ell)\hbar)}{f_0(2(r_i-r\ell)\hbar)}\)^\half\frac{\vac\ot h_i}{(q^{r_i}-q^{-r_i})(z+2r\ell\hbar)}\\
  =&\exp\(\(\frac{1-e^{z_1\partial}}{\partial [r\ell]_{q^{\partial}}}h_j\)_{-1}\)Y_{\wh\ell}(x_j^+,z_1)x_j^-\ot h_i\\
  &+\exp\(\(\frac{1-e^{z_1\partial}}{\partial [r\ell]_{q^{\partial}}}h_j\)_{-1}\)Y_{\wh\ell}(x_j^+,z_1)x_j^-\ot \vac
     \\
  &\quad\ot
    \(e^{z_1\pd{z}}-1+1-e^{z_1\pd{z}}\)\log f(z)^{ [r_ia_{ij}]_{q}[r\ell']_{q}(q-q\inv) }\\
  &-\frac{\vac\ot h_i}{(q^{r_i}-q^{-r_i})z}
    +\(\frac{f_0(2(r_i+r\ell)\hbar)}{f_0(2(r_i-r\ell)\hbar)}\)^\half\frac{\vac\ot h_i}{(q^{r_i}-q^{-r_i})(z+2r\ell\hbar)}\\
  =&W_j(z_1)\ot h_i,
\end{align*}
which proves \eqref{eq:S-W-h}.
From \eqref{eq:S-twisted-3} and \eqref{eq:qyb-der-shift}, we have that
\begin{align*}
  &S_{\ell,\ell'}(z)\(\frac{1-e^{z_1\partial}}{\partial [r\ell]_{q^{\partial}}}h_j\ot x_i^\pm\)
  =\frac{1-e^{z_1\partial\ot1+z_1\pd{z}}}{(\partial\ot1+\pd{z}) [r\ell]_{q^{\partial\ot1+\pd{z}}}}S_{\ell,\ell'}(z)(h_j\ot x_i^\pm)\\
  =&\frac{1-e^{z_1\partial}}{\partial [r\ell]_{q^{\partial}}}h_j\ot x_i^\pm
  \mp\vac\ot x_i^\pm \ot [r_ja_{ji}]_{q^{\pd{z}}}[r\ell]_{q^{\pd{z}}}
  \(q^{\pd{z}}-q^{-\pd{z}}\)\frac{1-e^{z_1\pd{z}}}{\pd{z}[r\ell]_{q^{\pd{z}}}}\pd{z}\log f(z)\\
  =&\frac{1-e^{z_1\partial}}{\partial [r\ell]_{q^{\partial}}}h_j\ot x_i^\pm
  \mp\vac\ot x_{i,\hbar}^\pm \ot \(1-e^{z_1\pd{z}}\)\log f(z)^{q^{r_ja_{ji}}-q^{-r_ja_{ji}}}\\
  =&\frac{1-e^{z_1\partial}}{\partial [r\ell]_{q^{\partial}}}h_j\ot x_i^\pm
  -\vac\ot x_i^\pm \ot \log \(f(z+z_1)/f(z)\)^{ \pm q^{-r_ia_{ij}}\mp q^{r_ia_{ij}} }\\
  =&\frac{1-e^{z_1\partial}}{\partial [r\ell]_{q^{\partial}}}h_j\ot x_i^\pm
  -\vac\ot x_i^\pm \ot \log
  \(f(z+z_1)/f(z)\)^{\pm q^{-r_ia_{ij}}\mp q^{r_ia_{ij}}}.
\end{align*}
Combining this with \eqref{eq:S-wtW-x} and \cite[Lemma 4.16]{K-Coproduct-q-aff-va}, we get that
\begin{align*}
  &S_{\ell,\ell'}(z)\(W_j(z_1)\ot x_i^\pm\)\\
  =&S_{\ell,\ell'}(z)\(\exp\(\(\frac{1-e^{z_1\partial}}{\partial [r\ell]_{q^{\partial}}}h_j\)_{-1}\)Y_{\wh\ell}(x_j^+,z_1)x_j^-\ot x_i^\pm\)\\
  &-\frac{\vac\ot x_i^\pm}{(q^{r_i}-q^{-r_i})z}
    +\(\frac{f_0(2(r_i+r\ell)\hbar)}{f_0(2(r_i-r\ell)\hbar)}\)^\half\frac{\vac\ot x_i^\pm}{(q^{r_i}-q^{-r_i})(z+2r\ell\hbar)}\\
  =&\exp\(\(\frac{1-e^{z_1\partial}}{\partial [r\ell]_{q^{\partial}}}h_j\)_{-1}\)Y_{\wh\ell}(x_j^+,z_1)x_j^-\ot x_i^\pm\\
  &\quad\ot
  \(f(z+z_1)/f(z)\)^{\pm q^{-r_ia_{ij}}\mp q^{r_ia_{ij}}}
  \(f(z+z_1)/f(z)\)^{\mp q^{-r_ia_{ij}}\pm q^{r_ia_{ij}}}\\
  &-\frac{\vac\ot x_i^\pm}{(q^{r_i}-q^{-r_i})z}
    +\(\frac{f_0(2(r_i+r\ell)\hbar)}{f_0(2(r_i-r\ell)\hbar)}\)^\half\frac{\vac\ot x_i^\pm}{(q^{r_i}-q^{-r_i})(z+2r\ell\hbar)}\\
  =&W_j(z_1)\ot x_i^\pm,
\end{align*}
which proves \eqref{eq:S-W-x}.
\end{proof}

\section*{Acknowledgment}
The author would like to thank the referee for the valuable comments that greatly improved the exposition of the paper.

Part of this paper was finished during my visit at Xiamen University and Tianyuan Mathematical Center in Southeast China, in
August 2023. I am very grateful to Professor Shaobin Tan, Fulin Chen, Qing Wang for their hospitality.

\section*{Declarations}
{\bf Funding:} This study was funded by National Natural Science Foundation of China grant no. 12371027.

{\bf Conflict of interest:} The author declares that he has no conflict of interest.

\bibliographystyle{unsrt}
%\bibliographystyle{alpha}
%\bibliographystyle{alpha}

%\bibliography{../reference}


\end{document}